\documentclass{article}

\usepackage{arxiv}

\usepackage[utf8]{inputenc} 
\usepackage[T1]{fontenc}    
\usepackage{url}            
\usepackage{booktabs}       
\usepackage{amsfonts}       
\usepackage{nicefrac}       
\usepackage{microtype}      
\usepackage{lipsum}
\usepackage{graphicx}

\usepackage[round]{natbib}

\usepackage{comment,multirow,stmaryrd,subfigure,empheq,caption,stackrel,mathabx,mathtools,algorithm2e,amsthm,rotating,charter}
\allowdisplaybreaks

\def\ie{{\itshape i.e.},\/}
\def\adhoc{{\itshape{ad hoc}}\/}
\def\resp{\textit{resp. }}
\def\MINLP{(A)}

\def\rec{\textnormal{rs}}
\def\strip{\textnormal{ss}}
\def\feed{\textnormal{in}}

\def\upbnd{\textnormal{up}}
\def\lobnd{\textnormal{lo}}
\def\bin{\textnormal{bin}}
\def\polydeg{\textnormal{deg}}

\def\FC{\mathit{FC}}
\def\FR{\mathit{FR}}
\def\vapduty{\mathit{VD}}

\newtheorem{corollary}{Corollary}
\newtheorem{proposition}{Proposition}
\newtheorem{remark}{Remark}
\newtheorem{lemma}{Lemma}
\newtheorem{definition}{Definition}
\newtheorem{example}{Example}
\DeclareMathOperator{\Conv}{Conv}
\DeclareMathOperator{\Relax}{Relax}
\DeclareMathOperator{\proj}{proj}

\newcounter{eqnind}
\newcounter{eqnalind}

\newenvironment{eqn}
{\stepcounter{eqnind}%
	\addtocounter{equation}{-1}%
	\renewcommand\theequation{A\arabic{eqnind}}\equation}
{\endequation}

\newenvironment{eqnal}
{\setcounter{eqnalind}{\theequation}
	\setcounter{equation}{\theeqnind}%
	\renewcommand\theequation{A\arabic{equation}}
	\align}
{\endgather\setcounter{eqnind}{\value{equation}}\setcounter{equation}{\theeqnalind}}

\title{Advances in MINLP to Identify Energy-efficient Distillation Configurations}

\author{
 Radhakrishna Tumbalam Gooty, Rakesh Agrawal \\
  Davidson School of Chemical Engineering\\
  Purdue University\\
  West Lafayette, IN 47907 \\
  \texttt{rtumbala@purdue.edu, agrawalr@purdue.edu} 
  \And
 Mohit Tawarmalani \\
  Krannert School of Management\\
  Purdue University\\
  West Lafayette, IN 47907 \\
  \texttt{mtawarma@purdue.edu} \\
}

\begin{document}
\maketitle
\begin{abstract}
In this paper, we describe the first mixed-integer nonlinear programming (MINLP) based solution approach that successfully identifies the most energy-efficient distillation configuration sequence for a given separation. Current sequence design strategies are largely heuristic. The rigorous approach presented here can help reduce the significant energy consumption and consequent greenhouse gas emissions by separation processes, where crude distillation alone is estimated to consume 6.9 quads of energy per year globally \citep{sholl2016seven}. The challenge in solving this problem arises from the large number of feasible configuration sequences and because the governing equations contain non-convex fractional terms. We make several advances to enable solution of these problems. First, we model discrete choices using a formulation that is provably tighter than previous formulations. Second, we highlight the use of partial fraction decomposition alongside Reformulation-Linearization Technique (RLT). Third, we obtain convex hull results for various special structures. Fourth, we develop new ways to discretize the MINLP. Finally, we provide computational evidence to demonstrate that our approach significantly outperforms the state-of-the-art techniques.
\end{abstract}


\section{Introduction}
\label{sec:Introduction}
Separation of mixtures of chemical components is ubiquitous in all chemical and petrochemical industries. Among the numerous technologies available for separation of multicomponent mixtures (three or more components) into almost pure components, distillation is the predominant choice. A few well-known applications include fractionation of crude oil \citep{sholl2016seven}, production of ultra pure nitrogen and oxygen from air \citep{agrawal1991}, Natural Gas Liquid (NGL) recovery, Benzene-Toluene-Xylene (BTX) separation, etc. It is estimated that distillation accounts for 90 -- 95\% of the liquid phase separations in the US \citep{humphrey}. With the increased potential to harness shale reserves \citep{siirola,ridha}, the use of distillation is projected to increase further. Industrial distillations are energy intensive, and the energy consumed constitutes about 40 -- 60\% of the total operating cost \citep{humphrey}. Since energy consumed affects the effective fuel consumption, suboptimal configurations also tend to release more CO\textsubscript{2}. In this article, we develop the first tractable approach to solve distillation sequencing via Mixed-integer Nonlinear Programming (MINLP) techniques.

Mixtures are separated in a \textit{distillation configuration} consisting of a series of distillation columns/towers (see Figure \ref{fig:conDC}) arranged to carry out the separation in a specific order (see Figure \ref{fig:exconfiguration} for an example). The number of admissible configurations grows rapidly with the number of components in the given mixture. For example, over half a million alternative configurations are admissible for separating a six component mixture \citep{shah2010}. Besides the number of choices, the nonconvex fractional terms used to model the minimum energy requirement make this problem hard to solve. Prior to this work, the state-of-the-art methods are unable optimize the material flows even for a specific configuration. As such, conventional design practices are based on heuristics and intuition of the process engineer, and they often result in suboptimal solutions.  

To address these challenges, we develop a new mixed-integer nonlinear programming (MINLP) based approach and make several advances. First, we develop a new model for the space of admissible configurations, by incorporating convex hulls of various substructures. The resulting formulation is provably tighter than those in the literature \citep{caballero2006,giridhar2010b,tumbalamgooty2019}. Second, we show that polynomial division and partial fraction decomposition can significantly improve the quality of relaxations developed using the classical Reformulation-Linearization Technique (RLT) \citep{sherali1992new} when the constraints involve fractional terms. Third, we develop simultaneous convex hulls of multiple nonlinear terms over a polytope obtained by intersecting bounds on variables with material balance equations. Fourth, we derive the first rigorous relaxation for distillation sequencing. The governing equations for this problem involve fractions, whose denominator can approach zero. To sidestep this issue, the literature has imposed an \adhoc{} bound on the denominator. This approach, however, can prune optimal solutions when some component flows are small. Instead, using the cuts from the RLT variant described above, our approach derives a rigorous relaxation for the problem. 

In \textsection \ref{sec:Preliminaries}, we briefly describe the key concepts of multicomponent distillation, and survey the current literature. \textsection \ref{sec:Prob-Definition} defines the problem statement and introduces relevant notation. We formulate the MINLP in \textsection \ref{sec:Prob-Formulation}, and outline the overall relaxation and solution procedure in \textsection \ref{sec:Relaxation}. 
We report on computational experiments in \textsection \ref{sec:Comp-Results}. Finally, we make concluding remarks in \textsection \ref{sec:Conclusion}.

\begin{figure}
	\centering
	\includegraphics[scale=1]{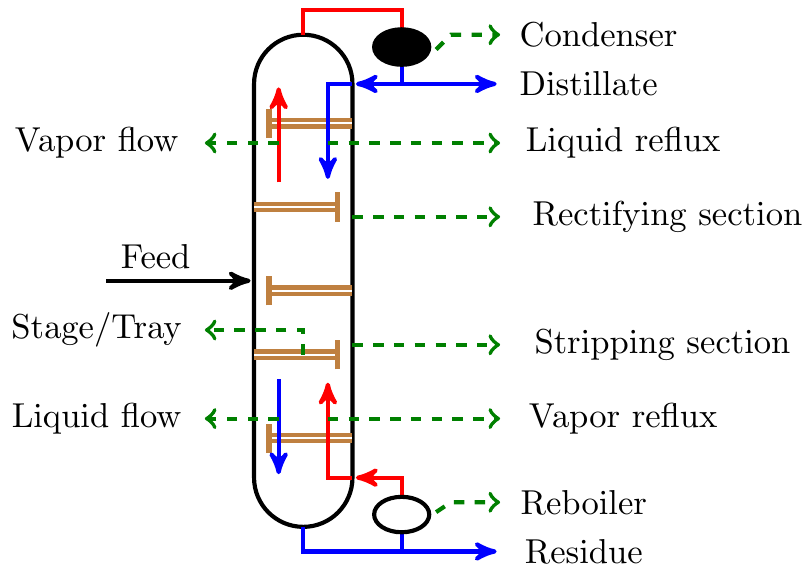}
	\caption{Schematic of a conventional distillation column. Double-lines in brown, known as \textit{trays/stages}, establish contact between the vapor (red arrows) and liquid (blue arrows) for mass transfer. Here, and in the rest of the article, condensers and reboilers are denoted by filled and open ellipses (or circles), respectively.}
	\label{fig:conDC}
\end{figure}

\section{The Distillation Process}
\label{sec:Preliminaries}
Distillation is a way to separate mixtures, consisting of two or more components with different relative volatilities, by boiling the mixture so that the vapor produced is rich in more volatile (or \textit{light}) components, while the residual liquid is enriched in less volatile (or \textit{heavy}) components. Industrial distillation is carried out in a staged-tower/column (see Figure \ref{fig:conDC}), where each stage establishes liquid-vapor contact for mass transfer. The \textit{feed} (mixture of components) is introduced at an intermediate location of the column. The \textit{sections} above and below the feed stream are known as \textit{rectifying} and \textit{stripping sections}, respectively. Conventional columns have a condenser (\resp{} reboiler) at the top (\resp{} bottom) which condenses (\resp{} vaporizes) the vapor (\resp{} liquid), and feeds a portion of it back to the column, known as \textit{liquid} (\resp{} \textit{vapor}) \textit{reflux}. The liquid flowing from the top to bottom strips away heavy components from the vapor, while the vapor flowing from bottom to top gets enriched with lighter components. The net outflow from the rectifying and stripping sections, respectively, are known as \textit{distillate} and \textit{residue}. In short, distillation enriches the distillate with light components, and the residue with heavy components.
\begin{remark}
	\label{rem:recovery}
	The recovery of a lighter component in distillate (ratio of component flowrate in distillate to flowrate in feed) is higher than the recovery of a heavier component, and the converse is true for residue \citep{nallasivam2016,mathewtony}. \qed
\end{remark}

A given $N-$component mixture, referred to as the \textit{process feed}, is separated into $N$ constituent components using a sequence of distillation columns (see Figure \ref{fig:exconfiguration}, for example). Let $C_i\ldots C_j$ denote an intermediate stream (referred to as \textit{submixture}), where components are sorted in a decreasing order of relative volatilities, and $C_p$ denotes the $p^{\text{th}}$ component in the process feed. Each column \textit{splits} a feed submixture into two product submixtures, each of which has at least one component less than the feed. The composition of the product submixtures governs the threshold vapor requirement for the column. This requirement can be determined using the classical Underwood method \citep{underwood}, as long as relative volatilities are constant, each section has infinite stages, and there is constant molar overflow. As shown in Figure \ref{fig:exconfiguration}, condensers and reboilers can be replaced with two-way vapor-liquid transfer streams known as \textit{thermal couplings}, so that the required liquid/vapor reflux is borrowed from other columns. Since thermal couplings allow vapor to be transferred between two or more columns, a column may be operated above its threshold vapor requirement to supply the vapor to another column. 
We remark that configurations with many thermal couplings may be hard to control \citep{agrawal2000thermally} and require hot and/or cold utilities at extreme temperatures. Although we do not explicitly model these issues, configurations with few thermal couplings can easily be found by simple changes to our formulation.

\begin{figure}
	\centering
	\includegraphics[scale=0.8]{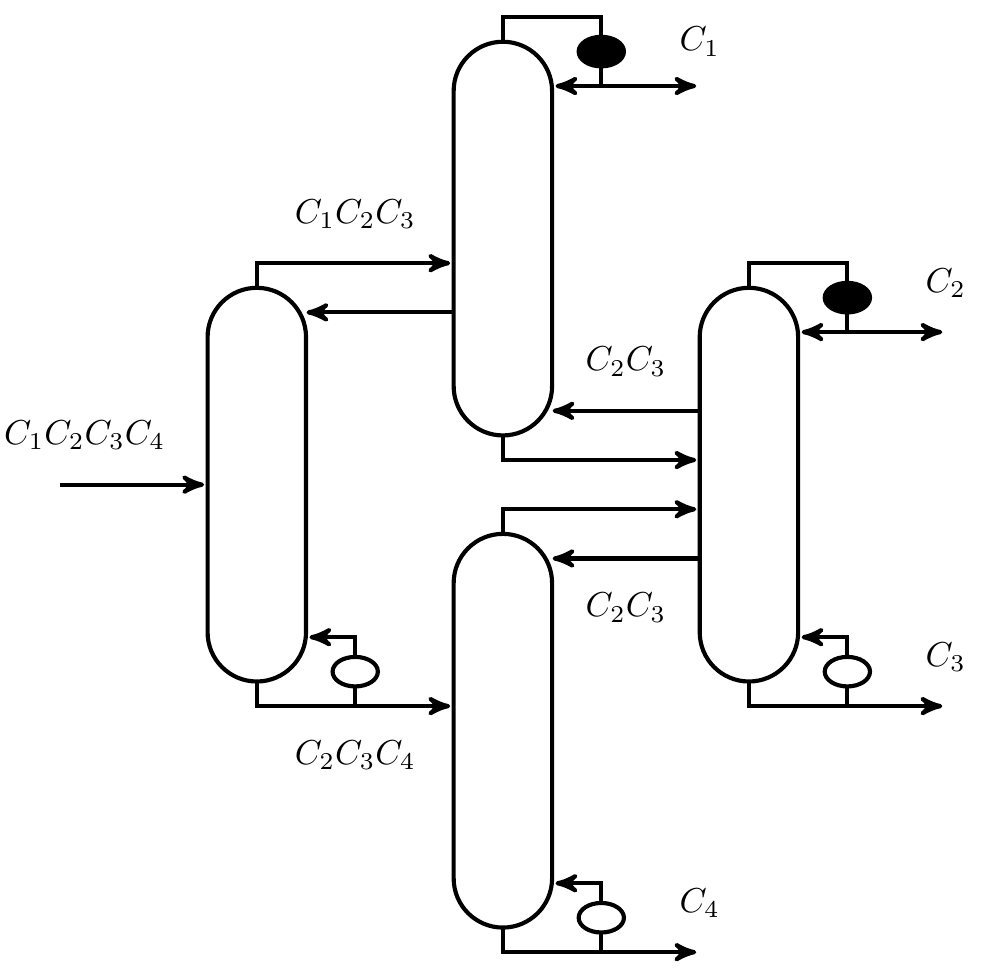}
	\caption{Example of a distillation configuration. $C_i\dots C_j$ denotes a mixture, and each $C_p$ corresponds to a distinct chemical component. $C_1\dots C_4$ is the process feed, and intermediate mixtures $C_1C_2C_3$, $C_2C_3C_4$ and $C_2C_3$ are referred as submixtures.}
	\label{fig:exconfiguration}
\end{figure}

For above ambient distillation, the required vapor flow is generated at reboilers by a hot utility. By adding these vapors, we obtain the \textit{vapor duty} of the configuration, which is often used as a proxy for its energy consumption and operating cost. The vapor duty indirectly affects the capital cost as well, since internal vapor flows dictate column diameters. For these reasons, we will minimize vapor duty, an objective that has also been used in previous studies \citep{fidkowski,fidkowski4compftc,nallasivam2016}. Industrial practitioners may instead be interested in minimizing the total annualized cost (capital plus operating costs), or maximizing the thermodynamic efficiency. The model we propose can be tailored to the desired objective by appending the relevant constraints and modifying the objective as in \citet{jiang2019exergy,jiang2019cost}.

Given its importance, this problem has been studied extensively, but has resisted formal solution guarantees. \citet{caballero2004,caballero2006} formulated an MINLP to identify configurations with lowest total annualized cost, but did not certify global optimality. Given the non-convexity, these local approaches do not always find optimal designs \citep{nallasivam2013,jiang2019cost}. \citet{giridhar2010b} proposed an alternate MINLP formulation to minimize vapor duty, and solved it using BARON \citep{tawarmalani2005} for three and four-component mixtures. However, their methodology does not scale to five component mixtures. \citet{nallasivam2016} enumerated all configurations, and solved a nonlinear program for each using BARON. However, some configurations fail to converge. The current state-of-the-art formulation of \citet{tumbalamgooty2019} still fails to converge on 36\% of the MINLP instances.  

\section{Problem Definition}
\label{sec:Prob-Definition}
Figure \ref{fig:superstructure} shows all possible streams and heat exchangers in a distillation configuration that separates a four-component mixture into pure components. We represent streams as squares, condensers as filled circles and reboilers as open circles. Each condenser/reboiler is associated with a process stream, that is not the process feed $C_1\dots C_N$. Throughout the formulation, we denote a stream $C_i\dots C_j$ as $[i,j]$, and heat exchangers as $(i,j)$, so that condenser $(i,j)$ (\resp reboiler $(i,j)$) represents the heat exchanger through which $[i,j]$ is withdrawn as distillate (\resp residue). By Remark \ref{rem:recovery}, a configuration cannot contain streams of the form $C_i\dots C_kC_{k+l}\dots C_j$, where $l>1$. 

\begin{figure}[h]
	\centering
	\includegraphics[scale=1]{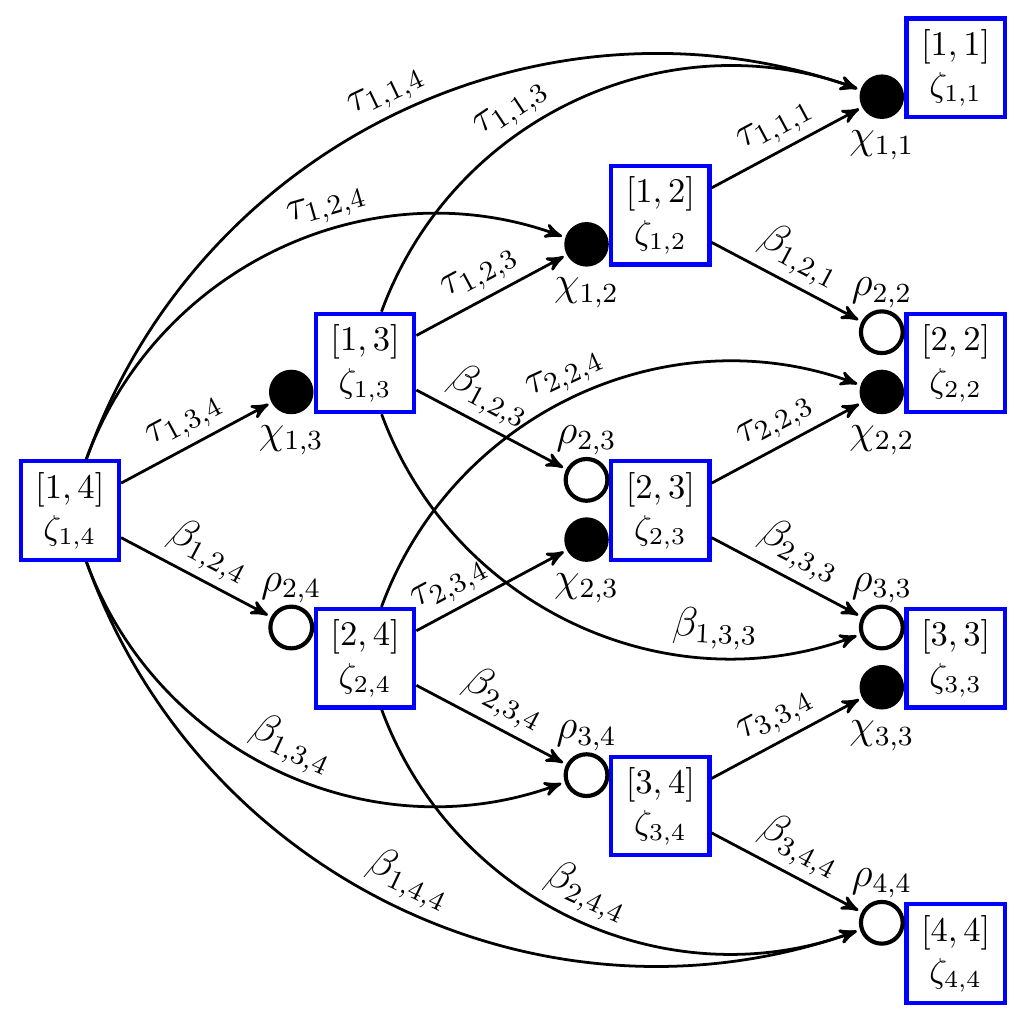}
	\caption{Figure depicting streams ($\zeta_{i,j}$), reboilers ($\rho_{i,j}$) and condensers ($\chi_{i,j}$) present in a four-component system. Section variables $\tau_{i,k,j}$ and $\beta_{i,l,j}$ are defined in \eqref{eq:secvardef}. }
	\label{fig:superstructure}
\end{figure}

We denote the set of streams as $\mathcal{T}$, the set of condensers as $\mathcal{C}$, and the set of reboilers as $\mathcal{R}$ (see Table \ref{tab:setdef} for definition). For convenience, we create a set containing streams that are mixtures $\mathcal{P}=\mathcal{T} \setminus\{[1,1],\dots,[N,N]\} $, and a set containing submixtures $\mathcal{S} = \mathcal{P}\setminus \{[1,N]\}$. Note that every stream in $\mathcal{P}$ is a mixture, and must undergo a split in order to produce products.

\begin{table}[ht]
	\centering
	\renewcommand\arraystretch{1.5}
	\begin{tabular}{lcl}
		\toprule
		\textbf{Set}		&	\textbf{Symbol} &	\textbf{Definition}\\
		\midrule
		Streams &	$\mathcal{T}$	&	$\{[i,j]:\;1\leq i\leq j\leq N\}$\\
		Splits	&	$\mathcal{P}$	&	$\mathcal{T}\setminus\{[i,i]\}_{i=1}^{N}$\\
		Submixtures & $\mathcal{S}$	&	$\mathcal{P}\setminus\{[1,N]\}$\\
		Condensers & $\mathcal{C}$		&	$\{(i,j):\;1\leq i\leq j \leq N-1\}$\\
		Reboilers & $\mathcal{R} $		&	$\{(i,j):\;2\leq i\leq j\leq N\}$\\
		\bottomrule
	\end{tabular}
	\caption{Definition of sets.}
	\label{tab:setdef}
\end{table}

The required input to the problem consists of (1) composition of the process feed $\{F_p\}_{p=1}^{N}$ either in terms of mole fractions or molar flowrates of the components in the stream, (2) relative volatilities $\{\alpha_p\}_{p=1}^{N}$ (such that $\alpha_{N} < \dots < \alpha_1$) of its constituent components; and (3) liquid fraction (fraction of the total flow in liquid phase) of the process feed $\Phi_{1,N}$ and that of the pure components $\{\Phi_{i,i}\}_{i=1}^N$. We write $\{p\}_{p=1}^{N}$ or $\{p\}_{1 \le p \le n}$ to denote the set $\{1,2,\dots,N\}$, and $\llbracket p\rrbracket_{1}^N $ to denote $\forall \; p\in \{1,\dots,N\}$.
Given a process feed, the problem is then to identify the best distillation configuration, along with its optimal operating conditions, that requires least vapor duty.

\section{Problem Formulation}
\label{sec:Prob-Formulation}
We formulate the MINLP in this section. Before proceeding further, we introduce the definition of \textit{parents} and \textit{children} of a stream. By \textit{top} (\resp \textit{bottom}) \textit{parents} of $[i,j]$: we refer to streams $\{[i,n]\}_{n=j+1}^{N}$ (\resp $\{[m,j]\}_{m=1}^{i-1}$) which can produce $[i,j]$ as distillate (\resp residue). Analogously, by \textit{top} (\resp \textit{bottom}) \textit{children} of $[i,j]$, we refer to streams $\{[i,k]\}_{k=i}^{j-1}$ (\resp $\{[l,j]\}_{l=i+1}^j$) which can be produced as distillate (\resp residue) from $[i,j]$. For conciseness, we write $[i,j]\uparrow[i,k]$ (\resp $[i,j]\downarrow[l,j]$) to denote stream $[i,k]$ (\resp $[l,j]$) is produced as the distillate (\resp residue) from $[i,j]$, and $[i,k]/[l,j]$ to denote $[i,k]$ and $[l,j]$ are produced as the distillate and residue from $[i,j]$.

\subsection{Objective Function}
The objective is to determine the configuration(s) which minimizes the total vapor duty:
\begin{eqn}
	\text{\MINLP{}: Minimize  } \sum_{(i,j)\in \mathcal{R}} \FR_{i,j},
	\label{eq:objectivefun}
\end{eqn}
where $\FR_{i,j}$ is the vapor flow generated in reboiler $(i,j)$. The MINLP we develop will be denoted as MINLP \MINLP{}, and the constraints will be numbered as (A\#).

\subsection{Space of Admissible Configurations}\label{sec:space_of_admissible_configurations}
We define \textit{column/stream} binary variables so that $\forall\;[i,j]\in \mathcal{T}$, $\zeta_{i,j} = 1$ if $[i,j]$ is present and 0 otherwise. Further, we define binary variables associated with the presence/absence of condensers and reboilers so that $\forall\;(i,j)\in \mathcal{C}$ (\resp{} $\forall\;(i,j)\in \mathcal{R}$), $\chi_{i,j} = 1$ (\resp{} $\rho_{i,j}=1$) if condenser (\resp{} reboiler) $(i,j)$ is present and 0 otherwise (See Table \ref{tab:setdef} for set definitions). Although these variables suffice \citep{tumbalamgooty2019}, we introduce auxiliary variables to derive a tighter representation. 

For every $[i,j]\in \mathcal{P}$, we define \textit{section} variables $\{\tau_{i,k,j}\}_{k=i}^{j-1}$ and $\{\beta_{i,l,j}\}_{l=i+1}^j$, such that $\tau_{i,k,j}=\{1,\text{ if }[i,j]\uparrow [i,k];\,0,\text{ otherwise}\}$ and $\beta_{i,l,j}=\{1,\text{ if }[i,j]\downarrow [l,j];\,0,\text{ otherwise}\}$. In other words, section variables model distillate and residue streams from a mixture. Figure \ref{fig:superstructure} shows all the section variables for a four-component mixture. We now relate column and section variables. Consider the split of stream $[i,j]$. In configurations of interest, known as \textit{regular-column configurations}, if $[i,j]\uparrow[i,k]$, for any $i \le k \le j-1$, then $[i,j]$ and $[i,k]$ must be present and $\{[i,n]\}_{n=k+1}^{j-1}$ must be absent \citep{caballero2006,giridhar2010b}. Analogously, if $[i,j]\downarrow[l,j]$, for any $i+1 \le l \le j$, $[i,j]$ and $[l,j]$ must be present, while $\{[m,j]\}_{m=i+1}^{l-1}$ must be absent. Therefore, section variables are defined as
\begin{equation}
	\allowdisplaybreaks
	\begin{split}
		\tau_{i,k,j}  & = \zeta_{i,j}(1-\zeta_{i,j-1})\dots (1-\zeta_{i,k+1})\zeta_{i,k}\\
		& = \prod_{n=k+1}^{j-1}(1-\zeta_{i,n})-\prod_{n=k}^{j-1}(1-\zeta_{i,n})-\prod_{n=k+1}^{j}(1-\zeta_{i,n})+\prod_{n=k}^{j}(1-\zeta_{i,n}), \\
		\beta_{i,l,j} & = \zeta_{i,j}(1-\zeta_{i+1,j})\dots (1-\zeta_{l-1,j})\zeta_{l,j}\\
		& = \prod_{m=i+1}^{l-1}(1-\zeta_{m,j})-\prod_{m=i}^{l-1}(1-\zeta_{m,j})-\prod_{m=i+1}^{l}(1-\zeta_{m,j})+\prod_{m=i}^{l}(1-\zeta_{m,j}).
	\end{split}
	\label{eq:secvardef}
\end{equation}
We introduce variables $\{\nu_{i,k,j}:\;1\leq i\leq k\leq j\leq N\}$ and $\{\omega_{i,l,j}:\;1\leq i\leq l\leq j\leq N\}$ to linearize \eqref{eq:secvardef}:
\begin{eqnal}
	\label{eq:secvar}
	\text{for }\; [i,j] \in \mathcal{P}\quad \left\{\begin{aligned}
		\tau_{i,k,j}  & = \nu_{i,k+1,j-1}-\nu_{i,k,j-1}-\nu_{i,k+1,j}+\nu_{i,k,j},\quad \llbracket k\rrbracket_i^{j-1}\\
		\beta_{i,l,j} & = \omega_{i+1,l-1,j}-\omega_{i,l-1,j}-\omega_{i+1,l,j}+\omega_{i,l,j},\quad \llbracket l\rrbracket_{i+1}^{j},
	\end{aligned}\right.
\end{eqnal}
where $\nu_{i,k,j} = \prod_{n=k}^j (1-\zeta_{i,n})$ and  $\omega_{i,l,j}=\prod_{m=i}^l (1-\zeta_{m,j})$. Note that $\nu_{i,k+1,j-1}$ (\resp $\omega_{i+1,l-1,j}$) are defined as one if $k+1=j$ (\resp $i+1=l$). Clearly, $\nu_{i,k,j}=\omega_{i,l,j}=1-\zeta_{i,j}$ if $k=j$ and $l=i$. Besides this relationship, the introduced variables $\nu_{i,k,j}$ and $\omega_{i,l,j}$ are linearly independent.
To see this, note that $\prod_{j\in J} x_j$, where $J \subseteq \{1,\dots,n\}$ are linearly independent and, therefore, so are $\prod_{j\in J} (1-y_j)$, where $y_j =1-x_j$. Since $\nu_{i,k,j}$ and $\omega_{i,l,j}$ are of the latter form, they are linear independent.

We now relax $\nu_{i,k,j}$ and $\omega_{i,l,j}$ variables for $k\ne j$ and $l \ne i$ as follows. Since $\zeta_{i,j}$ is binary, $(1-\zeta_{i,j})^2=(1-\zeta_{i,j})$. We use the definition of $\nu_{i,k,j}$ and $\omega_{i,l,j}$, to derive the following:
\begin{gather}
	\text{for }\; [i,j]\in \mathcal{P} \quad \left\{\begin{gathered}
		\nu_{i,k,j} = \nu_{i,k,m}\nu_{i,n,j},\quad \llbracket n \rrbracket_{k+1}^{m+1},\; \llbracket m\rrbracket_{k}^{j-1}, \; \llbracket k\rrbracket_i^{j-1}\\
		\omega_{i,l,j} = \omega_{i,m,j}\omega_{n,l,j},\quad \llbracket n \rrbracket_{i+1}^{m+1},\; \llbracket m\rrbracket_{i}^{l-1}, \; \llbracket l\rrbracket_{i+1}^{j}.
	\end{gathered}\right.
	\label{eq:uwprod}
\end{gather}
In the above, for $n \le m+1$, $\nu_{i,n,m}$ (\resp $\omega_{n,m,j}$) is a common factor for both $\nu_{i,k,m}$ and $\nu_{i,n,j}$ (\resp $\omega_{i,m,j}$ and $\omega_{n,l,j}$). we regard $\nu_{i,n,m}$ and $\omega_{n,m,j}$ as one if $n= m+1$. Thus, $0 \le \nu_{i,k,m} \le \nu_{i,n,m}$, $0 \le \nu_{i,n,j} \le \nu_{i,n,m}$, $0 \le \omega_{i,m,j} \le \omega_{n,m,j}$, and $0 \le \omega_{n,l,j} \le \omega_{n,m,j}$. Using these bounds, we relax \eqref{eq:uwprod} as:
\begin{eqnal}
	\text{for }\; [i,j]\in \mathcal{P} \; \left\{
	\begin{aligned}
		& \nu_{i,j,j} = \omega_{i,i,j} = 1-\zeta_{i,j}\\
		& \max\{0,\nu_{i,k,m}+\nu_{i,n,j}-\nu_{i,n,m}\} \leq \nu_{i,k,j}\leq \min\{\nu_{i,k,m},\nu_{i,n,j}\},\quad \llbracket n \rrbracket_{k+1}^{m+1},\; \llbracket m\rrbracket_{k}^{j-1}, \; \llbracket k\rrbracket_i^{j-1} \\
		& \max\{0,\omega_{i,m,j}+\omega_{n,l,j}-\omega_{n,m,j}\}\leq\omega_{i,l,j}\leq \min\{\omega_{i,m,j},\omega_{n,l,j}\},\quad \llbracket n \rrbracket_{i+1}^{m+1},\; \llbracket m\rrbracket_{i}^{l-1}, \; \llbracket l\rrbracket_{i+1}^{j},
	\end{aligned}\right.
	\label{eq:mccor}
\end{eqnal}
where we used $\nu_{i,k,m} = \nu_{i,k,m}\nu_{i,n,m}$, $\nu_{i,n,j} = \nu_{i,n,j}\nu_{i,n,m}$, $\omega_{i,m,j}=\omega_{i,m,j}\omega_{n,m,j}$, and $\omega_{n,l,j}=\omega_{n,l,j}\omega_{n,m,j}$.

\begin{proposition}
	Let $S = \{(x,z) \in [0,1]^{2n}\; | \; z_j = \prod_{k=1}^{j} x_k, \; \llbracket j\rrbracket_1^{n}\}$. The convex hull of $S$, $\Conv(S)$, is the intersection of convex hulls of $z_j = z_{j-1} \cdot x_j$, $\llbracket j\rrbracket_2^{n}$ over $[0,1]^2$ (McCormick relaxation). 
	\label{prop:bin-simultaneous-hull}
\end{proposition}
\begin{proof}
	See \textsection \ref{sec:ecom:acyclic-proof} in the appendix.
\end{proof}

We remark that the result in Proposition \ref{prop:bin-simultaneous-hull} also follows from Theorem 10 in \citet{del2018multilinear}. Our proof is, however, different and elementary. We mention that this proof shows a previously unobserved connection to the recursive McCormick procedure. Our proof can be used to show that the recursive McCormick procedure, with a few additional linearization variables, yields the convex hull of the multilinear polytopes for $\gamma$-acyclic hypergraphs, as obtained in \citet{del2018multilinear}.

\begin{remark}
	\label{rem:linearization}
	Proposition \ref{prop:bin-simultaneous-hull} shows that the set of $\nu$ (\resp $\omega$) variables satisfying \eqref{eq:mccor} belong to the intersection of simultaneous convex hulls of $(\nu_{i,j,j+1},\dots,\nu_{i,j,N},\nu_{i,j,j},\dots,\nu_{i,N,N})$ for all $[i,j] \in \mathcal{T}\setminus\{[k,N]\}_{k=1}^N $ (\resp $(\omega_{1,2,j},\dots,\omega_{1,i,j},\omega_{1,1,j},\dots,\omega_{i,i,j})$ for all $ [i,j] \in \mathcal{T}\setminus\{[1,l]\}_{l=1}^N $).\qed
\end{remark}

\begin{remark}
	\label{rem:linearization-tb}
	For every $[i,j] \in \mathcal{P}$, $\llbracket k \rrbracket_{i}^{j-1}$ (\resp $\llbracket l \rrbracket_{i+1}^j$), the convex hull of $\tau_{i,k,j}$ (\resp $\beta_{i,l,j}$) over $(\zeta_{i,k},\dots, \zeta_{i,j}) \in [0,1]^{j-k+1}$ (\resp $(\zeta_{i,j},\dots,\zeta_{l,j}) \in [0,1]^{l-i+1}$) is implied by \eqref{eq:secvar} and \eqref{eq:mccor}. (see \textsection \ref{sec:ecom:tau-hull-proof} for the proof). \qed
\end{remark}

We now describe the constraints to model the space of admissible distillation configurations.

\subsubsection{Presence of process feed and products}
Every admissible configuration has the process feed ($[1,N]$) and the pure components ($\{[i,i]\}_{i=1}^{N}$), \ie{}
\begin{eqn}
	\zeta_{1,N}=\zeta_{1,1}= \dots \zeta_{N,N}=1.
	\label{eq:feedprod}
\end{eqn}
To restrict the search to a subset of configurations, for example, in order to retrofit an existing design, we may 
explicitly include (\resp eliminate) a specific submixture $[i,j]$ by setting $\zeta_{i,j} = 1$ (\resp $\zeta_{i,j} = 0$). We show next that $\zeta_{i,j}$ variables are affinely related to $\tau_{i,k,j}$ and $\beta_{i,l,j}$ variables.

\begin{proposition}
	\label{prop:invertible-transform}
	Let, $x \in [0,1]^n$, $y_{i,j} = (1-x_i) x_{i+1}\dots x_{j-1} (1-x_j)$ for $1 \le i < j \le n$, $z_{i,j} = \prod_{r = i}^j x_r$ for $1 \le i \le j \le n$, and $x_n=0$, which in turn implies that $z_{i,n}=0$ for $1\le i \le n$. Then, there is an invertible affine transformation between $\{y_{i,j}\}_{1\le i < j \le n}$ and $\{z_{i,j}\}_{1\le i \le j \le n}$, given by
	\begin{align*}
	& y_{i,j} = z_{i+1,j-1} -z_{i+1,j} - z_{i,j-1} + z_{i,j}, \\
	& z_{p,q} = 1 - \sum_{r=p}^q \sum_{s=q+1}^n y_{r,s}.
	\end{align*}
\end{proposition}

\begin{proof}
First, we show that $y_{i,j}$ can be written as an affine transformation of $z_{i,j}$. By definition, $y_{i,j} = (1-x_i) x_{i+1}\dots x_{j-1} (1-x_j) = \prod_{r = i+1}^{j-1} x_r - \prod_{r = i+1}^jx_r - \prod_{r = i}^{j-1} x_r + \prod_{r = i}^j x_r = \prod_{r = i+1}^{j-1} x_r -z_{i+1,j} - z_{i,j-1} + z_{i,j}$. Substituting the first term in the last equality, $\prod_{r = i+1}^{j-1} x_r$, with 1 if $i+1 = j$, and $z_{i+1,j-1}$ if $i+1 <j$, yields the required affine transformation.
	
Next, to obtain the inverse affine transformation, we define $w_{k,l} = (1-x_k)x_{k+1}\dots x_l$ for $1 \le k \le l \le n$. We show the affine transformation between $\{w_{k,l}\}_{1 \le k \le l \le n}$ and $\{y_{i,j}\}_{1\le i < j \le n}$ variables to be
\begin{equation}
w_{k,l} = \sum_{r=l+1}^{n} y_{k,r},
\label{eq:w-y-affine}
\end{equation}
using induction on $n-l$. For $l=n$, \eqref{eq:w-y-affine} is trivially satisfied because $w_{k,n}=0$ as $x_n=0$. Now, assuming that \eqref{eq:w-y-affine} holds for $l+1$, \ie{} $w_{k,l+1} = \sum_{r=l+2}^{n} y_{k,r}$, we show that it holds for $w_{k,l}$ as well: $w_{k,l} = (1-x_k)x_{k+1} \dots x_l (1-x_{l+1}+x_{l+1}) = y_{k,l+1} + w_{k,l+1} = y_{k,l+1} + \sum_{r=l+2}^n y_{k,r} = \sum_{r=l+1}^n y_{k,r}$.
	
	In a similar vein, we show for $1 \le p \le q \le n$, the affine transformation between $\{z_{p,q}\}$ and $\{w_{k,l}\}$ variables to be
	\begin{equation}
	z_{p,q} = 1- \sum_{r=p}^q w_{r,q},
	\label{eq:z-x-affine}
	\end{equation}
	using induction on $q-p$. For $q=p$, \eqref{eq:z-x-affine} follows because $z_{q,q} = x_q = 1-(1-x_q) = 1-w_{q,q}$. Next, assuming \eqref{eq:z-x-affine} holds for $p+1$ \ie{} $z_{p+1,q} = 1 - \sum_{r=p+1}^q w_{r,q}$, we show that it holds for $z_{p,q}$ as well: $z_{p,q} = \prod_{r=p}^q x_r = [1-(1-x_p)]\prod_{r=p+1}^q x_r = z_{p+1,q} - w_{p,q} = 1-\sum_{r=p+1}^q w_{r,q} -w_{p,q} = 1 - \sum_{r=p}^q w_{p,q}$.
	Finally, substituting \eqref{eq:w-y-affine} in \eqref{eq:z-x-affine} leads to the required inverse affine transformation given below:
	\begin{equation}
	z_{p,q} = 1 - \sum_{r=p}^q \sum_{s=q+1}^n y_{r,s}.
	\label{eq:inv-affine}
	\end{equation}
	Indeed, the correctness of \eqref{eq:inv-affine} can be checked via direct verification using $y_{r,s}=z_{r+1,s-1}-z_{r+1,s}-z_{r,s-1}+z_{r,s}$, $z_{i,n}=0$ for $1 \le i \le n$, and $z_{i+1,i}=1$ for $1\le i \le n$. 
\end{proof}

We note that Proposition \ref{prop:invertible-transform} shows, by defining $n=N-i+1$ (\resp $n=j$) and $x_r=1-\zeta_{i,N-r+1}$ (\resp $x_r=1-\zeta_{r,j}$), there is an invertible linear transformation between $\{\tau_{i,k,j}\}_{i \le k < j \le N}$ and $\{\nu_{i,k,j}\}_{i \le k \le j \le N}$ (\resp $\{\beta_{i,l,j}\}_{1 \le i < l \le j}$ and $\{\omega_{i,l,j}\}_{1 \le i < l \le j}$). We expressed $\tau$ (\resp $\beta$) as an affine function of $\nu$ (\resp $\omega$) in \eqref{eq:secvar}. The inverse transformation is:
\begin{align}
& \text{for } \; [i,j] \in \mathcal{T}, \quad \nu_{i,k,j} = \begin{cases}
0, & \text{for } k = i\\
1-\sum\limits_{s=k}^j\sum\limits_{r=i}^{k-1} \tau_{i,r,s}, & \text{for } i+1 \le k \le j
\end{cases}, \label{eq:nu-tau-inverse}\\
& \text{for } \; [i,j] \in \mathcal{T}, \quad \omega_{i,l,j} = \begin{cases}
1-\sum\limits_{r=i}^l\sum\limits_{s=l+1}^{j} \beta_{r,s,j}, & \text{for } i \le l \le j-1\\
0, & \text{for } l = j.
\end{cases} \label{eq:omega-tau-inverse}
\end{align}
Since $\nu_{i,j,j}=\omega_{i,i,j}=1-\zeta_{i,j}$, Corollary \ref{prop:sumteqz} follows directly from \eqref{eq:nu-tau-inverse} and \eqref{eq:omega-tau-inverse}.

\begin{corollary}
	\eqref{eq:secvar}--\eqref{eq:feedprod} imply that $\sum_{k=i}^{j-1}\tau_{i,k,j} =\sum_{l=i+1}^j \beta_{i,l,j} = \zeta_{i,j}$ for all $[i,j]\in \mathcal{P}$. \qed 
	\label{prop:sumteqz}
\end{corollary}

\subsubsection{Conservation of components}
\label{sec:compconservation}
Corollary \ref{prop:sumteqz} has the physical interpretation that the stream $[i,j]$, when present, produces exactly one stream as distillate and one stream as residue. However, the distillate and residue streams cannot be chosen arbitrarily. They must be chosen such that, all components are conserved when $[i,j]$ undergoes a split. In other words, for $\llbracket k \rrbracket_{i}^{j-1}$ (\resp $\llbracket l \rrbracket_{i+1}^j$), if $[i,j]\uparrow [i,k]$ (\resp $[i,j]\downarrow [l,j]$), then for conservation of components, the residue (\textit{resp.} distillate) from $[i,j]$ must be one of $\{[l,j]\}_{l=i+1}^{k+1}$ (\resp $\{[i,k]\}_{k=l-1}^{j-1}$). Consider the digraph shown in Figure \ref{fig:SplitFeasibility} for stream $[i,j]$.

\begin{figure}[h]
	\centering
	\includegraphics[scale=1]{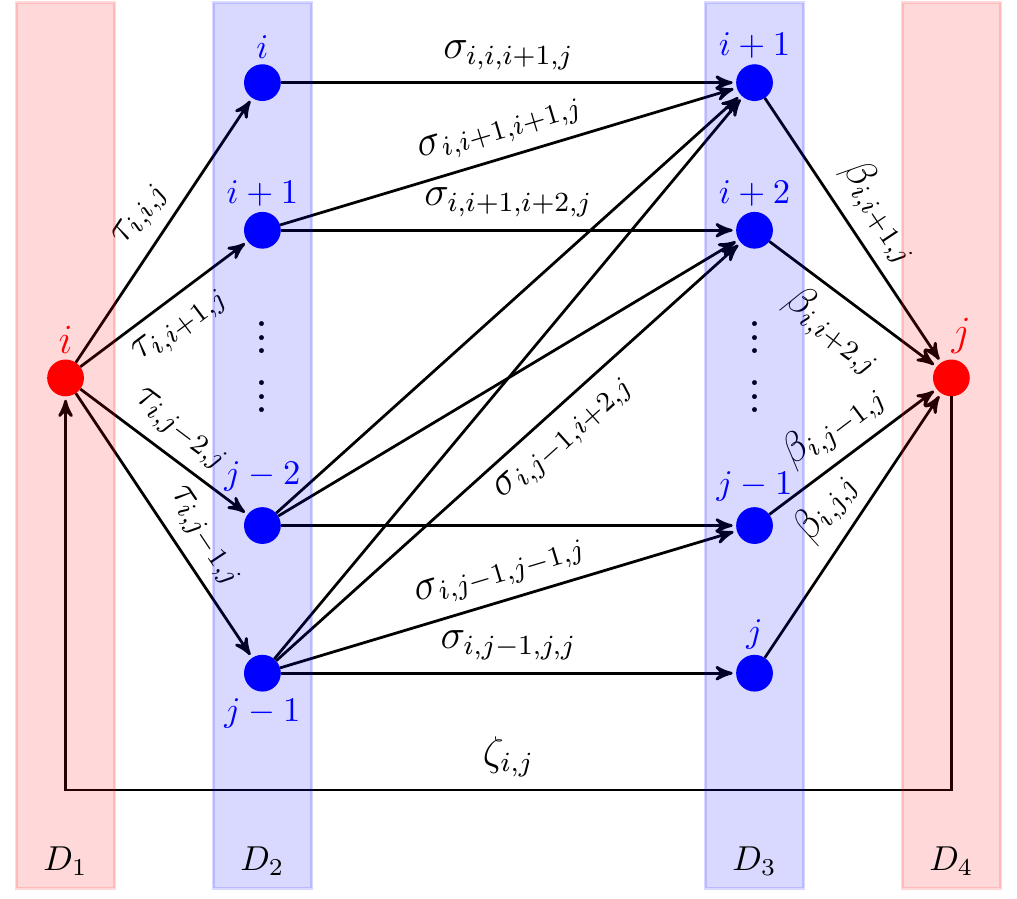}
	\caption{Digraph for deriving conservation of components constraint in \textsection\ref{sec:compconservation}}
	\label{fig:SplitFeasibility}
\end{figure}
 
We partition the nodes into four sets $D_1$ through $D_4$, where $D_1 = \{i\}$ (\resp $D_4=\{j\}$), and $D_2 = \{k\}_{k=i}^{j-1}$ (\resp $D_3 = \{l\}_{l=i+1}^j$) contains the heaviest (resp. lightest) component in the top (\resp bottom) children of $[i,j]$. The edges in $D_1\times D_2$ (\resp $D_3 \times D_4$) correspond to all plausible distillate (\resp residue) streams from $[i,j]$. Edges in $D_2\times D_3$ correspond to feasible \textit{splits} of $[i,j]$, \ie{} each node $k \in D_2$ connects to $\{i+1,\dots,k+1\} \in D_3$. We associate these edges with auxiliary variables $\bigcup_{k=i}^{j-1}\; \{\sigma_{i,k,l,j}\}_{l=i+1}^{k+1}$, referred as \textit{split} variables hereafter (see Figure \ref{fig:SplitFeasibility}). We let $\sigma_{i,k,l,j}=\{1,\;\text{if }[i,k]/[l,j];\;0,\;\text{otherwise}\}$, and write mass balances on the network by interpreting stream, section and split variables as material flows along the respective edges of the graph.
\begin{eqnal}
	\text{For } \; [i,j] \in \mathcal{P}\mskip 15mu  \left\{\begin{aligned}
		& \sum_{l=i+1}^{k+1}\sigma_{i,k,l,j} =\tau_{i,k,j},\, \llbracket k \rrbracket_{i}^{j-1}; \mskip 10mu
		\sum_{k=l-1}^{j-1}\sigma_{i,k,l,j} =\beta_{i,l,j},\, \llbracket l \rrbracket_{i+1}^j; \\
		& \sigma_{i,k,l,j} \ge 0, \llbracket l \rrbracket_{i+1}^{k+1}, \llbracket k \rrbracket_{i}^{j-1}.
	\end{aligned}\right.
	\label{eq:RLTbin}
\end{eqnal}
Mass balances around the nodes in $D_1$ and $D_4$, and non-negativity constraint on section variables are implied from \eqref{eq:secvar}-- \eqref{eq:feedprod} (see Corollary \ref{prop:sumteqz} and Remark \ref{rem:linearization-tb}), so it is not required to impose them explicitly. We show below that, for any $[i,j] \in \mathcal{P}$, the relaxation \eqref{eq:secvar}--\eqref{eq:RLTbin} is the best possible for the substructure represented by the digraph in Figure \ref{fig:SplitFeasibility}.

\begin{proposition}
	The constraints \eqref{eq:secvar}--\eqref{eq:RLTbin}, and $0 \le \zeta_{i,j} \le 1$ define a set such that, for any $[i,j] \in \mathcal{P}$, $(\sigma,\tau,\beta,\zeta)$ is contained in the convex hull of
	\begin{gather}
	S_{i,j} = \left\{ (\sigma,\tau,\beta,\zeta) \middle|
	\begin{aligned}
	& \sigma_{i,k,l,j} = \tau_{i,k,j}\beta_{i,l,j}, & & \llbracket l \rrbracket_{i+1}^{k+1}; \quad \llbracket k \rrbracket_{i}^{j-1}\\
	& \tau_{i,k,j}\beta_{i,l,j} =0,		& & \llbracket l \rrbracket_{k+2}^{j}; \quad \llbracket k \rrbracket_{i}^{j-2}\\
	& \sum_{k=i}^{j-1} \tau_{i,k,j} = \sum_{l=i+1}^{j} \beta_{i,l,j} = \zeta_{i,j}, & &\\
	& \tau_{i,k,j}, \; \beta_{i,l,j}, \; \zeta_{i,j} \in \{0,1\}& &\llbracket l \rrbracket_{i+1}^{j}; \quad \llbracket k \rrbracket_{i}^{j-1}
	\end{aligned}
	\right\}.
	\end{gather}
	\label{prop:conv-hull-SplitFeas}
\end{proposition}
\begin{proof}
First, note that \eqref{eq:RLTbin}, equations in Corollary \ref{prop:sumteqz}, $0 \le \zeta_{i,j} \le 1$ and non-negativity of section variables together constitute a network flow polytope (see Figure \ref{fig:SplitFeasibility}) in $(\tau,\beta,\sigma,\zeta)$ space. The extreme points of the polytope are integral, and are given by
	\begin{subequations}
		\begin{align}
		& \left. \begin{aligned}
		& \zeta_{i,j} = \tau_{i,k,j} = \beta_{i,l,j} = \sigma_{i,k,l,j} = 1, \\
		& \tau_{i,k',j} = \beta_{i,l',j} = 0, \quad \text{for} \quad k' \ne k, \; l' \ne l
		\end{aligned} \right\}\quad \llbracket l \rrbracket_{i+1}^{k+1}; \quad \llbracket k \rrbracket_{i}^{j-1},
		\label{eq:feas-split-proof1}\\
		& \left. \zeta_{i,j} = \tau_{i,k,j} = \beta_{i,l,j} = \sigma_{i,k,l,j} = 0.\right. \label{eq:feas-split-proof2}
		\end{align}
	\end{subequations}
	We show that the only solutions to $S_{i,j}$ are those in \eqref{eq:feas-split-proof1} and \eqref{eq:feas-split-proof2}. Assume $\zeta_{i,j} =0$. Then, $\tau_{i,k,j}=0$ for $\llbracket k \rrbracket_{i}^{j-1}$, $\beta_{i,l,j} = 0$ for $\llbracket l \rrbracket_{i+1}^{j}$ and $\sigma_{i,k,l,j}=0$ for $\llbracket l \rrbracket_{i+1}^{k+1}; \; \llbracket k \rrbracket_{i}^{j-1}$. Now, assume $\zeta_{i,j}=1$. Then, there exists $k$ and $l$ satisfying $i<l\le k+1 \le j$ such that $\tau_{i,k,j}=\beta_{i,l,j}=\sigma_{i,k,l,j}=1$ and for $k' \ne k$, $l' \ne l$; $\tau_{i,k',j}=\beta_{i,l',j}=\sigma_{i,k',l',j}=0$. 
\end{proof}

\subsubsection{Presence of a parent}
\label{sec:parent-presence}
Stream $[i,j] \in \mathcal{T} \setminus \{[1,N]\}$ is present in a configuration, only if it is produced as a distillate from one of its top parents and/or as a residue from one of its bottom parents. To derive the required constraints, we consider the digraph shown in Figure \ref{fig:ParentConstraints}. 

\begin{figure}[h]
	\centering
	\includegraphics[scale=1]{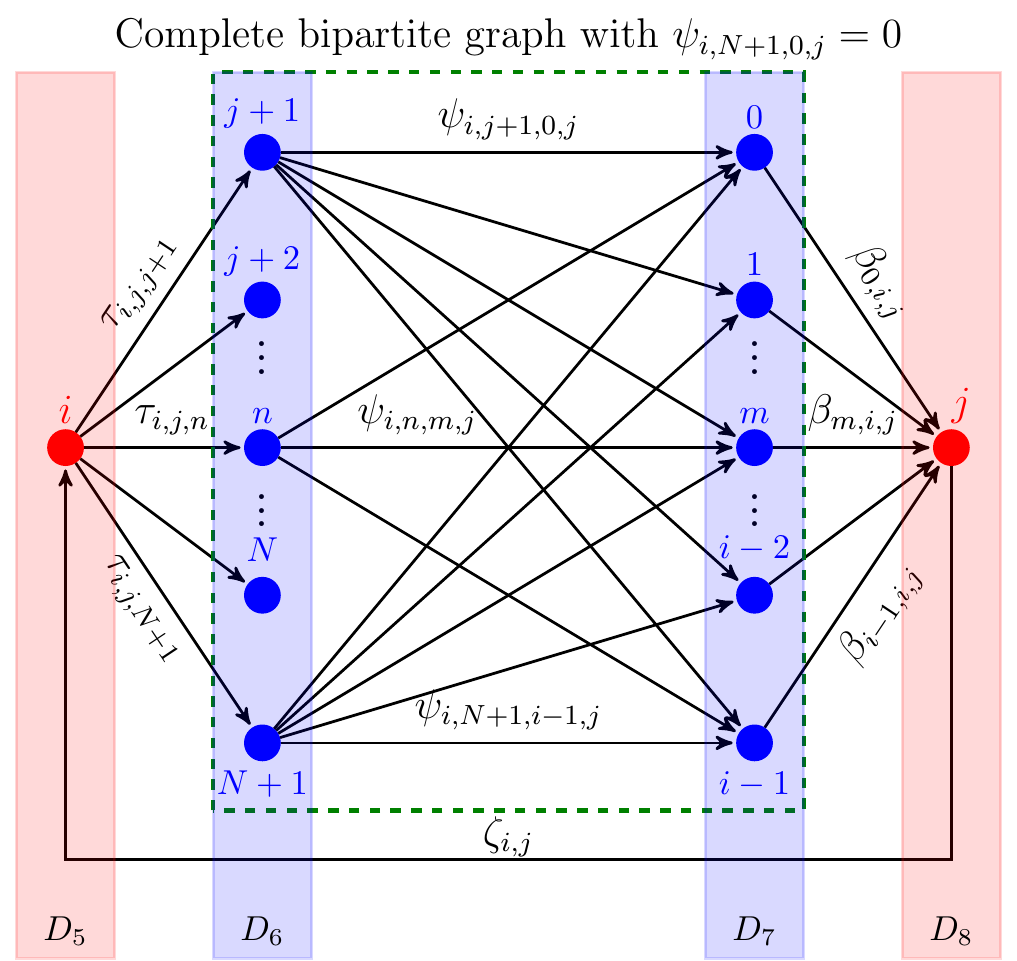}
	\caption{Digraph for deriving presence of parent constraint in \textsection\ref{sec:parent-presence}}
	\label{fig:ParentConstraints}
\end{figure}
 
The graph is inspired from the observation that $\sum_{n=j+1}^{N+1} \tau_{i,j,n} = \zeta_{i,j}$ and $\sum_{m=0}^{i-1} \beta_{m,i,j} = \zeta_{i,j}$, where we define $\tau_{i,j,N+1} = \nu_{i,j+1,N} - \nu_{i,j,N}$ and $\beta_{0,i,j} = \omega_{1,i-1,j} - \omega_{1,i,j} $. From \eqref{eq:mccor}, it can be verified that $0 \le \tau_{i,j,N+1} \le 1$ and $0 \le \beta_{0,i,j} \le 1$. Physically, $\tau_{i,j,N+1} = 1$ (\resp $\beta_{0,i,j}=1$) indicates that $[i,j]$ is not produced as distillate (\resp residue), because $\tau_{i,j,N+1}=1$ (\resp $\beta_{0,i,j}=1$) \textit{iff} $[i,j]$ is present ($\zeta_{i,j}=1$) and all its top (\resp bottom) parents are absent i.e., $\nu_{i,j+1,N}=1$ (\resp $\omega_{1,i-1,j}=1$).

As in \textsection \ref{sec:compconservation}, we partition the nodes into four sets $D_5$ through $D_8$ (see Figure \ref{fig:ParentConstraints}), where $D_5 = \{i\}$ (\resp $D_8=\{j\}$), 
and $D_6 = \{n\}_{n=j+1}^{N+1}$ (\resp $D_7 = \{m\}_{m=0}^{i-1}$) contains the heaviest (\resp lightest) component in the top (\resp bottom) parents of $[i,j]$. Recall that $m=0$ and $n=N+1$ have a special meaning as described in the previous paragraph. The edges in $D_5\times D_6$ (\resp $D_7 \times D_8$) correspond to all plausible ways $[i,j]$ can be produced as distillate (\resp residue), and the edges in $D_6 \times D_7$ indicate whether $[i,j]$ is produced only as distillate or only as residue or both. We introduce variables for edges in $D_6 \times D_7$ such that $\psi_{i,n,m,j}=1$ \textit{iff} $[i,n] \uparrow [i,j]$ and $[m,j] \downarrow [i,j]$.

We require that $\psi_{i,N+1,0,j} = 0$, which, otherwise, would mean that $[i,j]$ can be present even if it is neither produced as distillate nor as residue. Now, we write mass balances on the network.
\begin{align}
\text{for } \; [i,j] \in \mathcal{T}\setminus\{[1,N]\} \quad \left\{\begin{aligned}
&\sum_{m=0}^{i-1} \psi_{i,n,m,j} = \tau_{i,j,n},\, \llbracket n \rrbracket_{j+1}^{N+1};\mskip 15mu
\sum_{n=j+1}^{N+1} \psi_{i,n,m,j} = \beta_{m,i,j}, \llbracket m \rrbracket_{0}^{i-1};\\
&\psi_{i,n,m,j} \ge 0,\, \llbracket n \rrbracket_{j+1}^{N+1},\, \llbracket m \rrbracket_{0}^{i-1};\mskip 15mu
\psi_{i,N+1,0,j} = 0.
\end{aligned}\right.
\label{eq:Parent-network}
\end{align}

Mass balances around the nodes in $D_5$ and $D_8$, and non-negativity constraint on section variables are implied from \eqref{eq:secvar} and \eqref{eq:mccor}, so it is not required to impose them explicitly.

\begin{proposition}
	The constraints \eqref{eq:secvar}, \eqref{eq:mccor}, \eqref{eq:Parent-network} and $0 \le \zeta_{i,j} \le 1$ define a set such that, for every $[i,j] \in \mathcal{T}\setminus\{[1,N]\}$, $(\tau,\beta,\zeta,\psi)$ is contained in the convex hull of
	\begin{gather}
	S_{i,j} = \left\{ (\tau,\beta,\zeta,\psi) \middle|
	\begin{aligned}
	& \psi_{i,n,m,j} = \tau_{i,j,n}\beta_{m,i,j}, & & \llbracket m \rrbracket_{0}^{i-1}; \quad \llbracket n \rrbracket_{j+1}^{N+1}\\
	& \sum_{n=j+1}^{N+1} \tau_{i,j,n} = \sum_{m=0}^{i-1} \beta_{m,i,j} = \zeta_{i,j},\quad  \psi_{i,N+1,0,j} =0,& &\\
	& \tau_{i,j,n}, \; \beta_{m,i,j}, \; \zeta_{i,j} \in \{0,1\}& &\llbracket m \rrbracket_{0}^{i-1}; \quad \llbracket n \rrbracket_{j+1}^{N+1}
	\end{aligned}
	\right\}.
	\end{gather}
	\label{prop:conv-hull-parents}
\end{proposition}

\begin{proof}
We use a similar argument as the one used to prove Proposition \ref{prop:conv-hull-SplitFeas}. We recognize that \eqref{eq:Parent-network}, $\sum_{n=j+1}^{N+1} \tau_{i,j,n} = \sum_{m=0}^{i-1} \beta_{m,i,j} = \zeta_{i,j}$, $0 \le \zeta_{i,j} \le 1$ and non-negativity requirement on section variables together constitute a network flow polytope, whose extreme points are integral and precisely those in $S_{i,j}$.
\end{proof}

\subsubsection{Constraints on Heat Exchanger Variables}
For every $(i,j) \in \mathcal{C}$, condenser $(i,j)$ is present only if the stream $[i,j]$ is not produced as residue, \ie{} $\beta_{0,i,j}=1$ \citep{tumbalamgooty2019}. Similarly, for every $(i,j) \in \mathcal{R}$, reboiler $(i,j)$ is present only if the stream $[i,j]$ is not produced as distillate, \ie{} $\tau_{i,j,N+1}=1$. Further, a condenser (\resp reboiler) must be present with a pure component $[i,i]$, if $[i,i]$ is not produced as residue (\resp distillate) i.e. $\beta_{0,i,i} = 1$ (\resp $\tau_{i,i,N+1}=1$).
\begin{eqnal}
	& 
	\begin{aligned}
		& \chi_{i,j} \le \beta_{0,i,j}, \forall \; (i,j) \in \mathcal{C};\quad
		& \rho_{i,j} \le \tau_{i,j,N+1}, \forall \; (i,j) \in \mathcal{R}
	\end{aligned}  \label{eq:hex1}\\
	& \begin{aligned}
		& \chi_{i,i} \ge \beta_{0,i,i}, \forall \; (i,i) \in \mathcal{C};\quad
		& \rho_{i,i} \ge \tau_{i,i,N+1}, \forall \; (i,i) \in \mathcal{R}.
	\end{aligned} \label{eq:hex2}
\end{eqnal}

\begin{proposition}
	\label{prop:conv-hull-PresParent}
	The constraints \eqref{eq:secvar}--\eqref{eq:hex2}, \eqref{eq:Parent-network}, $0\le \zeta_{i,j} \le 1$, $\chi_{i,j} \ge 0$ and $\rho_{i,j} \ge 0$ define a set that, for every $[i,j] \in \mathcal{S}$, is contained in the convex hull of solutions that satisfy at least one of the following conditions, where unspecified $\tau_{i,\cdot,j}$, $\beta_{i,\cdot,j}$, $\sigma_{i,\cdot,\cdot,j}$, $\psi_{i,\cdot,\cdot,j}$, $\chi_{i,j}$, and $\rho_{i,j}$ variables are zero:
	\begin{enumerate}
		\item for some $1 \le m \le i-1$, $j+1 \le n \le N$, and $i < l \le k +1 \le j$, we have $\zeta_{i,j} = \tau_{i,k,j} = \beta_{i,l,j} = \sigma_{i,k,l,j}= \tau_{i,j,n} = \beta_{m,i,j}=\psi_{i,n,m,j}  = 1$,
		\item for some $j+1 \le n \le N$, and $i < l \le k +1 \le j$, we have $\zeta_{i,j} = \tau_{i,k,j} = \beta_{i,l,j} = \sigma_{i,k,l,j} = \tau_{i,j,n} = \beta_{0,i,j} = \psi_{i,n,0,j} =1; \; \chi_{i,j}=1 \; \textnormal{or } \; 0$,
		\item for some $1 \le m \le m-1$, and $i < l \le k +1 \le j$, we have $\zeta_{i,j} = \tau_{i,k,j} = \beta_{i,l,j} = \sigma_{i,k,l,j} = \tau_{i,j,N+1} = \beta_{m,i,j} = \psi_{i,N+1,m,j} = 1; \; \rho_{i,j}=1 \; \textnormal{or } \; 0$,
		\item all the variables are zero.
	\end{enumerate}
\end{proposition}
\begin{proof}
We modify the graph in Figure \ref{fig:ParentConstraints} to accommodate \eqref{eq:hex1} and \eqref{eq:hex2}, and combine it with the graph in Figure \ref{fig:SplitFeasibility}. The resulting graph is shown in Figure \ref{fig:HexProof}. Next, observe that \eqref{eq:RLTbin}, \eqref{eq:Parent-network}, $\sum_{k=i}^{j-1}\tau_{i,k,j} = \sum_{l=i+1}^j \beta_{i,l,j} = \sum_{n=j+1}^{N+1} \tau_{i,j,n} = \sum_{m=0}^{i-1} \beta_{m,i,j} = \zeta_{i,j}$, $0\le \zeta_{i,j} \le 1$ (which are implied from \eqref{eq:secvar}--\eqref{eq:feedprod}), and non-negative constraint on all variables together constitute a network flow polytope. The extreme points this polytope are integral, and are precisely those mentioned in the Proposition. 
\end{proof}

\begin{figure}[h]
	\centering
	\includegraphics[scale=0.9]{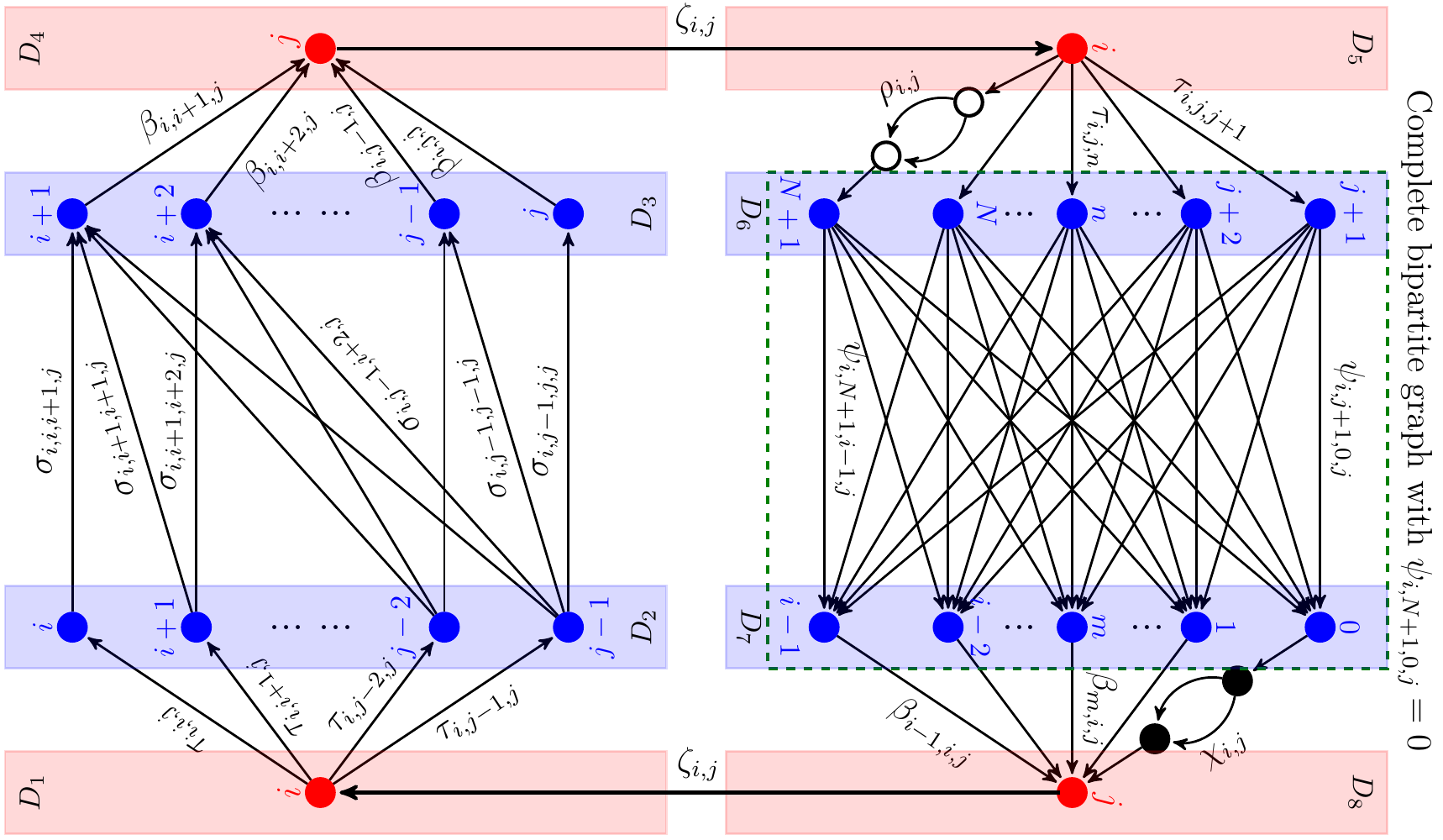}
	\caption{Digraph for the proof of Proposition \ref{prop:conv-hull-PresParent}}
	\label{fig:HexProof}
\end{figure}

Since $\psi$ variables are not used elsewhere, we project \eqref{eq:Parent-network} to the space of section variables  ($\tau,\beta$).
\begin{proposition}
	\label{prop:projection-parent-network}
	For every $[i,j] \in \mathcal{T}\setminus\{[1,N]\}$, let $S_{i,j} = \{ (\tau,\beta,\psi) \mid \eqref{eq:Parent-network};\;\sum_{m=0}^{i-1}\beta_{m,i,j} = \sum_{n=j+1}^{N+1} \tau_{i,j,n};\; \tau_{i,j,n} \ge 0,\; \llbracket n \rrbracket_{j+1}^{N+1}; \; \beta_{m,i,j} \ge 0,\; \llbracket m \rrbracket_{0}^{i-1} \}$. Then, the projection of $S_{i,j}$ in $(\tau,\beta)$ space is
	\begin{equation}
	\textnormal{proj}_{(\tau,\beta)}(S_{i,j})= \left\{ (\tau,\beta) \middle| \begin{aligned}
	& \beta_{0,i,j} \le \sum_{n=j+1}^N \tau_{i,j,n}; \quad \sum_{m=0}^{i-1}\beta_{m,i,j} = \sum_{n=j+1}^{N+1} \tau_{i,j,n}\\
	& \tau_{i,j,n} \ge 0, \; \llbracket n \rrbracket_{j+1}^{N+1}; \quad \beta_{m,i,j} \ge 0, \; \llbracket m \rrbracket_{0}^{i-1}
	\end{aligned} \right\}.
	\label{eq:projection-parent}
	\end{equation}
\end{proposition}
\begin{proof}
See \textsection \ref{sec:ecom:projection-proof} in the Appendix.
\end{proof}

Apart from the following, the remaining constraints in \eqref{eq:projection-parent} follow from \eqref{eq:secvar} and \eqref{eq:mccor}:
\begin{eqn}
	\text{for } \; [i,j]\in \mathcal{T}\setminus\{[1,N]\}, \quad \beta_{0,i,j} \le \sum_{n=j+1}^{N} \tau_{i,j,n}.
	\label{eq:Presence-of-parent}
\end{eqn}
\begin{remark}
	Using \eqref{eq:RLTbin}, \eqref{eq:nu-tau-inverse} and \eqref{eq:omega-tau-inverse}, $\tau$, $\beta$, $\nu$ and $\omega$ variables can be substituted out. \qed
\end{remark}
Constraints \eqref{eq:feedprod}--\eqref{eq:Presence-of-parent} model the space of admissible configurations. We compare this formulation with CG06, GA10, and TAT19, which refer to the formulations of \citet{caballero2006}, \citet{giridhar2010b}, and \citet{tumbalamgooty2019}, respectively.

\begin{proposition}
	The feasible region defined using constraints \eqref{eq:feedprod}--\eqref{eq:Presence-of-parent} is tighter than the set by imposing the constraints in the formulations of CG06, GA10, and TAT19. 
	\label{prop:searchspace}
\end{proposition}
\begin{proof}
See \textsection \ref{sec:ecom:search-space-proof} in the Appendix.
\end{proof}
The fact that our formulation is strictly tighter will follow from numerical examples.

\subsection{Mass Balance Constraints}
We model the problem as a network flow problem. Figure \ref{fig:flow-network} shows the representative nodes and arcs in the network, and variable definitions are in Table \ref{tab:notation}. Each split $[i,k]/[l,j]$ is performed in a distillation column $Q_{iklj}$ (see Figures \ref{fig:flow-network}(a) and \ref{fig:flow-network}(b)). Material flows to and from the column $Q_{iklj}$ only when $\sigma_{iklj}=1$. The material balances across each column $Q_{iklj}$ are as follows
\begin{eqnal}
	& \text{for } \; [i,j] \in \mathcal{S}, \;  \llbracket k \rrbracket_{i}^{j-1}, \; \llbracket l \rrbracket_{i+1}^{k+1}:\nonumber\\
	&\quad \left.
	\begin{alignedat}{2}
		& f^\feed_{ikljp} = f^\rec_{ikljp}\delta_{p\le k} + f^\strip_{ikljp}\delta_{p\ge l}, \llbracket p \rrbracket_{i}^j;
		&\mskip 10mu&U^\rec_{iklj}\delta_{j < N} - U^\strip_{iklj}\delta_{1<i} = V^\rec_{iklj} - V^\strip_{iklj}\\
		&K^\strip_{iklj}\delta_{1<i} - K^\rec_{iklj}\delta_{j<N} = L^\strip_{iklj} - L^\rec_{iklj}\\
		&\mathrlap{0 \le (\cdot)  \le \sigma_{iklj} \; (\cdot)^\upbnd,
			\forall \; (\cdot) \in \{ \text{All component, liquid and vapor flows}\}}
	\end{alignedat}
	\right\}, \label{eq:comp-vap-bal-on-S}\\\noalign{\vskip 1em}
	& \text{for } \; [i,j] \in \{[1,N]\}, \;  \llbracket k \rrbracket_{i}^{j-1}, \; \llbracket l \rrbracket_{i+1}^{k+1}:\nonumber\\
	&\quad\left.
	\begin{alignedat}{2}
		& F_p \sigma_{iklj} = f^\rec_{ikljp}\delta_{p\le k} + f^\strip_{ikljp}\delta_{p\ge l}, \llbracket p \rrbracket_{i}^j;
		&\mskip 10mu&\left(\sum\nolimits_{p=1}^N F_p\right) (1-\Phi_{1,N})\sigma_{iklj} = V^\rec_{iklj} - V^\strip_{iklj}\\
		& \left(\sum\nolimits_{p=1}^N F_p\right) \Phi_{1,N} \sigma_{iklj} = L^\strip_{iklj} - L^\rec_{iklj}\\
		& \mathrlap{0 \le (\cdot)  \le \sigma_{iklj} \; (\cdot)^\upbnd,
			\forall \; (\cdot) \in \{ \text{All component, liquid and vapor flows} \}}
	\end{alignedat}
	\right\}, \label{eq:comp-vap-bal-on-feed}\\\noalign{\vskip 1em}
	& \text{for } [i,j] \in \mathcal{P}, \;  \llbracket k \rrbracket_{i}^{j-1}, \; \llbracket l \rrbracket_{i+1}^{k+1}:\nonumber\\
	&\quad 
	\begin{alignedat}{2}
		& V^\rec_{iklj} - L^\rec_{iklj} = \sum\nolimits_{p=i}^k f^\rec_{ikljp};
		&\mskip 20mu& L^\strip_{iklj} - V^\strip_{iklj} = \sum\nolimits_{p=l}^j f^\strip_{ikljp}.
	\end{alignedat}
	\label{eq:dis-res-bal}
\end{eqnal}
The constraints in \eqref{eq:comp-vap-bal-on-S} model component, vapor, and liquid mass balances across column $Q_{iklj}$. In the above $\delta_{(\cdot)}$ is $1$ if $(\cdot)$ is true and $0$ otherwise. 
\eqref{eq:comp-vap-bal-on-feed} handles the case where the feed stream is the process feed, $[1,N]$. $F_p$ and $\Phi_{1,N}$ are as defined in \textsection\ref{sec:Prob-Definition}.
The last constraint in both \eqref{eq:comp-vap-bal-on-S} and \eqref{eq:comp-vap-bal-on-feed} 
suppresses material flows to column $Q_{iklj}$ when $\sigma_{iklj}=0$. We use $(\cdot)^\upbnd$ to denote the upper bound on $(\cdot)$, and discuss how these are obtained later. The first (\resp second) constraint in \eqref{eq:dis-res-bal} models that the net distillate (\resp residue) flow $Q_{iklj}$ as the difference between the vapor and liquid (\resp liquid and vapor) flows in the rectifying (\resp stripping) section.

\begin{figure}
	\centering
	\subfigure[]{\includegraphics[scale=1]{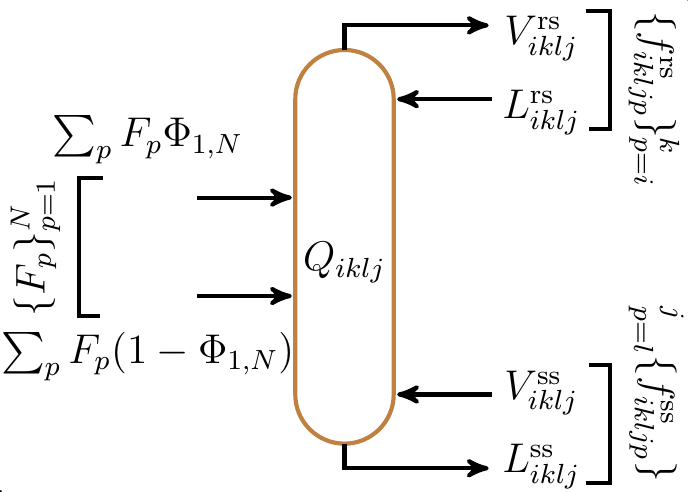}}
	\subfigure[]{\includegraphics[scale=1]{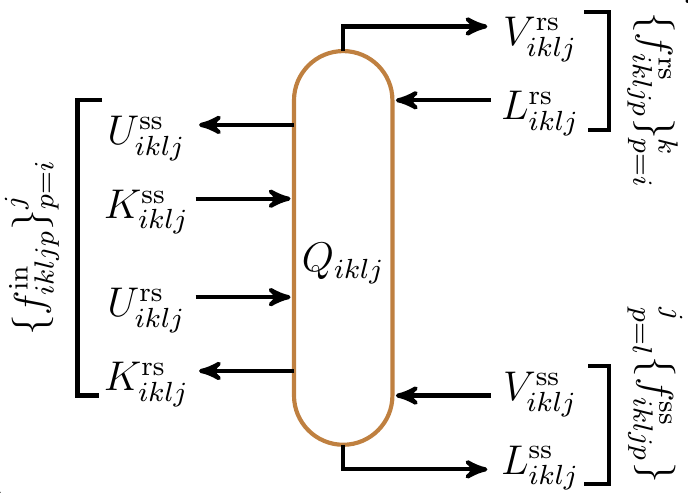}}
	\subfigure[]{\includegraphics[scale=1]{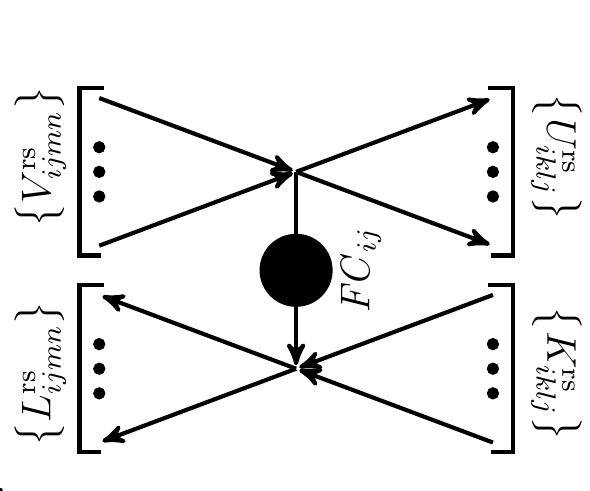}}
	\subfigure[]{\includegraphics[scale=1]{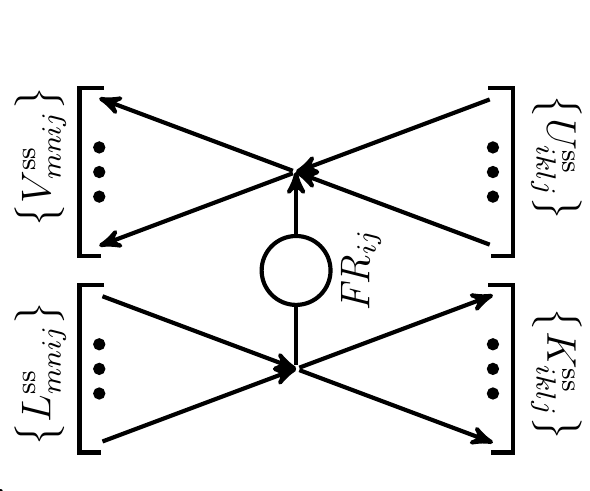}}
	\subfigure[]{\includegraphics[scale=1]{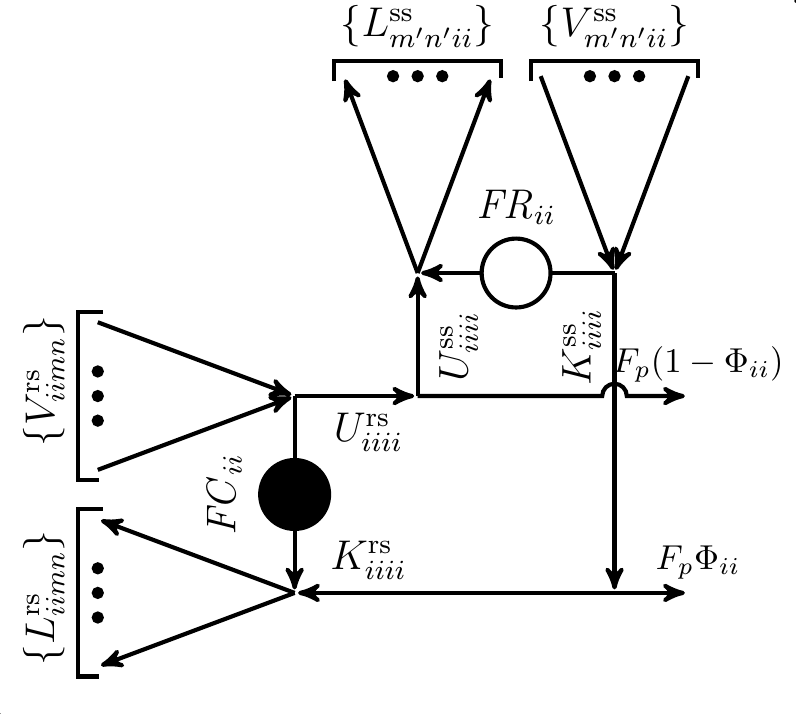}}
	\subfigure[]{\includegraphics[scale=1]{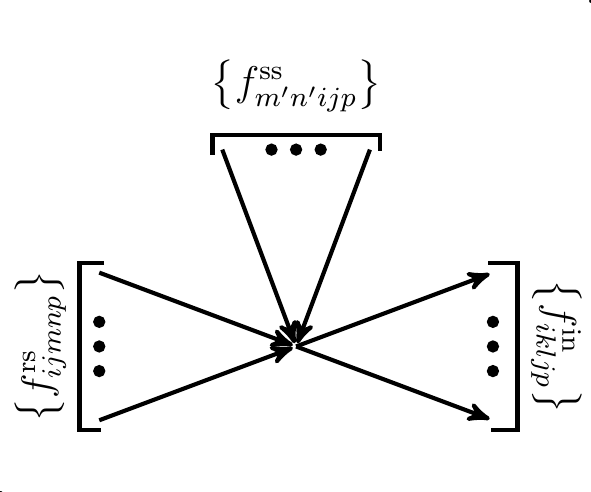}}
	\caption{(a) Representative column for splits of process feed i.e., $[i,j] \in \{[1,N]\}, \; \llbracket k \rrbracket_{i}^{j-1}, \; \llbracket l \rrbracket_{i+1}^{k+1}$ (b) Representative column for the remaining splits $[i,j] \in \mathcal{S}, \; \; \llbracket k \rrbracket_{i}^{j-1}, \; \llbracket l \rrbracket_{i+1}^{k+1}$ (c) Representative condenser for $(i,j) \in \mathcal{C}\setminus \{[i,i]\}_{i=1}^{N-1}$ (see \eqref{eq:cond-mass-bal} for domain of indices $m,n,k,l$) (d) Representative reboiler for $(i,j) \in \mathcal{R}\setminus \{[i,i]\}_{i=2}^N $ (see \eqref{eq:reb-mass-bal} for domain of indices $m,n,k,l$) (e) Representative arrangement for pure product withdrawals (see \eqref{eq:cond-mass-bal-pure} for domain of indices $m,n$, and \eqref{eq:reb-mass-bal-pure} for domain of indices $m',n'$) (f) Representative arrangement for overall component mass balance for $[i,j] \in \mathcal{S}$ (see \eqref{eq:net-comp-bal} for domain of indices $m,n,m',n',k,l$) }
	\label{fig:flow-network}
\end{figure}

\begin{table}[!hbt]
	\centering
	\renewcommand\arraystretch{1.5}
	\begin{tabular}{l@{\hskip 0.25in}l@{\hskip 0.25in}}
		\toprule
		Variable & Definition\\
		\midrule
		$\bigl\{f_{ikljp}^{\rec}\bigr\}_{p=i}^{k}$	&	Net molar flow of component $p$ in the rectifying section of $Q_{iklj}$\\
		$\bigl\{f_{ikljp}^{\strip}\bigr\}_{p=l}^j$	&	Net molar flow of component $p$ in the stripping section of $Q_{iklj}$\\
		$\bigl\{f_{ikljp}^{\feed}\bigr\}_{p=i}^{j}$	&	Net molar flow of component $p$ in the feed to $Q_{iklj}$\\
		$V_{iklj}^{\rec}$	&	Vapor flowrate in the rectifying section of $Q_{iklj}$\\
		$V_{iklj}^{\strip}$&	Vapor flowrate in the stripping section of $Q_{iklj}$\\
		$L_{iklj}^{\rec}$	&	Liquid flowrate in the rectifying section of $Q_{iklj}$\\
		$L_{iklj}^\strip$	&	Liquid flowrate in the stripping section of $Q_{iklj}$\\
		$U_{iklj}^{\rec}$&   Vapor in-flow into $Q_{iklj}$ from condenser $(i,j)$\\
		$U_{iklj}^{\strip}$&   Vapor out-flow from $Q_{iklj}$ to reboiler $(i,j)$\\
		$K_{iklj}^{\rec}$&   Liquid out-flow from $Q_{iklj}$ to condenser $(i,j)$\\
		$K_{iklj}^{\strip}$&   Liquid in-flow into $Q_{iklj}$ from reboiler $(i,j)$\\
		$\bigl\{\theta_{ijq}\bigr\}_{q=i}^{j-1}$&   Underwood root of $Q_{iklj}$ satisfying $\alpha_{q+1}\le \theta_{i,j,q}\le \alpha_q$\\
		$\Upsilon_{iklj}^{\rec}$&  Minimum vapor flow required in the rectifying section of $Q_{iklj}$\\
		$\Upsilon_{iklj}^{\strip}$ & Minimum vapor flow required in the stripping section of $Q_{iklj}$\\
		$\FC_{ij}$		&	Molar flowrate in condenser $(i,j)$\\
		$\FR_{ij}$		&	Molar flowrate in reboiler $(i,j)$\\
		\bottomrule
	\end{tabular}
	\caption{Definition of continuous decision variables.}
	\label{tab:notation}
\end{table}

Column $Q_{iklj}$ receives feed from the associated condenser $(i,j)$ and/or reboiler $(i,j)$ (see Figures \ref{fig:flow-network}(c) and \ref{fig:flow-network}(d)). Further, condenser (\resp reboiler)  $(i,j)$ regulates vapor-liquid traffic from all the splits producing $[i,j]$ as distillate (\resp residue), and distributes flows to all the splits of $[i,j]$. Material balances across these condensers and reboilers are given below:
\begin{eqnal}
	& \text{For } (i,j) \in \mathcal{C}\setminus\{[i,i]\}_{i=1}^{N-1}:\nonumber\\
	&\quad
	\left.\begin{alignedat}{2}
		& \sum_{n=j+1}^N \sum_{m=i+1}^{j+1} V^\rec_{ijmn} = \FC_{ij} + \sum_{k=i}^{j-1} \sum_{l=i+1}^{k+1} U^\rec_{iklj};
		&\mskip 10mu&\sum_{n=j+1}^N \sum_{m=i+1}^{j+1} L^\rec_{ijmn} = \FC_{ij} + \sum_{k=i}^{j-1} \sum_{l=i+1}^{k+1} K^\rec_{iklj}\\
		& 0 \le \FC_{ij} \le (\FC_{ij})^\upbnd \chi_{ij};
		&&0 \le K^\rec_{iklj} \le (K^\rec_{iklj})^\upbnd (1-\chi_{ij}),\; \llbracket l \rrbracket_{i+1}^{k+1},\; \llbracket k \rrbracket_{i}^{j-1}
	\end{alignedat}\right\}
	\label{eq:cond-mass-bal}\\\noalign{\vskip 0.5em}
	& \text{For } (i,j) \in \mathcal{R}\setminus\{[i,i]\}_{i=2}^{N}:\nonumber\\
	&\quad\left. 
	\begin{alignedat}{2}
		& \sum_{m=1}^{i-1} \sum_{n=i-1}^{j-1} V^\strip_{mnij} = \FR_{ij} + \sum_{k=i}^{j-1} \sum_{l=i+1}^{k+1} U^\strip_{iklj}; \quad 
		&&\sum_{m=1}^{i-1} \sum_{n=i-1}^{j-1} L^\strip_{mnij} = \FR_{ij} + \sum_{k=i}^{j-1} \sum_{l=i+1}^{k+1} K^\strip_{iklj}\\
		& 0 \le \FR_{ij} \le (\FR_{ij})^\upbnd \rho_{ij};\quad
		&&0 \le U^\strip_{iklj} \le (U^\strip_{iklj})^\upbnd (1-\rho_{ij}),\llbracket l \rrbracket_{i+1}^{k+1},\; \llbracket k \rrbracket_{i}^{j-1}.
	\end{alignedat}\right\}
	\label{eq:reb-mass-bal}
\end{eqnal}
We are interested in configurations that either have heat exchangers or thermal couplings, but not both. The last two constraints in \eqref{eq:cond-mass-bal} and \eqref{eq:reb-mass-bal} suppress flows in appropriate arcs if the heat exchangers are absent. The above constraints are written only for heat exchangers associated with mixtures. For heat exchangers associated with pure component products, the vapor and liquid flows are further constrained to produce $\Phi_{i,i} F_i $ and $(1-\Phi_{i,i}) F_i$ of component $i$ in liquid and vapor phases, respectively (see Figure \ref{fig:flow-network}(e)). Mass balances around these heat exchangers are given below.
\begin{eqnal}
	& \text{For } (i,i) \in \mathcal{C}: \quad \left.
	\begin{alignedat}{2}
		& \sum_{n=i+1}^N \sum_{m=i+1}^{i+1} V^\rec_{iimn} = \FC_{ii} + U^\rec_{iiii};\quad &&\mskip10mu \sum_{n=i+1}^N \sum_{m=i+1}^{i+1} L^\rec_{iimn} = \FC_{ii} + K^\rec_{iiii}\\
		& 0 \le U^\rec_{iiii};\;\; 0 \le \FC_{ii} \le (\FC_{ii})^\upbnd \chi_{ii};\quad	&&\mskip10mu (K^\rec_{iiii})^\lobnd \le K^\rec_{iiii} \le (K^\rec_{iiii})^\upbnd (1-\chi_{ii})
	\end{alignedat}
	\right\}, \label{eq:cond-mass-bal-pure}\\\noalign{\vskip 1em}
	& \text{For } (i,i) \in \mathcal{R}: \quad \left.
	\begin{alignedat}{2}
		& \sum_{m=1}^{i-1} \sum_{n=i-1}^{i-1} V^\strip_{mnii} = \FR_{ii} + U^\strip_{iiii};\quad &&\mskip10mu \sum_{m=1}^{i-1} \sum_{n=i-1}^{i-1} L^\strip_{mnii} = \FR_{ii} +  K^\strip_{iiii}\\
		& 0\le K^\strip_{iiii};\;\; 0 \le \FR_{ii} \le (\FR_{ii})^\upbnd \rho_{ii};\quad &&\mskip10mu (U^\strip_{iiii})^\lobnd \le U^\strip_{iiii} \le (U^\strip_{iiii})^\upbnd (1-\rho_{ii})
	\end{alignedat}
	\right\}, \label{eq:reb-mass-bal-pure}\\\noalign{\vskip 1em}
	& \text{For } (i,i) \in \mathcal{C}\cap \mathcal{R}: \quad
	\begin{aligned}
		& U^\rec_{iiii} - U^\strip_{iiii} = F_p (1-\Phi_{i,i});\quad &&\mskip10mu K^\strip_{iiii} - K^\rec_{iiii} = F_p\Phi_{i,i}.
	\end{aligned}
	\label{eq:mass-bal-pure}
\end{eqnal}
where $(\cdot)^\lobnd$ denotes the lower bound on $(\cdot)$. From \eqref{eq:mass-bal-pure} and \eqref{eq:reb-mass-bal-pure} (\resp \eqref{eq:cond-mass-bal-pure}), $(K^\rec_{iiii})^\lobnd = -F_p\Phi_{i,i}$ (\resp $(U^\strip_{iiii})^\lobnd = -F_p(1-\Phi_{i,i})$). For each submixture $[i,j] \in \mathcal{P}$, the net inflow of component $p$ equals the sum of component flows from all the splits that produce $[i,j]$ as distillate or residue. The net inflow is distributed among all splits of $[i,j]$ (see Figure \ref{fig:flow-network}(f)).
\begin{eqnal}
	& \text{For } [i,j] \in \mathcal{S}:  & & \sum_{k=i}^{j-1} \sum_{l=i+1}^{k+1} f^\feed_{ikljp} = \sum_{n=j+1}^N \sum_{m=i+1}^{j+1} f^\rec_{ijmnp} + \sum_{m=1}^{i-1}\sum_{n=i-1}^{j-1} f^\strip_{mnijp}, \quad \llbracket p\rrbracket_i^{j}.
	\label{eq:net-comp-bal}
\end{eqnal}
Finally, modeling the problem in the above manner requires rigorous bounds on all material flows. The net component inflow to and outflow from any column cannot exceed in steady-state the component flow in the process feed. Therefore, the upper bound on all flows of component $p$ is chosen to be $F_p$ i.e., $(f^\feed_{ikljp})^\upbnd = (f^\rec_{ikljp})^\upbnd = (f^\strip_{ikljp})^\upbnd = F_p$. However, although required for deriving rigorous relaxations, there is no simple upper bound on vapor and liquid flows in the columns and heat exchangers. For deriving a bound, we use optimality-based bound tightening, where we find feasible flows for an admissible configuration using the technique of \citet{nallasivam2013}. This technique can also be replaced with a local nonlinear programming solver.
Let this upper bound be $\vapduty^*$. Then, we solve the following linear programs (LP) to derive bounds:
\begin{align}
\label{eq:bounding-LPs}
\begin{aligned}
\max  \quad V^\rec_{iklj},\quad \text{s.t.}\quad \eqref{eq:feedprod}-\eqref{eq:net-comp-bal},\; \sum_{(i,j)\in\mathcal{R}} \FR_{i,j} \le \phi \vapduty^*
\end{aligned}
\end{align}
We choose $\phi = 1$, if only the optimal solution is desired. Since the model does not capture all operability concerns, such as controllability and suitability w.r.t heat integration with the rest of the plant, and vapor flow predictions are based on shortcut methods rather than rigorous simulations, industrial practitioners are often interested in identifying a ranklist of a few best solutions for this MINLP. Such a ranklist allows them to a posteriori incorporate such considerations. Therefore, to allow construction of such a ranklist, we choose $\phi = 1.5$. With this choice, any configuration that consumes at most 50\% more energy than the feasible solution remains in the search space. Our numerical experiments show that each LP can be solved in a fraction of a second using solvers such as Gurobi \citep{gurobi}, and the computational time taken to solve all the LPs for a five-component mixture is typically negligible.

\subsection{Underwood Constraints}
\label{sec:underwoodconst}
As mentioned in \textsection 2, for a given split, there is a minimum threshold vapor requirement in each section of a column, below which the products are not produced with the desired purity. A column can, however, carry more vapor than the threshold, and the excess vapor can, if transferred to other columns, be utilized in those columns. The threshold vapor requirement can be computed using Underwood constraints included below:
\begin{eqnal}
	\allowdisplaybreaks
	\text{For } \; [i,j] \in \mathcal{P}, \;  \llbracket k \rrbracket_{i}^{j-1}, \; \llbracket l \rrbracket_{i+1}^{k+1}:\nonumber&\\
	& \sum_{p=i}^j \frac{\alpha_pf^\feed_{ikljp}}{\alpha_p-\theta_{ijq}} = U_{iklj}^{\rec}\delta_{j<N}-U_{iklj}^{\strip}\delta_{1<i}, \quad\llbracket q \rrbracket_{l-1}^k,
	\label{eq:UWFeedEqn} \\
	& \sum_{p=i}^{k} \frac{\alpha_pf_{ikljp}^{\rec}}{\alpha_p-\theta_{ijq}} \leq \Upsilon_{iklj}^\rec,   \quad \sum_{p=l}^j \frac{\alpha_p f_{ikljp}^{\strip}}{\alpha_p-\theta_{ijq}} \geq -\Upsilon_{iklj}^\strip,   \quad \llbracket q \rrbracket_{l-1}^k,
	\label{eq:UWmvr-strip1}\\
	& \sum\limits_{p=i}^{k} \frac{\alpha_pf_{ikljp}^{\rec}}{\alpha_p-\theta_{ijq}} \geq \Upsilon_{iklj}^\rec, \quad \sum_{p=l}^j \frac{\alpha_p f_{ikljp}^{\strip}}{\alpha_p-\theta_{ijq}} \le -\Upsilon_{iklj}^\strip,  \quad \llbracket q \rrbracket_{l}^{k-1},
	\label{eq:UWmvr-strip2}\\
	& \alpha_{q+1} \le \theta_{ijq} \le \alpha_q,
	\label{eq:UWroot-bounds}\\
	& U^\rec_{iklj<N} - U^\strip_{1<iklj} = \Upsilon^\rec_{iklj} - \Upsilon^\strip_{iklj},
	\label{eq:minvap-mass-bal}\\
	& 0 \le \Upsilon^\rec_{iklj} \le V^\rec_{iklj}, \quad 0 \le \Upsilon^\strip_{iklj} \le V^\strip_{ikljp},
	\label{eq:minvap-less-actual}
\end{eqnal}
where $\Upsilon^\rec_{iklj}$ and $\Upsilon^\strip_{iklj}$ denote the threshold vapor flow in rectifying and stripping sections, respectively. Note that, for the process feed $[1,N]$, $f^\feed_{ikljp}$ and $U_{iklj}^{\rec}\delta_{j<N}-U_{1<iklj}^{\strip}\delta_{1<i}$ in \eqref{eq:UWFeedEqn} and \eqref{eq:minvap-mass-bal} are replaced by $F_p \sigma_{iklj}$ and $\left(\sum_{p=1}^{N}F_p\right)(1-\Phi_{1,N}) \sigma_{i,k,l,j}$, respectively. \eqref{eq:UWFeedEqn} is commonly known in the literature as the \textit{Underwood feed equation}, and it computes \textit{Underwood roots} $\{\theta_{ijq}\}_{q=l-1}^k$, which satisfy $\alpha_{q+1} \le \theta_{ijq} \le \alpha_q$ \citep{underwood}. \eqref{eq:UWmvr-strip1} governs the minimum vapor requirement in rectifying and stripping sections as a function of the distillate and residue compositions. \eqref{eq:UWmvr-strip2} ensures that the minimum vapor constraints are binding for $\{\theta_{ijq}\}_{q=l}^{k-1}$. These constraints are required for the model to have the correct degrees of freedom as described in \citet{tumbalamgooty2019}. \eqref{eq:minvap-mass-bal} models vapor balance at the feed location in terms of minimum vapor flows. \eqref{eq:minvap-less-actual} ensures that the actual vapor in each section is at least as high as the threshold vapor flow.

\begin{remark}
	Since the process feed is always present i.e., $\zeta_{1,N}=1$, and the net component and vapor inflow to columns $Q_{1klN}$ where $1 < l \le k+1 \le N$ are known, we solve the Underwood feed equation \eqref{eq:UWFeedEqn} \textit{a priori} to determine the Underwood roots $\{\theta_{1Nq}\}_{q=1}^{N-1}$, and fix these variables to the calculated values. \qed
	\label{rem:Feed-Bounds}
\end{remark}

\begin{remark}
	Recognizing that $f^\rec_{ikljp} \ge 0$, $\theta_{ijk} \le \alpha_k < \alpha_{k-1} < \dots < \alpha_i$, $f^\strip_{ikljp} \ge 0$ and $\alpha_j < \alpha_{j-1}<\dots < \alpha_l < \theta_{ijl-1}$, we have
	\begin{align}
	\text{for } \; [i,j] \in \mathcal{P}, \;  \llbracket k \rrbracket_{i}^{j-1}, \; \llbracket l \rrbracket_{i+1}^{k+1} \quad
	\left\{
	\begin{aligned}
	& 0 \le \sum_{p=i}^k \frac{\alpha_p f^\rec_{ikljp}}{\alpha_p-\theta_{ijk}}; \quad 0 \le -\sum_{p=l}^j \frac{\alpha_p f^\strip_{ikljp}}{\alpha_p-\theta_{ijl-1}}.
	\end{aligned} \right.
	\label{eq:LBs-minvap-1}
	\end{align}
	Next, using \eqref{eq:LBs-minvap-1}, component mass balance $f^\feed_{ikljp} = f^\rec_{ikljp} \delta_{p \le k} + f^\strip_{ikljp} \delta_{p \ge l}$, and \eqref{eq:UWFeedEqn}, it can be shown that
	\begin{align}
	\text{for } \; [i,j] \in \mathcal{P}, \;  \llbracket k \rrbracket_{i}^{j-1}, \; \llbracket l \rrbracket_{i+1}^{k+1} \quad \left\{
	\begin{aligned}
	& U^\rec_{iklj}\delta_{j<N}-U^\strip_{iklj}\delta_{1<i} \le \sum_{p=i}^k \frac{\alpha_p f^\rec_{ikljp}}{\alpha_p-\theta_{ijl-1}}\\
	& U^\strip_{iklj}\delta_{1<i}-U^\rec_{iklj}\delta_{j<N} \le -\sum_{p=l}^{j} \frac{\alpha_p f^\strip_{ikljp}}{\alpha_p-\theta_{ijk}}.
	\end{aligned}\right.
	\label{eq:LBs-minvap-2}
	\end{align}
	Since the vapor flows are bounded, we have finite upper and lower bounds on all nonlinear expressions in \eqref{eq:UWFeedEqn}--\eqref{eq:UWmvr-strip2}.\qed
\end{remark}

\subsection{Exploiting Monotonicity of Underwood Equations}
These cuts are inspired from \citet{carlberg1989a} and \citet{halvorsen2003b}. 
Although these relations are implicit in the model, they are not implied in the relaxation, when Underwood constraints are relaxed. We refer to \cite{tumbalamgooty2019} for a derivation.

When $[i,j]$ is produced as distillate from one of its top parent $[i,n]$ where $j+1 \le n \le N$ \ie{} $\tau_{i,j,n}=1$, but not produced as residue from any of its bottom parents \ie{} $\beta_{0,i,j}=1$, and the associated condenser $(i,j)$ is absent, then $\theta_{inq}$ lower bounds $\theta_{ijq}$ for $\llbracket q \rrbracket_{i}^{j-1}$. Similarly, when $[i,j]$ is produced as residue from one of its bottom parent $[m,j]$ where $1\le m\le i-1$ \ie{} $\beta_{m,i,j}=1$, but not produced as distillate from any of its top parents \ie{} $\tau_{i,j,N+1}=1$, and the associated reboiler $(i,j)$ is absent, then $\theta_{mjq}$ upper bounds $\theta_{ijq}$ for $\llbracket q \rrbracket_{i}^{j-1}$. These constraints are imposed as follows:
\begin{eqnal}
	\label{eq:Flow-of-roots}
	\text{for } \; [i,j]\in \mathcal{S} \quad \left\{
	\begin{aligned}
		& \theta_{inq} -\theta_{ijq} \le M_q \left[\chi_{i,j} + (1-\tau_{i,j,n}) + (1-\beta_{0,i,j}) \right], \quad \llbracket n \rrbracket_{j+1}^N, \; \llbracket q \rrbracket_{i}^{j-1}\\
		& \theta_{ijq} - \theta_{mjq} \le M_q \left[ \rho_{i,j} + (1-\beta_{m,i,j}) + (1-\tau_{i,j,N+1}) \right], \quad \llbracket m \rrbracket_{1}^{i-1}, \; \llbracket q \rrbracket_{i}^{j-1},
	\end{aligned}\right.
\end{eqnal}
where $M_q = (\alpha_q-\alpha_{q+1})$ corresponds to the upper bound on the difference of Underwood roots (see \eqref{eq:UWroot-bounds}). Numerical examples in \cite{tumbalamgooty2019} illustrate that these cuts help branch \& bound converge faster. Given that our formulation has been developed in a lifted space, we use $\tau$ and $\beta$ variables to give a tighter representation of the constraint in \eqref{eq:Flow-of-roots}. Moreover, if the variables $\psi_{1,m,n,j}$ are not eliminated using Proposition \ref{prop:projection-parent-network}, they can be used to further tighten the above constraints. For example, in the first constraint, $(1-\tau_{i,j,n})+(1-\beta_{0,i,j})$ can be replaced with $(1-\psi_{i,n,0,j})$. This concludes the formulation of MINLP \MINLP.

\section{Relaxation and Solution Procedure}
\label{sec:Relaxation}
Apart from integrality requirements on stream $(\zeta_{i,j})$ and heat exchanger variables ($\rho_{i,j}$ and $\chi_{i,j}$), the remaining source of nonconvexity in the MINLP is the Underwood constraints. In this section, we describe the construction of a convex relaxation of Underwood constraints (\eqref{eq:UWFeedEqn}--\eqref{eq:UWroot-bounds}), referred to hereafter as the relaxation, defined using convex constraints that admits all feasible solutions. One of the challenges in constructing a valid relaxation is that the denominator of certain fractions in Underwood constraints can approach arbitrarily close to zero (see \eqref{eq:UWFeedEqn}--\eqref{eq:UWroot-bounds}). Consequently, off-the-shelf global solvers, such as BARON \citep{tawarmalani2005}, report an error and are not able to solve the problem. The common strategy used in the literature is to add/subtract $\epsilon_\theta$ (typically $10^{-2}-10^{-3}$) from the bounds of $\theta_{ijq}$ to prevent it from approaching either $\alpha_{q+1}$ or $\alpha_{q}$ (see \eqref{eq:UWroot-bounds}). However, this ad-hoc strategy has been adopted without a rigorous proof.  Our numerical experiments suggest that the choice of this $\epsilon_\theta$ is not straightforward, and varies from one instance to another.  In the following, we show that a rigorous relaxation for the fraction can be constructed although the denominator may approach close to zero.



In the following, we drop indices $iklj$. This is because, Underwood equations apply to a column, say $Q_{iklj}$, and these indices are easily gleaned from the column specification or the associated split $[i,k]/[l,j]$. Moreover, for notational convenience, we describe the relaxation using $\mathcal{U}=\{ (f,U,\Upsilon,\theta) \; | \;\eqref{eq:surrogate-UW};\; (f_p^\feed,f_p^\rec,f_p^\strip) \in [0,F_p]^3,\; p={1,2}; \; 0 \le (\cdot) \le (\cdot)^\upbnd, \; \forall \; (\cdot) \in \{U^\rec,U^\strip,\Upsilon^\rec,\Upsilon^\strip\}  \}$, where
\begin{subequations}\label{eq:surrogate-UW}
	\begin{align}
	& \frac{\alpha_1f_1^\feed}{\alpha_1-\theta} - \frac{\alpha_2f_2^\feed}{\theta-\alpha_2} = U^\rec - U^\strip,
	\label{eq:surrogate-UW-feedeq}\\
	& E^\rec \le \frac{\alpha_1f_1^\rec}{\alpha_1-\theta} - \frac{\alpha_2f_2^\rec}{\theta-\alpha_2} \le \Upsilon^\rec,
	\label{eq:surrogate-UW-receq}\\
	& E^\strip \le -\frac{\alpha_1f_1^\strip}{\alpha_1-\theta} + \frac{\alpha_2f_2^\strip}{\theta-\alpha_2} \le \Upsilon^\strip,
	\label{eq:surrogate-UW-stripeq}\\
	& \alpha_2 \le \theta^\lobnd \le \theta \le \theta^\upbnd \le \alpha_1
	\label{eq:surrogate-UW-Thbnds},\\
	& U^\rec-U^\strip = \Upsilon^\rec - \Upsilon^\strip,
	\label{eq:surrogate-UW-vapbal}\\
	& f_p^\feed =f_p^\rec + f_p^\strip, \quad p = {1,2}.
	\label{eq:surrogate-UW-flowbal}
	\end{align}
\end{subequations}
Here, we assume that column $Q_{iklj}$ performs the split of a binary mixture. Observe that \eqref{eq:surrogate-UW-feedeq}, the second inequality in \eqref{eq:surrogate-UW-receq} and \eqref{eq:surrogate-UW-stripeq} are simplified versions of \eqref{eq:UWFeedEqn} and \eqref{eq:UWmvr-strip1} for binary mixtures. We ensure that all fractions are non-negative by factoring out a negative sign from the fractions whose denominator is negative (see \eqref{eq:surrogate-UW}). Next, $E^\rec$ and $E^\strip$ denote lower bounds on nonlinear expressions in \eqref{eq:surrogate-UW-receq} and \eqref{eq:surrogate-UW-stripeq}, respectively. We choose $E^\rec$ (\resp $E^\strip$) to be $\Upsilon^\rec$ (\resp $\Upsilon^\strip$) if the second inequality in \eqref{eq:surrogate-UW-receq} (\resp \eqref{eq:surrogate-UW-stripeq}) needs to be binding, as in \eqref{eq:UWmvr-strip2}. Else, we choose the lower bound derived in \eqref{eq:LBs-minvap-1} and \eqref{eq:LBs-minvap-2}. \eqref{eq:surrogate-UW-Thbnds}, \eqref{eq:surrogate-UW-vapbal}, and \eqref{eq:surrogate-UW-flowbal} correspond to \eqref{eq:UWroot-bounds}, \eqref{eq:minvap-mass-bal}, and \eqref{eq:comp-vap-bal-on-S}, repectively. Lastly, we remark that, in \eqref{eq:UWmvr-strip1} and \eqref{eq:UWmvr-strip2}, $f_2^\rec=f_1^\strip=0$ for a split of a binary mixture. Since our purpose in restricting to the binary case is to illustrate the mathematical structure of relaxations, we do not consider this restriction. In general splits, one or more components may distribute between the distillate and residue. 

The first step in standard approaches to relax $\mathcal{U}$ is to linearize Underwood constraints by introducing an auxiliary variable representing the graph of each fraction. Then, the restriction that this variable take the value of the fraction is replaced with the less stringent restriction that the variable lies in a convex set containing the graph of fraction. Instead, we reformulate $\mathcal{U}$ as described in \textsection \ref{sec:reformulation} before linearizing the Underwood constraints.

\subsection{Reformulation}
\label{sec:reformulation}
We adapt classical Reformulation-Linearization Technique (RLT) \citep{sherali1992new} to fractions, and reformulate $\mathcal{U}$ by appending RLT cuts derived using Underwood constraints. For clarity, we present the derivation of RLT cuts with Underwood minimum vapor constraint in the rectifying section (second inequality in \eqref{eq:surrogate-UW-receq}), and describe the entire reformulated set towards the end. We multiply each Underwood constraint with the bound factors of $\theta$, $(\theta-\theta^\lobnd)$, and $(\theta^\upbnd-\theta)$. A naive approach would then disaggregate the product, leading to
\begin{subequations}
	\label{eq:naive-RLT}
	\begin{align}
	& \frac{\alpha_1 f_1^\rec \theta}{\alpha_1-\theta} - \frac{\alpha_1 f_1^\rec \theta^\lobnd}{\alpha_1-\theta} - \frac{\alpha_2 f_2^\rec \theta}{\theta - \alpha_2} + \frac{\alpha_2 f_2^\rec \theta^\lobnd}{\theta - \alpha_2} \le \Upsilon^\rec\cdot\theta - \Upsilon^\rec \cdot \theta^\lobnd,
	\label{eq:naive-RLT-1}\\
	& \frac{\alpha_1 f_1^\rec \theta^\upbnd}{\alpha_1-\theta} - \frac{\alpha_1 f_1^\rec \theta}{\alpha_1-\theta} - \frac{\alpha_2 f_2^\rec \theta^\upbnd}{\theta - \alpha_2} + \frac{\alpha_2 f_2^\rec \theta}{\theta - \alpha_2} \le \Upsilon^\rec\cdot\theta^\upbnd - \Upsilon^\rec \cdot \theta,
	\label{eq:naive-RLT-2}
	\end{align}
\end{subequations}
following which auxiliary variables are introduced to linearize each nonlinear term: $H_p^\rec = f_p^\rec/|\alpha_p-\theta|$, $\underline{H\theta}_p^\rec = f_p^\rec\theta/|\alpha_p-\theta|$, for $p=1,2$, and $\underline{\Upsilon\theta}^\rec = \Upsilon^\rec \cdot \theta$. Here, and in the rest of the article, the variables introduced to linearize a product will be written by underlining the concatenation of symbols, as in $\underline{\Upsilon\theta}^\rec = \Upsilon^\rec \cdot \theta$. Instead, we use polynomial long division prior to linearization, which transforms \eqref{eq:naive-RLT} to
\begin{subequations}
	\label{eq:RDLT}
	\begin{align}
	& \phantom{-}\frac{\alpha_1(\alpha_1-\theta^\lobnd)f_1^\rec}{\alpha_1-\theta} -\alpha_1f_1^\rec + \frac{\alpha_2(\theta^\lobnd-\alpha_2)f_2^\rec}{\theta-\alpha_2}  -\alpha_2f_2^\rec \le \Upsilon^\rec \cdot \theta - \Upsilon^\rec\cdot \theta^\lobnd,
	\label{eq:RDLT-1}\\
	& -\frac{\alpha_1(\alpha_1-\theta^\upbnd)f_1^\rec}{\alpha_1-\theta} +\alpha_1f_1^\rec - \frac{\alpha_2(\theta^\upbnd-\alpha_2)f_2^\rec}{\theta-\alpha_2}  +\alpha_2f_2^\rec \le \Upsilon^\rec\cdot \theta^\upbnd - \Upsilon^\rec \cdot \theta.
	\label{eq:RDLT-2}
	\end{align}
\end{subequations}
Next, we introduce auxiliary variables to linearize nonlinear terms: $H_p^\rec = f_p^\rec/|\alpha_p-\theta|$, for $p=1,2$, and $\underline{\Upsilon\theta}^\rec = \Upsilon^\rec \cdot \theta$. We shall refer to the proposed variant as the \textit{Reformulation-Division-Linearization Technique} (RDLT) of fractional terms, in order to easily distinguish and emphasize the use of polynomial division as an intermediate step. Clearly, RDLT cuts require fewer variables than those derived by naive application of RLT as described above. In addition, RDLT cuts lead to a tighter relaxation of $\mathcal{U}$, which we demonstrate below.

\begin{proposition}
	\label{prop:std-RLT-RDT}
	Let $\mathcal{B} = [f_1^\lobnd,f_1^\upbnd] \times [f_2^\lobnd,f_2^\upbnd] \times [\Upsilon^\lobnd,\Upsilon^\upbnd] \times [\theta^\lobnd, \theta^\upbnd]$, and $S=\{ (f,\Upsilon,\theta) \in \mathcal{B} \mid \frac{\alpha_1f_1}{\alpha_1-\theta} - \frac{\alpha_2f_2}{\theta-\alpha_2} \le \Upsilon \}$. 	Let $\underline{\Upsilon\theta}$, $H_i$, $\underline{H\theta}_i$ be linearizations of $\Upsilon \cdot \theta$, $\frac{f_i}{|\alpha_i-\theta|}$, and $\frac{f_i\theta}{|\alpha_i-\theta|}$ respectively.
	Define $S_{\textnormal{std}}= \{(f,\Upsilon,\theta,H)\in\mathcal{B}\times \mathbb{R}^2 \mid  \alpha_1 H_1 - \alpha_2 H_2 \le \Upsilon, \; \widecheck{H}_i \le H_i \le \widehat{H}_i,\; i=1,2  \}$, $S_\textnormal{RLT} = \bigl\{ (f,\Upsilon,\theta,H,\underline{H\theta},\underline{\Upsilon\theta}) \in C \bigm| \eqref{eq:SRLT} \bigr\}$, where $C\subseteq \mathcal{B}\times\mathbb{R}^{5}$ and
	\begin{subequations}\label{eq:SRLT}
		\begin{align}
		& \alpha_1 (\underline{H\theta}_1-\theta^\lobnd H_1) - \alpha_2 (\underline{H\theta}_2 - \theta^\lobnd H_2) \le \underline{\Upsilon\theta} - \Upsilon\cdot \theta^\lobnd,
		\label{eq:SRLT-1}\\
		& \alpha_1 (\theta^\upbnd H_1 - \underline{H\theta}_1) - \alpha_2 (\theta^\upbnd H_2 - \underline{H\theta}_2) \le \Upsilon\cdot \theta^\upbnd - \underline{\Upsilon\theta}.
		\label{eq:SRLT-2}
		\end{align}
	\end{subequations}
	Let $S_\textnormal{RDLT}=\bigl\{ (f,\Upsilon,\theta,H,\underline{\Upsilon\theta}) \in C' \bigm| \eqref{eq:SRDLT} \bigr\}$, where $C'\subseteq \mathcal{B}\times\mathbb{R}^{5}$ and
	\begin{subequations}
		\label{eq:SRDLT}
		\begin{align}
		& \alpha_1 (\alpha_1 H_1-f_1-\theta^\lobnd H_1) - \alpha_2 (\alpha_2H_2 +f_2 -\theta^\lobnd H_2) \le \underline{\Upsilon\theta} - \Upsilon\cdot \theta^\lobnd,
		\label{eq:SRDLT-1}\\
		& \alpha_1 (\theta^\upbnd H_1 - \alpha_1H_1 +f_1) - \alpha_2 (\theta^\upbnd H_2 -f_2-\alpha H_2) \le \Upsilon\cdot \theta^\upbnd - \underline{\Upsilon\theta}.
		\label{eq:SRDLT-2}
		\end{align}
	\end{subequations}
	Assume that
	$C
	\supseteq
	\bigl\{(f, \Upsilon, \theta, H, \underline{H\theta}, \underline{\Upsilon\theta})
	\bigm|
	(f,\Upsilon,\theta,H, \underline{\Upsilon\theta})\in C',\;
	\underline{H\theta}_1 = \alpha_1H_1 - f_1,\;
	\underline{H\theta}_2 = \alpha_2H_2 + f_2
	\bigr\}
	$ and
	$\proj_{H_1,H_2} C
	\subseteq
	[\widecheck{H}_1,\widehat{H}_1]\times[\widecheck{H}_2,\widehat{H}_2]$.
	Then, $S_\textnormal{std} \supseteq \textnormal{proj}_{(f,\Upsilon,\theta,H)}(S_\textnormal{RLT})$ and $S_\textnormal{RLT} \supseteq \bigl\{ (f,\Upsilon,\theta,H,\underline{H\theta},\underline{\Upsilon\theta}) \in S_\textnormal{RDLT}\times\mathbb{R}^{2} \mid \underline{H\theta}_1 = \alpha_1H_1 - f_1, \underline{H\theta}_2 = \alpha_2H_2 + f_2\bigr\}$, where the right hand side is an affine lifting of $S_\textnormal{RDLT}$. 
\end{proposition}

\begin{proof}
The first part of the statement follows easily because $\alpha_1H_1 - \alpha_2H_2 \le \Upsilon$ is obtained by adding
	\eqref{eq:SRLT-1} and \eqref{eq:SRLT-2}, and the bounds on $H_i$ in $S_{\textnormal{std}}$ are implied by our assumption $\proj_{H_1,H_2} C
	\subseteq
	[\widecheck{H}_1,\widehat{H}_1]\times[\widecheck{H}_2,\widehat{H}_2]$.  The second part follows similarly because \eqref{eq:SRLT-1} is derived by adding \eqref{eq:SRDLT-1} with $\alpha_1(\underline{H\theta}_1 - \alpha_1 H_1 + f_1)=0$ and $\alpha_2(\underline{H\theta}_2 - \alpha_2 H_2 - f_2)=0$, and affine lifting of any point in $C'$ that satisfies this equation is assumed to be contained in $C$. 
\end{proof}

The sets $C$ and $C'$ in Proposition \ref{prop:std-RLT-RDT} are typically created by relaxing the nonlinear expressions. We illustrate, via an example, that the relations in Proposition \ref{prop:std-RLT-RDT} can be strict. 

\begin{example}
	\label{example:RLT}
	Let, $\alpha_1=15,\; \alpha_2 =9$, $f_1^\lobnd=f_1^\upbnd=0.6$, $f_2^\lobnd=f_2^\upbnd=0.4$, $\Upsilon^\lobnd = -10$, $\Upsilon^\upbnd = 10$, $\theta^\lobnd = 9.1$, $\theta^\upbnd = 14.9$. The sets $C$ and $C'$ are constructed by under- and over-estimating the nonlinear terms with their respective convex and concave envelopes. Figure \ref{fig:UWPlots}(a) depicts the projection of sets $S, \; S_\textnormal{std},\; S_\textnormal{RLT}$ and $S_\textnormal{RDLT}$ in $\Upsilon-\theta$ space. It is clear that $\textnormal{proj}_{(f,\Upsilon,\theta)}(S_\textnormal{std}) \supset \textnormal{proj}_{(f,\Upsilon,\theta)}(S_\textnormal{RLT}) \supset \textnormal{proj}_{(f,\Upsilon,\theta)}(S_\textnormal{RDLT}) \supset S$. 	Besides improving the quality of relaxation by introducing fewer auxiliary variables, RDLT has another benefit in our context that we describe next.
	
	\begin{figure}[h]
		\centering
		\subfigure[]{\includegraphics[scale=0.55]{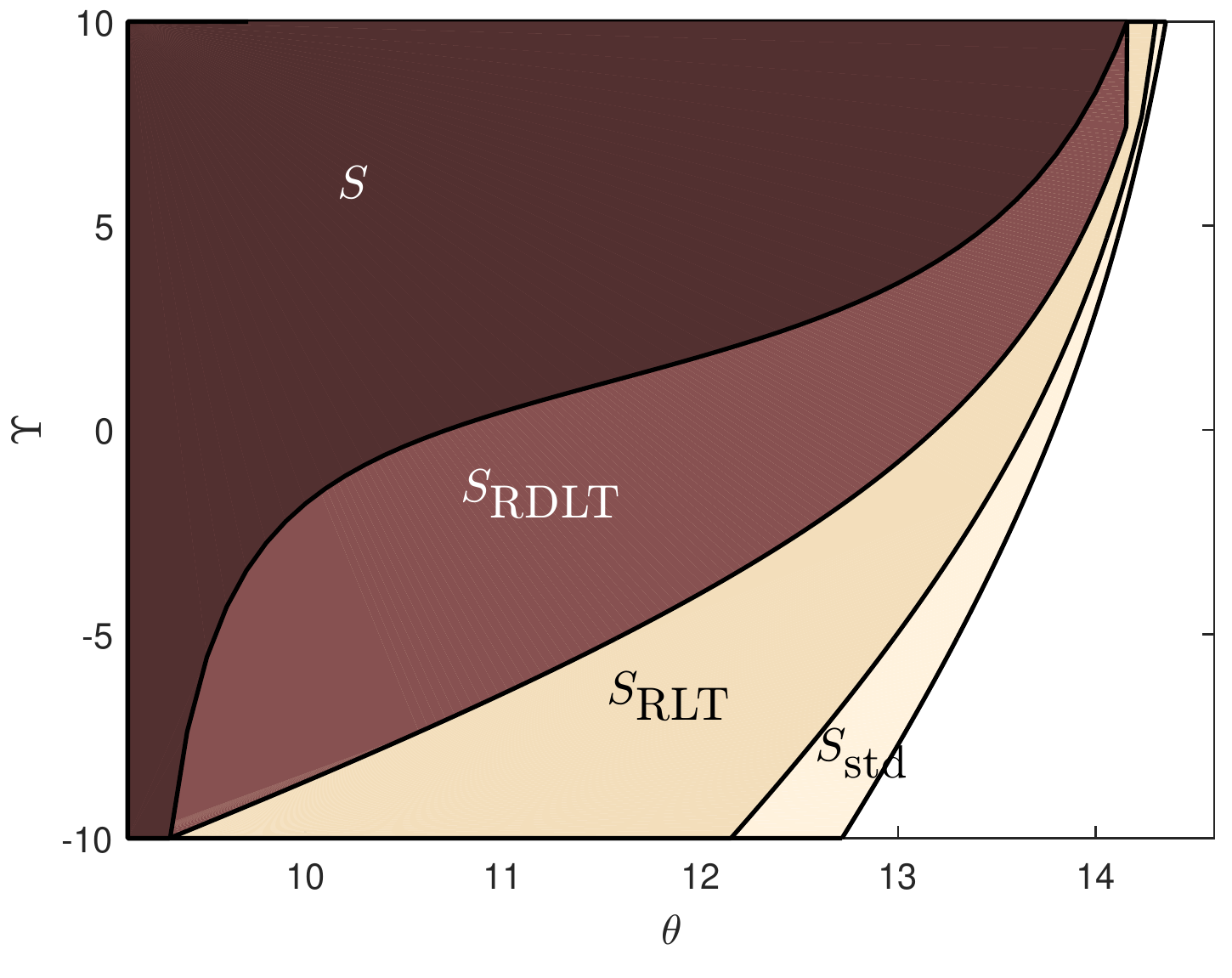}}
		\subfigure[]{\includegraphics[scale=0.55]{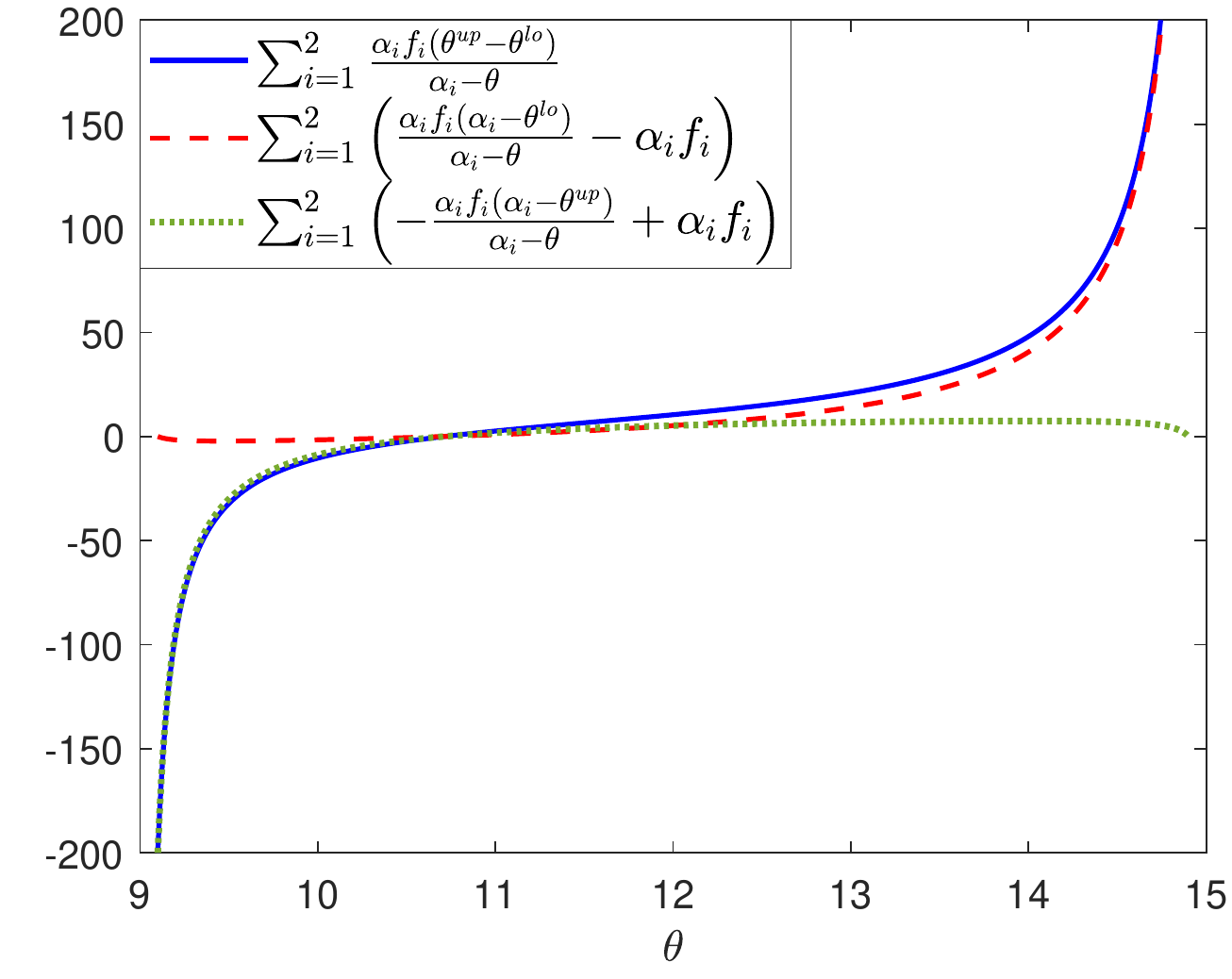}}
		\caption{(a) Projection of sets $S, \; S_\textnormal{std},\; S_\textnormal{RLT}$ and $S_\textnormal{RDLT}$ in Example \ref{example:RLT} in $\Upsilon-\theta$ space. (b) Plots of nonlinear expression in Underwood constraint }
		\label{fig:UWPlots}
	\end{figure}

	Even when $f_1$ and $f_2$ are fixed, the function $\frac{\alpha_1f_1}{\alpha_1-\theta} - \frac{\alpha_2f_2}{\theta-\alpha_2}$ is nonconvex (see Figure \ref{fig:UWPlots}(b)), because it is a difference of two convex functions. When this function is multiplied by $(\theta-\theta^\lobnd)$ (\resp ($\theta^\upbnd-\theta$)), it becomes convex (\resp concave) (see \ref{fig:UWPlots}(b)). In the naive RLT approach, where each fraction is relaxed independently, the product $(f_1/(\alpha_1-\theta)) \cdot (\theta-\theta^\lobnd)$ is disaggregated and relaxed as a difference of the convex envelope of $f_1\theta/(\alpha_1-\theta)$ with the concave envelope of $f_1/(\alpha_1-\theta)$. Whereas, the polynomial division step makes the convexity apparent revealing better ways to construct the relaxation. \qed
\end{example}

We use RDLT to obtain a reformulation of $\mathcal{U}$, denoted as $\mathcal{U}_\text{ref}$, in higher dimensional space as $\mathcal{U}_\text{ref}=\{ (f,U,\Upsilon,\theta,H,\underline{U\theta},\underline{\Upsilon\theta}) \;|\; \eqref{eq:reformulated-UW}; (f_p,\theta,H_p) \in \mathcal{F}_p, \; p=1,2; \; (U,\Upsilon,\theta,\underline{U\theta},\underline{\Upsilon\theta}) \in \mathcal{V}   \}$, where
\begin{subequations}
	\begin{align}
	& \sum\nolimits_{p=1}^2 \left(\alpha_p|\alpha_p-\theta^\lobnd|H_p^\feed -\alpha_pf_p^\feed\right) = (\underline{U\theta}^\rec - \theta^\lobnd U^\rec) - (\underline{U\theta}^\strip - \theta^\lobnd U^\strip),\\
	& \sum\nolimits_{p=1}^2 \left(\alpha_pf_p^\feed-\alpha_p|\alpha_p-\theta^\upbnd|H_p^\feed\right) = (\theta^\upbnd U^\rec-\underline{U\theta}^\rec) - (\theta^\upbnd U^\strip-\underline{U\theta}^\strip),\\
	& E^\rec (\theta-\theta^\lobnd) \le \sum\nolimits_{p=1}^2 \left(\alpha_p|\alpha_p-\theta^\lobnd|H_p^\rec -\alpha_pf_p^\rec\right) \le \underline{\Upsilon\theta}^\rec - \theta^\lobnd \Upsilon^\rec,
	\label{eq:reformulated-UW-3}\\
	& E^\rec(\theta^\upbnd-\theta) \le \sum\nolimits_{p=1}^2 \left(\alpha_pf_p^\rec-\alpha_p|\alpha_p-\theta^\upbnd|H_p^\rec\right) \le \theta^\upbnd \Upsilon^\rec -\underline{\Upsilon\theta}^\rec,
	\label{eq:reformulated-UW-4}\\
	& E^\strip(\theta-\theta^\lobnd) \le \sum\nolimits_{p=1}^2 \left(\alpha_pf_p^\strip-\alpha_p|\alpha_p-\theta^\upbnd|H_p^\strip\right) \le \underline{\Upsilon\theta}^\strip - \theta^\lobnd \Upsilon^\strip,\\
	& E^\strip(\theta^\upbnd-\theta) \le \sum\nolimits_{p=1}^2 \left(\alpha_p|\alpha_p-\theta^\lobnd|H_p^\strip -\alpha_pf_p^\strip\right) \le \theta^\upbnd \Upsilon^\strip -\underline{\Upsilon\theta}^\strip.
	\end{align}
	\label{eq:reformulated-UW}
\end{subequations}
In the above, $|\cdot|$ denotes absolute value function, and the sets $\mathcal{F}_p$, $p=1,2$, and $\mathcal{V}$ are defined as
\begin{align}
& \mathcal{F}_p = \left\{ (f_p,\theta,H_p) \,\middle|\,\begin{aligned}
& H_p^\feed = f_p^\feed\cdot T_p(\theta), \; H_p^\rec = f_p^\rec \cdot T_p(\theta), \; H_p^\strip = f_p^\strip \cdot T_p(\theta)\\
&  f_p^\feed=f_p^\rec+f_p^\strip\\
& (f_p^\feed,f_p^\rec,f_p^\strip) \in [0,F_p]^3 , \; \theta^\lobnd \le \theta \le \theta^\upbnd
\end{aligned} \right\},
\label{eq:set-Fp-def}\\
\noalign{and}
& \mathcal{V} = \left\{ (U,\Upsilon,\theta,\underline{U\theta},\underline{\Upsilon\theta}) \,\middle|\,\begin{aligned}
& \underline{U\theta}^\rec = U^\rec \cdot \theta, \; \underline{U\theta}^\strip = U^\strip \cdot \theta \\
& \underline{\Upsilon\theta}^\rec = \Upsilon^\rec \cdot \theta, \; \underline{\Upsilon\theta}^\strip = \Upsilon^\strip \cdot \theta\\
& U^\rec -U^\strip = \Upsilon^\rec - \Upsilon^\strip\\
& \theta^\lobnd \le \theta \le \theta^\upbnd\\
& 0 \le U^\rec \le (U^\rec)^\upbnd, \; 0 \le U^\strip \le (U^\strip)^\upbnd\\
& 0 \le \Upsilon^\rec \le (\Upsilon^\rec)^\upbnd, \; 0 \le \Upsilon^\strip \le (\Upsilon^\strip)^\upbnd
\end{aligned} \right\},
\label{eq:set-V-def}
\end{align}
where $T_1(\theta) = 1/(\alpha_1-\theta)$, and $T_2(\theta)=1/(\theta-\alpha_2)$.

\subsubsection{Generalizations}
\label{sec:RDLT-generalization}
We remark that RDLT can be used for problems with constraints that have the form $\sum_{i=1}^r \frac{x_i g_i(y)}{h_i(y)} \le x_0$, $\bigl\{g_i(y)\bigr\}_{i=1}^r$ and $\bigl\{h_i(y)\bigr\}_{i=1}^r$ are some polynomials of $y$. We follow the steps below to derive RDLT cuts.
\begin{enumerate}
	\item We multiply the constraint by some ratio of polynomials of $y$, $n(y)/d(y)$, such that the sign of the ratio does not change over the domain of $y$. Here, we assume, w.l.o.g, that $n(y)/d(y) \ge 0$ over the domain of $y$.
	\item We use polynomial long division to express each $\frac{g_i(y)\cdot n(y)}{h_i(y) \cdot d(y)} = m_i(y) + \frac{k_i(y)}{l_i(y)}$
	such that $\polydeg(k_i) < \polydeg(l_i)$, where $\polydeg(k_i)$ denotes degree of polynomial $k_i(y)$.
	\item We factorize $l_i(y)$ and express it as a product of polynomials $\{q_{ij}(y)\}_{j=1}^{s_i}$ that are non-factorizable over real numbers (e.g., $y+2$ or $y^2+y+1$).
	\item We use the general theorem of partial fraction decomposition to express each fraction $k_i(y)/l_i(y)$ as $\sum_{j=1}^{s_i} p_{ij}(y)/q_{ij}(y)$, where $\polydeg(p_{ij}) < \polydeg(q_{ij})$. This transforms the constraint to $\sum_{i=1}^r\left( x_i\cdot m_i(y) + \sum_{j=1}^{s_i} x_i \cdot p_{ij}(y)/q_{ij}(y) \right) \le x_0 \cdot n(y)/d(y)$.
	\item We linearize the constraint by introducing auxiliary variables for each nonlinear term.
\end{enumerate}
The reformulation described earlier is a specific case, where we chose to multiply each Underwood constraint by $(\theta-\theta^\lobnd)$ and $(\theta^\upbnd-\theta)$.
By changing the factor used in the reformulation step, we can derive alternative RDLT cuts by following the steps described above. As an illustration, we derive two types of additional RDLT cuts for reformulation of $\mathcal{U}$. While we do not use these cuts for our extensive computational experiments, we demonstrate with numerical examples in \textsection \ref{sec:Scenarios} that they further improve the relaxation for some instances.

\noindent\textbf{RDLT cuts with quadratic polynomials}: Here, we choose the product of bound factors of $\theta$, viz. $(\theta-\theta^\lobnd)^2$, $(\theta-\theta^\lobnd)\cdot (\theta^\upbnd-\theta)$ and $(\theta^\upbnd-\theta)^2$,  for reformulation.
As an illustration, we derive the RDLT cut by multiplying the second inequality in \eqref{eq:surrogate-UW-receq} with $(\theta-\theta^\lobnd)\cdot (\theta^\upbnd-\theta)$. The remaining RDLT cuts are derived in a similar fashion. Steps 1 and 2 lead to
\begin{align}
\sum_{p=1}^2 \left( \alpha_p f_p^\rec\cdot (\theta + \alpha_p-\theta^\lobnd-\theta^\upbnd)-\frac{\alpha_p(\alpha_p-\theta^\lobnd)(\alpha_p-\theta^\upbnd)f_p^\rec}{\alpha_p-\theta}  \right) \le \Upsilon^\rec\cdot (\theta-\theta^\lobnd)\cdot (\theta^\upbnd-\theta).
\label{eq:RDLT-quadratic}
\end{align}
Since \eqref{eq:RDLT-quadratic} is already in the form attained in Step 4, we do not need Steps 3 and 4. Finally, we disaggregate the products of $f_p^\rec$ and $\Upsilon^\rec$ with polynomials of $\theta$, and linearize \eqref{eq:RDLT-quadratic} by introducing auxiliary variables for $f_p^\rec/(\alpha_p-\theta)$, $f_p^\rec \cdot \theta$, $\Upsilon^\rec \cdot \theta^2$ and $\Upsilon \cdot \theta$.

\noindent \textbf{RDLT cuts with inverse bound factors}: Here, we use \textit{inverse bound factors} $\left(\frac{1}{\theta}-\frac{1}{\theta^\upbnd}\right)$ and $\left( \frac{1}{\theta^\lobnd}-\frac{1}{\theta} \right)$ for reformulation. Since $\left(\frac{1}{\theta}-\frac{1}{\theta^\upbnd}\right) = \frac{\theta^\upbnd-\theta}{\theta^\upbnd\cdot \theta}$, inverse bound factors are essentially ratios of first-degree polynomial to another first-degree polynomial. As before, for illustration, we derive the RDLT cut obtained by multiplying the second inequality in \eqref{eq:surrogate-UW-receq} with $\left(\frac{1}{\theta}-\frac{1}{\theta^\upbnd}\right)$. The remaining RDLT cuts are obtained in a similar fashion. Step 1 leads to $\sum_{p=1}^2 \frac{\alpha_p f_p^\rec}{(\alpha_p-\theta)\theta}-\frac{\alpha_p f_p^\rec}{(\alpha_p-\theta)\theta^\upbnd} \le \frac{\Upsilon^\rec}{\theta}-\frac{\Upsilon^\rec}{\theta^\upbnd}$, which is already in the form described in Step 2. Further, the denominator of each fraction is already expressed as product of non-factorizable polynomials. Next, we use partial fraction decomposition (Step 4) to obtain
\begin{align}
\sum_{p=1}^2 \left( \frac{f_p^\rec}{\theta} - \frac{(\alpha_p-\theta^\upbnd)f_p^\rec}{\theta^\upbnd(\alpha_p-\theta)} \right) \le \frac{\Upsilon^\rec}{\theta} - \frac{\Upsilon^\rec}{\theta^\upbnd}.
\label{eq:RDLT-inverse-bf}
\end{align}
Finally, we linearize \eqref{eq:RDLT-inverse-bf} by introducing auxiliary variables for $f_p^\rec/|\alpha_p-\theta|$, $f_p^\rec/\theta$ and $\Upsilon^\rec/\theta$.

\subsection{Relaxation for $\alpha_2<\theta^\lobnd$ and $\theta^\upbnd <\alpha_1$}
\label{sec:part-relaxation}
The nonconvexity in $\mathcal{U}_\text{ref}$ is
due to $\mathcal{F}_1$, $\mathcal{F}_2$, and $\mathcal{V}$. We convexify these sets to construct a convex relaxation of $\mathcal{U}_\text{ref}$. However, we first assume that $\alpha_2 < \theta^\lobnd$ and $\theta^\upbnd < \alpha_1$, and relax this assumption later in \textsection \ref{sec:full-relaxation}. This assumption prevents the denominator of fractions in $\mathcal{F}_1$ and $\mathcal{F}_2$ from becoming zero. This discussion is needed for two reasons: (i) it will guide us in deriving additional valid cuts needed to strengthen the relaxation when $\theta = \alpha_2$ and $\theta =\alpha_1$ are admissible (ii) it is needed to construct a piecewise relaxation in \textsection \ref{sec:Discretization}, where we discretize the domain of $\theta$ such that every partition excluding the extreme partitions satisfy $\alpha_2 < \theta^\lobnd \le \theta \le \theta^\upbnd < \alpha_1$.

The standard approach to create a relaxation is to replace each equality $H_p = f_p \cdot T_p(\theta)$ in $\mathcal{F}_p$ (\resp $\underline{\Upsilon\theta} = \Upsilon \cdot \theta$ in $\mathcal{V}$) with a less stringent restriction that $H_p$ (\resp $\underline{\Upsilon\theta}$) lies in the convex hull of $f_p\cdot T_p(\theta)$ (\resp $\Upsilon \cdot \theta$) over a rectangle defined by the ranges of $f_p$ (\resp $\Upsilon$) and $\theta$. However, this approach does not take advantage of the fact that the component (\resp vapor) flows are constrained by mass balances (see \eqref{eq:surrogate-UW-vapbal},\eqref{eq:surrogate-UW-flowbal}) and, thus, results in a weaker relaxation. Instead, we use Proposition \ref{prop:poly-simul-hull}, which describes the construction of simultaneous hull of multiple nonlinear terms over a polytope (not necessarily a hyperrectangle), to construct a tighter relaxation of $\mathcal{U}_\text{ref}$.
\begin{proposition}
	\label{prop:poly-simul-hull}
	Let $X = \{x \in \mathbb{R}^n_+  \mid Bx \le b \}$ be a polytope, $g(y)$ be continuous and convex for $y \in [y^\lobnd,y^\upbnd] \subset \mathbb{R}$, $\mathcal{D}= X \times [y^\lobnd,y^\upbnd] \times \mathbb{R}^{n+n}$, and $S = \{\pi \in  \mathcal{D} \mid z_j = x_j \cdot g(y),\; \underline{xy}_j = x_j \cdot y, \; \llbracket j \rrbracket_1^n   \}$, where $\pi=(x,y,z,\underline{xy})$ denotes an element of $S$. Then, $\Conv(S) = \textnormal{proj}_{\pi}\{ (\pi,y^1,\dots,y^m,w^1,\dots,w^m,w,\lambda^1,\dots,\lambda^m) \mid \eqref{eq:prop:simultaneous-hull-gy} \}$, where
	\begin{subequations}\label{eq:prop:simultaneous-hull-gy}
		\begin{align}
		& w^i \ge g^*(\lambda^i,y^i), \hspace{5.5cm} i=1,\dots,m
		\label{eq:prop:simultaneous-hull-gy-1}\\
		& w^i \le \lambda^i g(y^\lobnd) + \left(\frac{g(y^\upbnd)-g(y^\lobnd)}{y^\upbnd-y^\lobnd}\right)(y^i-\lambda^i y^\lobnd), \quad i = 1,\dots,m
		\label{eq:prop:simultaneous-hull-gy-2}\\
		& \lambda^i y^\lobnd \le y^i \le \lambda^i y^\upbnd, \hspace{5cm}  i = 1,\dots,m \label{eq:prop:simultaneous-hull-gy-3}\\
		& z = \sum\nolimits_{i=1}^m v^i w^i, \quad \underline{xy} = \sum\nolimits_{i=1}^m v^i y^i, \quad w = \sum\nolimits_{i=1}^m w^i,  \label{eq:prop:simultaneous-hull-gy-4} \\
		& y = \sum\nolimits_{i=1}^m y^i, \quad x = \sum\nolimits_{i=1}^m \lambda^i v^i, \quad (\lambda^1,\dots,\lambda^m) \in \Delta^m.
		\label{eq:prop:simultaneous-hull-gy-5}
		\end{align}
	\end{subequations}
	Here, $\textnormal{proj}_\pi\{\cdot\}$ represents projection of $\{\cdot\}$ onto the space of $(x,y,z,\underline{xy})$ variables, $\{v^i\}_{i=1}^m$ are the extreme points of $X$, $\Delta^m = \{(\lambda^1,\dots,\lambda^m) \in \mathbb{R}^m_+ \; | \; \sum_{i=1}^m \lambda^i =1 \}$, and positively homogeneous function $g^*(\lambda^*,y^*)$ related to $g(y) : [y^\lobnd,y^\upbnd] \rightarrow \mathbb{R}$ is defined as:
	\begin{align}
	g^*(\lambda^*,y^*) = \begin{cases}
	\lambda^* g((\lambda^*)^{-1}y^*), & \textnormal{if } (\lambda^*)^{-1}y^* \in [y^\lobnd,y^\upbnd], \; \lambda^* > 0\\
	0, & \textnormal{if } \lambda^* = 0, \; y^* = 0.
	\end{cases}
	\label{eq:perspective-fn}
	\end{align}
\end{proposition}
\begin{proof}
Since $S$ is compact, its convex hull is compact and, by Krein-Milman theorem, is the convex hull of its extreme points. Therefore, we determine the extreme points of $S$, and take their convex hull to obtain $\Conv(S)$. When $y$ is restricted to $\overline{y}\in [y^\lobnd,y^\upbnd]$, the set $S=\{(x,y,z,\underline{xy}) \mid z = g(\overline{y}) \; x, \; \underline{xy} = \overline{y}\; x, \; x\in X, \; y=\overline{y} \}$ can be expressed as an affine transform of $X$. Thus, the extreme points of $S$ project to the set of extreme points of $X$ and we may restrict attention to these points in order to construct $\Conv(S)$. Let $S^i$, for $i=1,\dots,m$, denote the set $S$ where $x$ is restricted to $v^i$ \ie{} $S^i = \{(x,y,z,\underline{xy}) \mid z=v^i\; g(y), \; \underline{xy} = v^i \; y, \; x=v^i, \; y\in[y^\lobnd,y^\upbnd] \}$. Then, $\Conv(S)$ is given as the convex hull of disjunctive union of $S^i$, $i=1,\dots,m$, \ie{} $\Conv(S) = \Conv(S^1\cup \dots \cup S^m) = \Conv(\Conv(S^1)\cup \dots \cup \Conv(S^m))$.
	
	To determine $\Conv(S^i)$, we reformulate each $S^i$ as $S^i = \{(x,y,z,\underline{xy},w) \mid z=v^i \; w, \; \underline{xy} = v^i \; y, \; w = g(y), \; x = v^i, \; y^\lobnd \le y \le y^\upbnd \}$, which is an affine transform of the set $\{(y,w)\in[y^\lobnd,y^\upbnd] \times \mathbb{R} \; | \; w=g(y) \}$. This implies that it suffices to convexify the latter set to obtain $\Conv(S^i)=\text{proj}_\pi\{ (\pi,w) \; | \; \eqref{eq:conv-si} \}$, where
	\begin{subequations}
		\label{eq:conv-si}
		\begin{align}
		& w \ge g(y),\\
		& w \le g(y^\lobnd) + \left( \frac{g(y^\upbnd)-g(y^\lobnd)}{y^\upbnd-y^\lobnd} \right)(y-y^\lobnd),
		\label{eq:conv-si-2}\\
		& y^\lobnd \le y \le y^\upbnd,
		\label{eq:conv-si-3}\\
		& z=v^i \; w,\quad \underline{xy} = v^i \; y, \quad x = v^i.
		\label{eq:conv-si-4}
		\end{align}
	\end{subequations}
	The disjunctive union of $\Conv(S^i)$, $i=1,\dots,m$, leads to \eqref{eq:prop:simultaneous-hull-gy},  where $w^i$ and $y^i$ are to be regarded as linearization of $\lambda^i w$ and $\lambda^iy$, respectively.
\end{proof}

\begin{remark}
	\label{rem:CQR-representability}
	In Proposition \ref{prop:poly-simul-hull}, if $\Conv(S^i)$ (see proof for definition) is bounded, closed and cone-quadratic representable (CQR), for $i=1,\dots,m$, then $\Conv(S)$ is CQR (see Proposition 3.3.5 in \cite{bental2001}). This result also applies to other conic representations. Let $P_3^{\delta,1-\delta} \coloneqq \{ x\in \mathbb{R}^3 \mid x_1^\delta \cdot x_2^{1-\delta} \ge |x_3| \} $ where $0 < \delta < 1$ is the \textit{power-cone}, and $K_{\exp} = \{ x_1 \ge x_2 \cdot \exp(x_3/x_2), \; x_2>0 \} \cup \{(x_1,0,x_3) \mid x_1 \ge 0, x_3 \le 0\}$ is the \textit{exponential-cone}. It is known that various elementary functions have cone representations \citep{aps2020mosek}. For example, let $g(y) = |y|^\delta$ where $\delta > 1$ (\resp $g(y)=y^\delta$ where $\delta<0$). Then, $w^i \ge g^*(\lambda^i,y^i)$ in Proposition \ref{prop:poly-simul-hull} can be replaced with $(w^i,\lambda^i,y^i) \in P_3^{1/\delta,1-1/\delta}$ (\resp $(w^i,y^i,\lambda^i) \in P_3^{1/(1-\delta),-\delta/(1-\delta)}$). For this work, we are interested in $\delta=-1$ and $\delta=2$ (for reformulation with quadratic polynomials described in \textsection \ref{sec:RDLT-generalization}). Next, let $g(y) = -\ln (y)$, $y>0$ (\resp $g(y)=\exp(y)$), which arises in formulations for identifying thermodynamically efficient distillation configurations (see \cite{jiang2019exergy}). Here, we replace $w^i \ge g^*(\lambda^i,y^i)$ in Proposition \ref{prop:poly-simul-hull} with $(y^i,\lambda^i,-w^i) \in K_{\exp} $ (\resp $(w^i,\lambda^i,y^i) \in K_{\exp} $).\qed
\end{remark}

\begin{remark}
	\label{rem:outer-approx}
	In Proposition \ref{prop:poly-simul-hull}, when $g(y)$ is nonlinear, the convex hull description has nonlinear constraints (see \eqref{eq:prop:simultaneous-hull-gy-1}). To capitalize on LP solvers, we derive a polyhedral outer-approximation of $\Conv(S)$ by outer-approximating the convex hull of each $S^i$ before taking their disjunctive union. Let $\overline{y}^r \in [y^\lobnd,y^\upbnd]$ for $r=1,\dots,R$. Then, an outer-approximation of the convex hull of $S^i$ is given by $\Conv_{OA}{(S^i)} = \text{proj}_\pi\{(\pi,w) \mid w \ge \max\{g(\overline{y}^r)+g'(\overline{y}^r)(y-\overline{y}^r)\}_{r=1}^R; \; \eqref{eq:conv-si-2}-\eqref{eq:conv-si-4} \} $, where $g'(y)$ denotes the first derivative of $g(y)$. The disjunctive union of $\Conv_{OA}{(S^i)}$, $i=1,\dots,m$, yields an outer-approximation of the convex hull of $S$, given by $\Conv_{OA}(S)=\textnormal{proj}_{\pi}\{ (\pi,y^1,\dots,y^m,w^1,\dots,w^m,w,\lambda^1,\dots,\lambda^m) \mid w^i \ge  \max\{g(\overline{y}^r)\lambda^i + g'(\overline{y}^r)(y^i-\overline{y}^r\lambda^i)\}_{r=1}^R; \; \eqref{eq:prop:simultaneous-hull-gy-2}-\eqref{eq:prop:simultaneous-hull-gy-5} \}$.\qed
\end{remark}

Now, consider the set $\mathcal{F}_p$. We lift $\mathcal{F}_p$ to a higher dimensional space by appending bilinear terms of the form $f_p \cdot \theta$ \ie{} $\mathcal{F}_p = \{ (f,\theta,H,\underline{f\theta}) \mid \eqref{eq:set-Fp-def}, \;\underline{f\theta}_p^\feed = f_p^\feed \cdot \theta, \; \underline{f\theta}_p^\rec = f_p^\rec \cdot \theta, \; \underline{f\theta}_p^\strip = f_p^\strip \cdot \theta \}$. Observe that the fractions and bilinear terms in $\mathcal{F}_p$ are defined over the polytope obtained by the intersection of hyperplane $f_p^\feed = f_p^\rec + f_p^\strip$ with the hypercube $[0,F_p]^3$ (see \eqref{eq:set-Fp-def}). We now use Proposition \ref{prop:poly-simul-hull} to obtain $\Conv(\mathcal{F}_p) = \{ (f_p,\theta,H_p,\underline{f\theta}_p) \mid \eqref{eq:conv-Fp-final}\}$ (see \textsection \ref{sec:conv-hull-Fp} for a detailed derivation), where
\begin{subequations}\label{eq:conv-Fp-final}
	\begin{align}
	& H_p^\rec \ge F_p \; T_p^*\left(\frac{f_p^\rec}{F_p},\frac{\underline{f\theta}_p^\rec}{F_p}\right), \quad H_p^\strip \ge F_p \; T_p^*\left(\frac{f_p^\strip}{F_p},\frac{\underline{f\theta}_p^\strip}{F_p}\right),
	\label{eq:conv-Fp-final-1}\\
	& H_p^\rec \le f_p^\rec T_p(\theta^\lobnd) + \left[\frac{T_p(\theta^\upbnd)-T_p(\theta^\lobnd)}{\theta^\upbnd-\theta^\lobnd}\right] (\underline{f\theta}_p^\rec-f_p^\rec \theta^\lobnd),
	\label{eq:conv-Fp-final-2}\\
	& H_p^\strip \le f_p^\strip T_p(\theta^\lobnd) + \left[\frac{T_p(\theta^\upbnd)-T_p(\theta^\lobnd)}{\theta^\upbnd-\theta^\lobnd}\right] (\underline{f\theta}_p^\strip-f_p^\strip \theta^\lobnd),
	\label{eq:conv-Fp-final-3}\\
	& (F_p-f_p^\feed)\theta^\lobnd \le F_p\theta-\underline{f\theta}_p^\feed \le (F_p-f_p^\feed)\theta^\upbnd,\\
	& f_p^\rec \theta^\lobnd \le \underline{f\theta}_p^\rec \le f_p^\rec \theta^\upbnd, \quad  f_p^\strip \theta^\lobnd \le \underline{f\theta}_p^\strip \le f_p^\strip \theta^\upbnd,\\
	& H_p^\feed = H_p^\rec + H_p^\strip, \quad \underline{f\theta}_p^\feed = \underline{f\theta}_p^\rec + \underline{f\theta}_p^\strip, \quad f_p^\feed = f_p^\rec +f_p^\strip.
	\label{eq:conv-Fp-final-6}
	\end{align}
\end{subequations}
and the positively homogeneous function $T_p^*(\lambda,\theta)$ is defined as in \eqref{eq:perspective-fn} from $T_p(\theta)$. Note that the convex hull description does not require introduction of auxiliary variables. This is in contrast to the typical application of disjunctive programming, where new variables are introduced to derive the convex hull in a lifted space. We remark that the above yields a tighter relaxation of $\mathcal{F}_p$ compared to the one obtained by relaxing each fraction and bilinear term separately over the bounds of $f^{\textnormal{in}}_p$, $f^{\textnormal{rs}}_p$, $f^{\textnormal{ss}}_p$, and $\theta$. This is because the first two equations in \eqref{eq:conv-Fp-final-6} are not implied in the latter set. Although, these relations can be obtained using RLT, appending these constraints does not result in \eqref{eq:conv-Fp-final}. This is because, the set described in \eqref{eq:conv-Fp-final} is the simultaneous convex hull of the fraction and bilinear terms. It is known that the simultaneous hull of these functions is strictly contained in the intersection of their individual hulls (see Example 3.8 in \cite{tawarmalani2010}). In particular, \eqref{eq:conv-Fp-final-2} and \eqref{eq:conv-Fp-final-3}, which are linearizations of $-\frac{f_p^\rec}{|\alpha_p-\theta|} \cdot (\theta^\upbnd-\theta) \cdot (\theta-\theta^\lobnd) \le 0$ and $-\frac{f_p^\strip}{|\alpha_p-\theta|} \cdot (\theta^\upbnd-\theta) \cdot (\theta-\theta^\lobnd) \le 0$ respectively, are not implied in the intersection of individual convex hulls.

The convex hull description in \eqref{eq:conv-Fp-final} is cone-quadratic representable (see Remark \ref{rem:CQR-representability}), since the constraints in \eqref{eq:conv-Fp-final-1} can be expressed as second-order cones. For example, $H_1^\rec \ge F_1T_1^*(f_1^\rec/F_1,\underline{f\theta}_1^\rec/F_1) = (f_1^\rec)^2/(\alpha_1 f_1^\rec-\underline{f\theta}_1^\rec)$, or $\sqrt{H_1} \cdot \sqrt{(\alpha_1 f_1^\rec-\underline{f\theta}_1^\rec)} \ge |f_1^\rec|$ (Note that $0 \le \theta^\lobnd f_1^\rec -\underline{f\theta}_1^\rec < \alpha_1f_1^\rec - \underline{f\theta}_1^\rec$). However, we use the cone-quadratic representation only in \textsection \ref{sec:Scenarios}. For our computational experiments in \textsection \ref{sec:Comp-Results},
we use its outer-approximation given by $\Conv_{OA}(\mathcal{F}_p) = \{(f_p,\theta,H_p,\underline{f\theta}_p) \mid H_p^\rec \ge \max\{f_p^\rec T_p(\overline{\theta}^r)+T_p'(\overline{\theta}^r) (\underline{f\theta}_p^\rec - \overline{\theta}^rf_p^\rec ) \}_{r=1}^R, \; H_p^\strip \ge \max\{f_p^\strip T_p(\overline{\theta}^r)+T_p'(\overline{\theta}^r) (\underline{f\theta}_p^\strip - \overline{\theta}^rf_p^\strip) \}_{r=1}^R, \; \eqref{eq:conv-Fp-final-2}-\eqref{eq:conv-Fp-final-6}  \}$ for some $\overline{\theta}^r \in [\theta^\lobnd,\theta^\upbnd]$, $r=1,\dots,R$, where $T'_p(\overline{\theta}^r)$ denotes the derivative of $T_p(\theta)$
at $\overline{\theta}^r$; see Remark \ref{rem:outer-approx}.

Next, consider the set $\mathcal{V}$, which contains bilinear terms defined over a polytope obtained by the intersection of a hyperrectangle in the positive orthant with the hyperplane $U^\rec-U^\strip=\Upsilon^\rec-\Upsilon^\strip$. Clearly, Proposition \ref{prop:poly-simul-hull} can be used to construct the convex hull of $\mathcal{V}$ (only \eqref{eq:prop:simultaneous-hull-gy-3}, the equation with $\underline{xy}$ as the left-hand-side in \eqref{eq:prop:simultaneous-hull-gy-4}, and \eqref{eq:prop:simultaneous-hull-gy-5} are needed to construct the hull). However, Proposition~\ref{prop:poly-simul-hull} requires enumeration of the extreme points of $X$. Instead, in this context, it is more convenient to directly use Proposition \ref{prop:simultaneous-hull-xy}, which is a special case of Proposition~2.2 in \cite{davarnia2017simultaneous}, to obtain $\Conv(\mathcal{V}) = \{(U,\Upsilon,\theta,\underline{U\theta},\underline{\Upsilon\theta}) \mid \eqref{eq:conv-Vth} \}$, where
\begin{subequations}\label{eq:conv-Vth}
	\begin{align}
	& (\underline{U\theta}^\rec -\theta^\lobnd U^\rec) - (\underline{U\theta}^\strip -\theta^\lobnd U^\strip) = (\underline{\Upsilon\theta}^\rec-\theta^\lobnd \Upsilon^\rec) -(\underline{\Upsilon\theta}^\strip-\theta^\lobnd \Upsilon^\strip),
	\label{eq:conv-Vth-1}\\
	& (\theta^\upbnd U^\rec-\underline{U\theta}^\rec) - (\theta^\upbnd U^\strip-\underline{U\theta}^\strip) = (\theta^\upbnd \Upsilon^\rec-\underline{\Upsilon\theta}^\rec) -(\theta^\upbnd \Upsilon^\strip-\underline{\Upsilon\theta}^\strip),
	\label{eq:conv-Vth-2}\\
	& 0 \le \underline{(\cdot)\theta} - \theta^\lobnd\; (\cdot) \le (\cdot)^\upbnd\; (\theta-\theta^\lobnd), \quad \forall \; (\cdot) \in \{U^\rec,U^\strip,\Upsilon^\rec,\Upsilon^\strip\},
	\label{eq:conv-Vth-3}\\
	& 0 \le \theta^\upbnd\; (\cdot) - \underline{(\cdot)\theta} \le (\cdot)^\upbnd\; (\theta^\upbnd-\theta), \quad \forall \; (\cdot) \in \{U^\rec,U^\strip,\Upsilon^\rec,\Upsilon^\strip\}.
	\label{eq:conv-Vth-4}
	\end{align}
\end{subequations}
Finally, we construct the convex relaxation of $\mathcal{U}$ as, $\mathcal{U}_{\Relax} = \{(f, U, \Upsilon, \theta,H,\underline{U\theta}, \underline{\Upsilon\theta}, \underline{f\theta}) \mid \eqref{eq:reformulated-UW}, \; (f_p,H_p,\theta,\underline{f\theta}_p) \in \Conv(\mathcal{F}_p),\;p=1,2, \; (U,\Upsilon,\theta,\underline{U\theta},\underline{\Upsilon\theta}) \in \Conv(\mathcal{V}) \}$.

\begin{proposition}[\cite{davarnia2017simultaneous}]
	\label{prop:simultaneous-hull-xy}
	Let $X=\{x\in \mathbb{R}^n \mid Bx \le b \}$ be a polytope, $\mathcal{D} = X \times [y^\lobnd,y^\upbnd] \times \mathbb{R}^n$, and $S = \{(x,y,z)\in \mathcal{D} \mid \underline{xy}_j = x_j \cdot y, \; j = 1,\dots,n \}$. Then, $\Conv(S) = \{ (x,y,\underline{xy}) \mid y^\lobnd \le y \le y^\upbnd, \; B(\underline{xy}-y^\lobnd x) \le b(y-y^\lobnd), \; B(y^\upbnd x-\underline{xy}) \le b(y^\upbnd-y)  \} $. 
\end{proposition}
\begin{proof}
	See \textsection \ref{sec:ecom:xy-hull-proof} in the Appendix.
\end{proof}

\begin{remark}
	We remark that $\Conv(\mathcal{F}_p)$ and $\Conv(\mathcal{V})$ in \eqref{eq:conv-Fp-final} and \eqref{eq:conv-Vth} imply the convex envelope of $\sum_{p=1}^2 \left[- \frac{\alpha_p(\alpha_p-\theta^\upbnd)f_p^\rec}{\alpha_p-\theta}+\alpha_pf_p^\rec\right] + \Upsilon^\rec \cdot \theta - \Upsilon^\rec \cdot \theta^\upbnd$ over bound constraints on $f_p^\rec$, $\theta$, and $\Upsilon^\rec$ (see \eqref{eq:RDLT-2}). This is because, when all $f_p^\rec$ and $\Upsilon^\rec$ are fixed, the function is concave in $\theta$. Then, by Theorem 1.4 in \citet{rikun1997}, it follows that the convex envelope is obtained by replacing $-\frac{\alpha_p(\alpha_p-\theta^\upbnd)f_p^{\textnormal{rs}}}{\alpha_p-\theta}$ for all $p$ and $\Upsilon^\rec \cdot \theta$ by their convex envelopes.\qed
\end{remark}

We comment on the construction of convex relaxations of $\mathcal{U}$ when additional RDLT cuts described in \textsection \ref{sec:RDLT-generalization} are appended to $\mathcal{U}_\text{ref}$. Reformulation of Underwood constraints using quadratic polynomials of $\theta$ introduces nonconvex terms of the form $f_p \cdot \theta$, $\Upsilon \cdot \theta^2$ (see \eqref{eq:RDLT-quadratic}), in addition to the existing $f_p \cdot T_p(\theta)$ and $\Upsilon\cdot \theta$ terms in $\mathcal{U}_\text{ref}$. We relax $f_p \cdot T_p (\theta)$ and $f_p \cdot \theta$ using the simultaneous hull description in \eqref{eq:conv-Fp-final}. Although Proposition \ref{prop:poly-simul-hull} yields the simultaneous hull of $\Upsilon \cdot \theta^2$ and $\Upsilon\cdot \theta$ terms over the polytope in $\mathcal{V}$, we do not implement this relaxation. This is because the hull description does not project onto the space of problem variables in a striaghtforward manner. Instead, we convexify each pair of $\Upsilon \cdot \theta^2$ and $\Upsilon\cdot \theta$ terms over a box using Corollary \ref{prop:simultaneous-hull}, and append the RLT cuts $U^\rec \cdot \theta^2 - U^\strip \cdot \theta^2 = \Upsilon^\rec \cdot \theta^2 - \Upsilon^\strip \cdot \theta^2 $ and $U^\rec \cdot \theta - U^\strip \cdot \theta = \Upsilon^\rec \cdot \theta - \Upsilon^\strip \cdot \theta $.

On the other hand, reformulation of Underwood constraints using inverse bound factors introduces nonconvex terms of the form $f_p \cdot \theta^{-1}$ and $\Upsilon \cdot \theta^{-1}$ (see \eqref{eq:RDLT-inverse-bf}), in addition to the existing $f_p \cdot T_p(\theta)$ and $\underline{\Upsilon\theta}$ terms in $\mathcal{U}_\text{ref}$. We relax $f_p \cdot T_p (\theta)$ and $f_p \cdot \theta$ using the simultaneous hull description in \eqref{eq:conv-Fp-final}. We use a similar hull description, obtained using Proposition \ref{prop:poly-simul-hull}, to relax $f_p \cdot \theta^{-1}$ and $f_p\cdot \theta$. Finally, for the same reason mentioned above, we convexify each pair of $\Upsilon \cdot \theta^{-1}$ and $\Upsilon \cdot \theta$ terms using Corolloary \ref{prop:simultaneous-hull}, and append RLT cuts $U^\rec \cdot \theta^{-1} - U^\strip \cdot \theta^{-1} = \Upsilon^\rec \cdot \theta^{-1} - \Upsilon^\strip \cdot \theta^{-1} $ and $U^\rec \cdot \theta - U^\strip \cdot \theta = \Upsilon^\rec \cdot \theta - \Upsilon^\strip \cdot \theta $.

\begin{corollary}
	\label{prop:simultaneous-hull}
	Let $\mathcal{B}=[x^\lobnd,x^\upbnd]\times [y^\lobnd,y^\upbnd]\times \mathbb{R}^2$, where we assume $0\le x^\lobnd$, $g(y):[y^\lobnd,y^\upbnd] \rightarrow \mathbb{R}$ is convex, and $S=\{ (x,y,z,\underline{xy}) \in \mathcal{B} \mid \underline{xy}=x\cdot y,\; z = x\cdot g(y) \}.$ Then, $\Conv(S) =\{  (x,y,z,\underline{xy})  \mid \eqref{eq:Simul-hull} \} $, where
	\begin{subequations}\label{eq:Simul-hull}
		\begin{align}
		& z \ge x^\lobnd  g^*\left(\frac{x^\upbnd-x}{x^\upbnd-x^\lobnd},\frac{x^\upbnd y-\underline{xy}}{x^\upbnd-x^\lobnd}\right) + x^\upbnd  g^*\left(\frac{x-x^\lobnd}{x^\upbnd-x^\lobnd},\frac{\underline{xy}-x^\lobnd y}{x^\upbnd-x^\lobnd}\right), 
		\label{eq:Simul-hull-1}\\
		& z \le g(y^\lobnd)\cdot x + \left[\frac{g(y^\upbnd)-g(y^\lobnd)}{y^\upbnd-y^\lobnd}\right] \left(\underline{xy}-y^\lobnd x\right),
		\label{eq:Simul-hull-2}\\
		& (x^\upbnd-x)y^\lobnd \le x^\upbnd y-\underline{xy} \le (x^\upbnd-x)y^\upbnd, \quad (x-x^\lobnd)y^\lobnd \le \underline{xy}-x^\lobnd y \le (x-x^\lobnd)y^\upbnd,
		\label{eq:Simul-hull-4}
		\end{align}
	\end{subequations}
	and $g^*(\lambda^*,y^*)$ is defined as in \eqref{eq:perspective-fn}. Further, the outer-approximation of the convex hull is $\Conv_{OA}(S) = \{ (x,y,z,\underline{xy}) \mid \eqref{eq:Simul-hull-OA}, \; \eqref{eq:Simul-hull-2}-\eqref{eq:Simul-hull-4} \}$, where
	\begin{align}
	z \ge & \frac{x^\lobnd}{x^\upbnd-x^\lobnd} \max\{ g(\overline{y}^r) (x^\upbnd-x) + g'(\overline{y}^r) (x^\upbnd y-\underline{xy} - (x^\upbnd-x)\overline{y}^r) \}_{r=1}^R + \notag\\
	& \frac{x^\upbnd}{x^\upbnd-x^\lobnd} \max\{ g(\overline{y}^r) (x-x^\lobnd) + g'(\overline{y}^r) (\underline{xy} - x^\lobnd y - (x-x^\lobnd)\overline{y}^r) \}_{r=1}^R,
	\label{eq:Simul-hull-OA}
	\end{align}
	for some $\overline{y}^r \in [y^\lobnd,y^\upbnd]$, $r=1,\dots,R$, and $g'(y)$ denotes the first derivative of $g(y)$ w.r.t $y$.
\end{corollary}
\begin{proof}
See \textsection \ref{sec:ecom:hull-of-xg-xy} in the Appendix.
\end{proof}

\subsection{Valid Relaxation for $\theta^\lobnd=\alpha_2$ and/or $\theta^\upbnd = \alpha_1$}
\label{sec:full-relaxation}
In the previous subsection, we have assumed that $\alpha_2 < \theta^\lobnd$ and $\theta^\upbnd < \alpha_1$. Instead, if $\alpha_1$ and/or $\alpha_2$ is an admissible value of $\theta$, we cannot directly use \eqref{eq:conv-Fp-final} to convexify $\mathcal{F}_p$, because $T_1(\alpha_1)$ and $T_2(\alpha_2)$ are not well-defined. To construct a valid relaxation, we first restrict the admissible values of $\theta$ to a subset of the interval $[\alpha_2,\alpha_1]$ by recognizing that each fraction in $\mathcal{U}$ is bounded.
\begin{proposition}
	\label{prop:theta-bounds}
	(i) Valid upper bounds on $\frac{f_1^\rec}{\alpha_1-\theta}$, $(H_1^\rec)^\upbnd$, and on $\frac{f_2^\rec}{\theta-\alpha_2}$, $(H_2^\rec)^\upbnd$, are given by
	\begin{subequations}
		\begin{align}
		& (H_1^\rec)^\upbnd = \frac{(\Upsilon^\rec)^\upbnd(\alpha_1-\alpha_2)+\alpha_1F_1 + \alpha_2F_2}{\alpha_1(\alpha_1-\alpha_2)},
		\label{eq:H1up-bnd}\\
		& (H_2^\rec)^\upbnd = \frac{\alpha_1F_1+\alpha_2F_2-(E^\rec \cdot (\alpha_1-\theta))^\lobnd}{\alpha_2 (\alpha_1-\alpha_2))}.
		\label{eq:H2up-bnd}
		\end{align}
	\end{subequations}
	(ii) the admissible region of $\theta$ in the interval $[\alpha_2,\alpha_1]$ is given by
	\begin{align}
	\alpha_2 + \frac{f_2^\rec}{(H_2^\rec)^\upbnd} \le \theta \le \alpha_1-\frac{f_1^\rec}{(H_1^\rec)^\upbnd}.
	\label{eq:th-bound-cuts}
	\end{align}
\end{proposition}
\begin{proof}
(i) Consider the second inequality in \eqref{eq:reformulated-UW-3}. Since this inequality holds for any $\theta^\lobnd$ less than $\theta$, if we substitute $\theta^\lobnd$ with $\alpha_2$, the inequality remains valid. Then, we obtain \eqref{eq:H1up-bnd} from $\alpha_1(\alpha_1-\alpha_2)H_1^\rec \le (\underline{\Upsilon\theta}^\rec-\alpha_2 \Upsilon^\rec) + \alpha_1f_1^\rec + \alpha_2f_2^\rec \le (\Upsilon^\rec)^\upbnd(\alpha_1-\alpha_2) + \alpha_1F_1 + \alpha_2 F_2$, where the last inequality is because $f_1^\rec\le F_1$, $f_2^\rec\le F_2$, and $\underline{\Upsilon\theta}^\rec-\alpha_2 \Upsilon^\rec\le (\Upsilon^\rec)^\upbnd(\alpha_1-\alpha_2)$. Similarly, we substitute $\theta^\upbnd=\alpha_1$ in the first inequality in \eqref{eq:reformulated-UW-4}, and rearrange to get $\alpha_2(\alpha_1-\alpha_2) H_2^\rec \le -E^\rec \cdot (\alpha_1-\theta) + \alpha_1f_1^\rec + \alpha f_2^\rec$. We maximize the right hand side by substituting $f_1^\rec = F_1$, $f_2^\rec=F_2$, and $(E^\rec \cdot (\alpha_1-\theta))$ by its lower bound which is computed using the bounds on $E^\rec$ and $\theta$. This leads to the bound in \eqref{eq:H2up-bnd}.
	
(ii) Every point feasible to $\mathcal{U}$ satisfies $f_1^\rec/(\alpha_1-\theta) \le (H_1^\rec)^\upbnd$ and $f_2^\rec/(\theta-\alpha_2) \le (H_2^\rec)^\upbnd$. Rearranging the inequalities yields \eqref{eq:th-bound-cuts}. 
\end{proof}

We remark that the bounds on $H_p^\feed$ and $H_p^\strip$ for $p=1,2$ can be computed in the same manner as in the proof of (i) in Proposition \ref{prop:theta-bounds}. Even when additional fractions are present in the Underwood constraints, each fraction can be bounded, since the remaining fractions are strictly bounded in the interval of $\theta$. We revisit the argument on bounds of $\theta$ in light of Proposition \ref{prop:theta-bounds}. As mentioned before, the common approach used in the literature to overcome the singularity arising due to $\theta$ approaching one of the adjoining relative volatilities has been to restrict $\theta$ to belong to $[\alpha_2+\epsilon_\theta,\alpha_1-\epsilon_\theta]$. However, observe that our bounds in \eqref{eq:th-bound-cuts} depend on $f_1^\rec$ and $f_2^\rec$. This explains the difficulty we encountered in choosing a value for $\epsilon_\theta$ in our computations with prior formulations. We have found that there are instances when $\theta$ is fairly close to one of the relative volatilities, particularly when the corresponding flow is small. We will provide a rigorous approach to addressing this singularity using \eqref{eq:th-bound-cuts}. Our approach will be to construct a relaxation of $\mathcal{F}_1$ as the intersection of simultaneous convex hulls of $f_1 \cdot T_1(\theta)$ and $f_1\cdot \theta$. For brevity, we only discuss the relaxation for $\mathcal{F}_1$ in detail, and remark that a similar result is easily derived for $\mathcal{F}_2$.

\begin{proposition}
	\label{prop:hull-of-H1}
	Let $\mathcal{H}_1=\{(f_1,\theta,H_1,\underline{f\theta}_1) \mid 0 \le f_1 \le F_1, \; \theta^\lobnd \le \theta \le \alpha_1-f_1/H_1^\upbnd, \; H_1 = f_1 \cdot T_1(\theta),  \textnormal{  if  } \theta < \alpha_1; \; H_1\in [0,H_1^\upbnd] \textnormal{  if   } \; \theta = \alpha_1,\; \underline{f\theta}_1 = f_1 \cdot \theta  \}$, where $\alpha_2 \le \theta^\lobnd$. Then, $\Conv(\mathcal{H}_1) = \textnormal{proj}_{(f_1,\theta,H_1,\underline{f\theta}_1)} \{ (f_1,\theta,H_1,\underline{f\theta}_1,\theta^a,\theta^b,\theta^c,\lambda^a,\lambda^b,\lambda^c) \mid \eqref{eq:prop:hull-of-H1} \}$, where
	\begin{subequations}\label{eq:prop:hull-of-H1}
		\begin{align}
		& H_1^\upbnd \lambda^b + F_1 T_1^*(\lambda^c,\theta^c) \le H_1 \le H_1^\upbnd\left(\frac{\theta^a-\theta^\lobnd \lambda^a}{\alpha_1-\theta^\lobnd}\right) + H_1^\upbnd\lambda^b + \frac{F_1\lambda^c}{\alpha_1-\theta^\lobnd} + H_1^\upbnd \left(\frac{\theta^c-\theta^\lobnd\lambda^c}{\alpha_1-\theta^\lobnd}\right),\\
		& H_1^\upbnd \left(\alpha_1-\frac{F_1}{H_1^\upbnd}\right)(\alpha_1\lambda^b-\theta^b)+F_1\theta^c \le \underline{f\theta}_1 \le H_1^\upbnd \left(\alpha_1\theta^b - \frac{(\theta^b)^2}{\lambda^b}\right)+F_1\theta^c,\\
		& \theta^\lobnd \lambda^a  \le \theta^a \le \alpha_1 \lambda^a, \quad \left(\alpha_1-\frac{F_1}{H_1^\upbnd}\right) \lambda^b \le \theta^b \le \alpha_1 \lambda^b, \quad \theta^\lobnd \lambda^c \le \theta^c \le \left(\alpha_1-\frac{F_1}{H_1^\upbnd}\right) \lambda^c,\\
		& f_1 = H_1^\upbnd(\alpha_1\lambda^b-\theta^b) + F_1\lambda^c, \quad \theta = \theta^a + \theta^b + \theta^c, \quad \lambda^a+\lambda^b+\lambda^c = 1, \quad \lambda^a,\lambda^b,\lambda^c \ge 0.
		\end{align}
	\end{subequations}
\end{proposition}
\begin{proof}
See \textsection \ref{sec:proof-of-H1} in the Appendix.
\end{proof}

The convex hull in Proposition~\ref{prop:hull-of-H1} requires several additional variables. To avoid the introduction of these additional variables, we use its relaxation, $\mathcal{H}_{1,\Relax}$, derived in \textsection \ref{sec:relax:hull-of-H1} and shown below:
\begin{subequations}\label{eq:relax-hull-of-H1}
	\begin{align}
	& \max \left\{ f_1 T_1(\overline{\theta}^r) + T_1'(\overline{\theta}^r)(\underline{f\theta}_1-\overline{\theta}^r f_1) \right\}_{r=1}^R
	\le H_1 \le \frac{f_1}{\alpha_1-\theta^\lobnd} + H_1^\upbnd \left(\frac{\theta-\theta^\lobnd}{\alpha_1-\theta^\lobnd}\right),
	\label{eq:relax-hull-of-H1-1}\\
	& \max\left\{ \theta^\lobnd f_1, \; F_1\theta + \alpha_1 f_1 - \alpha_1F_1 \right\} \le \min \left\{\alpha_1f_1, \; F_1\theta + \theta^\lobnd f_1 - \theta^\lobnd F_1 \right\},
	\label{eq:relax-hull-of-H1-2}\\
	&\theta^\lobnd \le \theta \le \alpha_1-\frac{f_1}{H_1^\upbnd},
	\label{eq:relax-hull-of-H1-5}
	\end{align}
\end{subequations}
where $\overline{\theta}^r \in [\theta^\lobnd,\alpha_1)$, $r=1,\dots,R$.
Here, we argue from first principles that \eqref{eq:relax-hull-of-H1} is a valid relaxation. To derive the first inequality in \eqref{eq:relax-hull-of-H1-1}, observe that $H_1 \ge f_1 \cdot T_1 (\theta) \ge f_1 \cdot \max \{ T_1(\overline{\theta}^r) + T_1'(\overline{\theta}^r)(\theta-\overline{\theta}^r) \}_{r=1}^R$. Disaggregating the product and linearizing the bilinear term yields \eqref{eq:relax-hull-of-H1-1}. To derive the second inequality in \eqref{eq:relax-hull-of-H1-1}, we begin with $H_1 \cdot (\alpha_1-\theta) \le f_1$, and replace the bilinear term on the left hand side with its convex envelope. \eqref{eq:relax-hull-of-H1-2} is the convex hull of $\underline{f\theta}_1 = f_1 \cdot \theta$ over $[0,F_1]\times[\theta^\lobnd,\alpha_1] $, and \eqref{eq:relax-hull-of-H1-5} is the same as \eqref{eq:th-bound-cuts}. Using \eqref{eq:relax-hull-of-H1}, we obtain a valid relaxation of $\mathcal{F}_1$ given by $\mathcal{F}_{1,\Relax} = \{ (f_1,\theta,H_1,\underline{f\theta}_1) \mid (f_1^\feed,\theta,H_1^\feed,\underline{f\theta}_1^\feed) \in \mathcal{H}_{1,\Relax}^\feed, \; (f_1^\rec,\theta,H_1^\rec,\underline{f\theta}_1^\rec) \in \mathcal{H}_{1,\Relax}^\rec, \; (f_1^\rec,\theta,H_1^\rec,\underline{f\theta}_1^\rec) \in \mathcal{H}_{1,\Relax}^\strip, \; H_1^\feed = H_1^\rec + H_1^\strip, \; \underline{f\theta}_1^\feed = \underline{f\theta}_1^\rec + \underline{f\theta}_1^\strip \} $. Inspired from \eqref{eq:conv-Fp-final}, the last two equations in the relaxation are derived by multiplying the component mass balance, \eqref{eq:surrogate-UW-flowbal}, with $T_1(\theta)$ and $\theta$, respectively.

\subsection{Discretization and Solution Procedure}
\label{sec:Discretization}
In this work, instead of using convex relaxations of $\mathcal{U}$ in a spatial branch-and-bound framework to solve the MINLP, we construct a piecewise relaxation (see Definition \ref{def:piecewise}) that is iteratively improved until we prove $\epsilon_r-$optimality. This approach capitalizes on state-of-the-art MIP solvers, such as Gurobi.

\begin{definition}[Piecewise Relaxation]
	\label{def:piecewise}
	Let $x = (x_1,\dots,x_n)$, $\mathcal{B}=[x^\lobnd,x^\upbnd]\times[y^\lobnd,y^\upbnd] \subset \mathbb{R}^{n+1} $, $S=\{ (x,y)\in \mathcal{B} \mid g_i(x,y) \le 0, \; i=1,\dots,m \}$, and $S_{\Relax}=\{(x,y)\in \mathcal{B} \mid \widecheck{g}_i(x,y) \le 0, \; i=1,\dots,m\}$ be its convex relaxation , where $\{\widecheck{g}_i\}_{i=1}^m$ denote convex underestimators of $\{g_i\}_{i=1}^m$ over $\mathcal{B}$. Let, the domain of $y$ be partitioned as $\mathcal{I} = \{[Y^0,Y^1],\dots,[Y^{|\mathcal{I}|-1},Y^{|\mathcal{I}|}]\}$ with $Y^0=y^\lobnd$, $Y^{|\mathcal{I}|} = y^\upbnd$ and $Y^0 \le Y^1 \le \dots Y^{|\mathcal{I}|}$. By piecewise relaxation of $S$, we refer to $\bigcup_{t=1}^{|\mathcal{I}|} S_{t,\Relax}$, where $S_{t,\Relax} = \{(x,y,z)\in \mathcal{B}_t \mid \widecheck{g}_{i,t}(x,y) \le 0, \; i=1,\dots,m \}$, $\mathcal{B}_t=[x^\lobnd,x^\upbnd]\times[Y^{t-1},Y^t]$, and $\widecheck{g}_{i,t}$ is the convex under-estimator of $g_i$ over $\mathcal{B}_t$. \qed
\end{definition}
Piecewise relaxation of $\mathcal{U}$ can be constructed by partitioning the domain of Underwood root as $\mathcal{I} =\{ [\Theta^0,\Theta^1],\dots,[\Theta^{|\mathcal{I}-1|},\Theta^{|\mathcal{I}|}] \}$, where $\Theta^0 = \alpha_2$, $\Theta^{|\mathcal{I}|} = \alpha_1$, and $\Theta^0 \le \Theta^1 \le \dots \le \Theta^{|\mathcal{I}|}$, and taking the union of sets $\bigcup_{t=1}^{|\mathcal{I}|} \mathcal{U}_{t,\Relax}$, where $\mathcal{U}_{t,\Relax}$ denotes the convex relaxation of $\mathcal{U}$ restricted to $\theta \in [\Theta^{t-1},\Theta^t]$. The set $\mathcal{U}_{t,\Relax}$ is constructed as outlined in \textsection\ref{sec:part-relaxation} and \textsection\ref{sec:full-relaxation}. Next, using standard disjunctive programming techniques, the piecewise relaxation can be expressed as a Mixed Integer Program (MIP). While this approach leads to a \textit{locally ideal} formulation, it leads to a bigger problem size, because of which the computational time required is higher. Thus, in favor of smaller problem size, we do the following.

Instead of reformulating $\mathcal{U}$ in each partition using the local bound factors of $\theta$, we reformulate with the overall bound factors of $\theta$: $(\theta-\alpha_2)$ and $(\alpha_1-\theta)$. Next, we require that $(f_p,\theta,H_p,\underline{f\theta}_p)$, $p=1,2$, and $(U,\Upsilon,\theta,\underline{U\theta},\underline{\Upsilon\theta})$ lie in piecewise relaxations of $\mathcal{F}_p$ and $\mathcal{V}$, respectively. We choose piecewise relaxation of $\mathcal{F}_1$ to be $\bigcup_{t=1}^{|\mathcal{I}|-1} \Conv_{OA}(\mathcal{F}_{1,t}) \cup \mathcal{F}_{1,|\mathcal{I}|,\Relax}$, piecewise relaxation of $\mathcal{F}_2$ to be $\mathcal{F}_{2,1,\Relax} \cup \bigcup_{t=2}^{|\mathcal{I}|} \Conv_{OA}(\mathcal{F}_{2,t}) $, and piecewise relaxation of $\mathcal{V}$ to be $\bigcup_{t=1}^{|\mathcal{I}|} \Conv(\mathcal{V}_t)$. Here, the additional subscript $t$ denotes that the set is restricted to $\theta \in [\Theta^{t-1},\Theta^t]$. Observe that if zero is not an admissible value to the denominators of the fractions, we use outer-approximation of convex hulls derived in \textsection \ref{sec:part-relaxation} to relax $\mathcal{F}_p$. Otherwise, we use a relaxation of the convex hull description, such as the one derived in \textsection \ref{sec:full-relaxation}. We use disjunctive programming to express the piecewise relaxations as mixed-integer sets (see \textsection\ref{sec:ecom-MIP-Representation} for description of the sets). In \textsection \ref{sec:Scenarios}, we illustrate through numerical examples the impact of various aspects described in this section in strengthening the overall relaxation of MINLP \MINLP{}.
\long\def\MIPRepresentation{\begin{subequations}
		\begin{align}
		& H_1^\rec \ge f_{1,t}^\rec T_1(\Theta^{t-1}) + T_1'(\Theta^{t-1}) (\underline{f\theta}_{1,t}^\rec - \Theta^{t-1} f_{1,t}^\rec), & & \quad \llbracket t \rrbracket_{1}^{|\mathcal{I}|}
		\label{eq:mip-F1-1}\\
		& H_1^\rec \ge f_{1,t}^\rec T_1(\Theta^{t}) + T_1'(\Theta^{t}) (\underline{f\theta}_{1,t}^\rec - \Theta^{t} f_{1,t}^\rec), & & \quad \llbracket t \rrbracket_{1}^{|\mathcal{I}|-1}\\
		& H_1^\strip \ge f_{1,t}^\strip T_1(\Theta^{t-1}) + T_1'(\Theta^{t-1}) (\underline{f\theta}_{1,t}^\strip - \Theta^{t-1} f_{1,t}^\strip), & & \quad \llbracket t \rrbracket_{1}^{|\mathcal{I}|}\\
		& H_1^\strip \ge f_{1,t}^\strip T_1(\Theta^{t}) + T_1'(\Theta^{t}) (\underline{f\theta}_{1,t}^\strip - \Theta^{t} f_{1,t}^\strip), & & \quad \llbracket t \rrbracket_{1}^{|\mathcal{I}|-1}
		\label{eq:mip-F1-4}\\
		& H_1^\rec \le \sum_{t=1}^{|\mathcal{I}|-1} f_{1,t}^\rec T_1(\Theta^{t-1}) + \left[\frac{T_1(\Theta^{t-1}) -T_1(\Theta^{t}) }{\Theta^{t-1}-\Theta^t}\right] (\underline{f\theta}_{1,t}^\rec - \Theta^{t-1} f_{1,t}^\rec) \notag\\
		& \hspace{2.5cm} + \frac{f_{1,|\mathcal{I}|}^\rec}{\alpha_1-\Theta^{|\mathcal{I}|-1}} + (H_1^\rec)^\upbnd \left[\frac{\theta_{|\mathcal{I}|}-\Theta^{|\mathcal{I}|-1}\mu_t}{\Theta^{|\mathcal{I}|}-\Theta^{|\mathcal{I}|-1}}\right]
		\label{eq:mip-F1-5}\\
		& H_1^\strip \le \sum_{t=1}^{|\mathcal{I}|-1} f_{1,t}^\strip T_1(\Theta^{t-1}) + \left[\frac{T_1(\Theta^{t-1}) -T_1(\Theta^{t}) }{\Theta^{t-1}-\Theta^t}\right] (\underline{f\theta}_{1,t}^\strip - \Theta^{t-1} f_{1,t}^\strip) \notag\\
		& \hspace{2.5cm} + \frac{f_{1,|\mathcal{I}|}^\strip}{\alpha_1-\Theta^{|\mathcal{I}|-1}} + (H_1^\strip)^\upbnd \left[\frac{\theta_{|\mathcal{I}|}-\Theta^{|\mathcal{I}|-1}\mu_t}{\Theta^{|\mathcal{I}|}-\Theta^{|\mathcal{I}|-1}}\right]\\
		& (F_1\mu_t-f_{1,t}^\rec-f_{1,t}^\strip) \Theta^{t-1}\le  (F_1\theta_t-\underline{f\theta}_{1,t}^\rec-\underline{f\theta}_{1,t}^\strip) \le (F_1\mu_t-f_{1,t}^\rec-f_{1,t}^\strip) \Theta^t, & & \quad \llbracket t \rrbracket_{1}^{|\mathcal{I}|} \\
		& f_{1,t}^\rec \Theta^{t-1} \le \underline{f\theta}_{1,t}^\rec \le f_{1,t}^\rec \Theta^t, \quad f_{1,t}^\strip \Theta^{t-1} \le \underline{f\theta}_{1,t}^\strip \le f_{1,t}^\strip \Theta^t, & & \quad \llbracket t \rrbracket_{1}^{|\mathcal{I}|}
		\label{eq:mip-F1-8}\\
		& H_1^\feed = H_1^\rec + H_1^\strip, \quad \underline{f\theta}_1^\feed = \underline{f\theta}_1^\rec + \underline{f\theta}_1^\strip\\
		& f_1^\rec = \sum_{t=1}^{|\mathcal{I}|} f_{1,t}^\rec, \quad f_1^\strip = \sum_{t=1}^{|\mathcal{I}|} f_{1,t}^\strip, \quad \theta = \sum_{t=1}^{|\mathcal{I}|} \theta_t\\
		& \sum_{t=1}^{|\mathcal{I}|} \mu_t = 1, \quad  \mu_t \in\{0,1\}, \quad \llbracket t \rrbracket_{1}^{|\mathcal{I}|}
		\label{eq:mip-F1-11}
		\end{align}
		\label{eq:mip-F1}
	\end{subequations}
	and
	\begin{subequations}
		\begin{align}
		& \underline{U\theta}^\rec -\underline{U\theta}^\strip = \underline{\Upsilon\theta}^\rec - \underline{\Upsilon\theta}^\strip\\
		& U^\rec_t - U^\strip_t = \Upsilon^\rec_t - \Upsilon^\strip_t, & & \llbracket t \rrbracket_{1}^{|\mathcal{I}|}\\
		& 0 \le \underline{(\cdot)\theta} - \sum_{t=1}^{|\mathcal{I}|}\Theta^{t-1} (\cdot)_t \le (\cdot)^\upbnd \theta - (\cdot)^\upbnd\sum_{t=1}^{|\mathcal{I}|}\Theta^{t-1} \mu_t, \; \forall \; (\cdot)\in\{U^\rec,U^\strip,\Upsilon^\rec,\Upsilon^\strip\}, & & \llbracket t \rrbracket_{1}^{|\mathcal{I}|}\\
		& 0 \le \sum_{t=1}^{|\mathcal{I}|}\Theta^{t} (\cdot)_t - \underline{(\cdot)\theta} \le (\cdot)^\upbnd\sum_{t=1}^{|\mathcal{I}|}\Theta^{t} \mu_t - (\cdot)^\upbnd \theta, \; \forall \; (\cdot)\in\{U^\rec,U^\strip,\Upsilon^\rec,\Upsilon^\strip\} , & & \llbracket t \rrbracket_{1}^{|\mathcal{I}|}\\
		& (\cdot) = \sum_{t=1}^{|\mathcal{I}|} (\cdot)_t, \quad 0 \le (\cdot) \le (\cdot)^\upbnd \mu_t, \quad; \forall \; (\cdot)\in\{U^\rec,U^\strip,\Upsilon^\rec,\Upsilon^\strip\}\\
		&\sum_{t=1}^{|\mathcal{I}|} \mu_t = 1, \quad \mu_t \in\{0,1\},\; \llbracket t \rrbracket_{1}^{|\mathcal{I}|}.
		\end{align}
\end{subequations}}

Algorithm \ref{algo:adaptive} outlines our approach to solve the MINLP. We start with a coarse discretization and use an adaptive partitioning scheme to iteratively refine the partitions until $\epsilon_r-$optimality is achieved. To avoid numerical issues, we maintain that each partition, $(\Theta_{ijq}^t-\Theta_{ijq}^{t-1})$, is at least \texttt{MinPrtSize} in length.

\begin{algorithm}
	\includegraphics[scale=0.95]{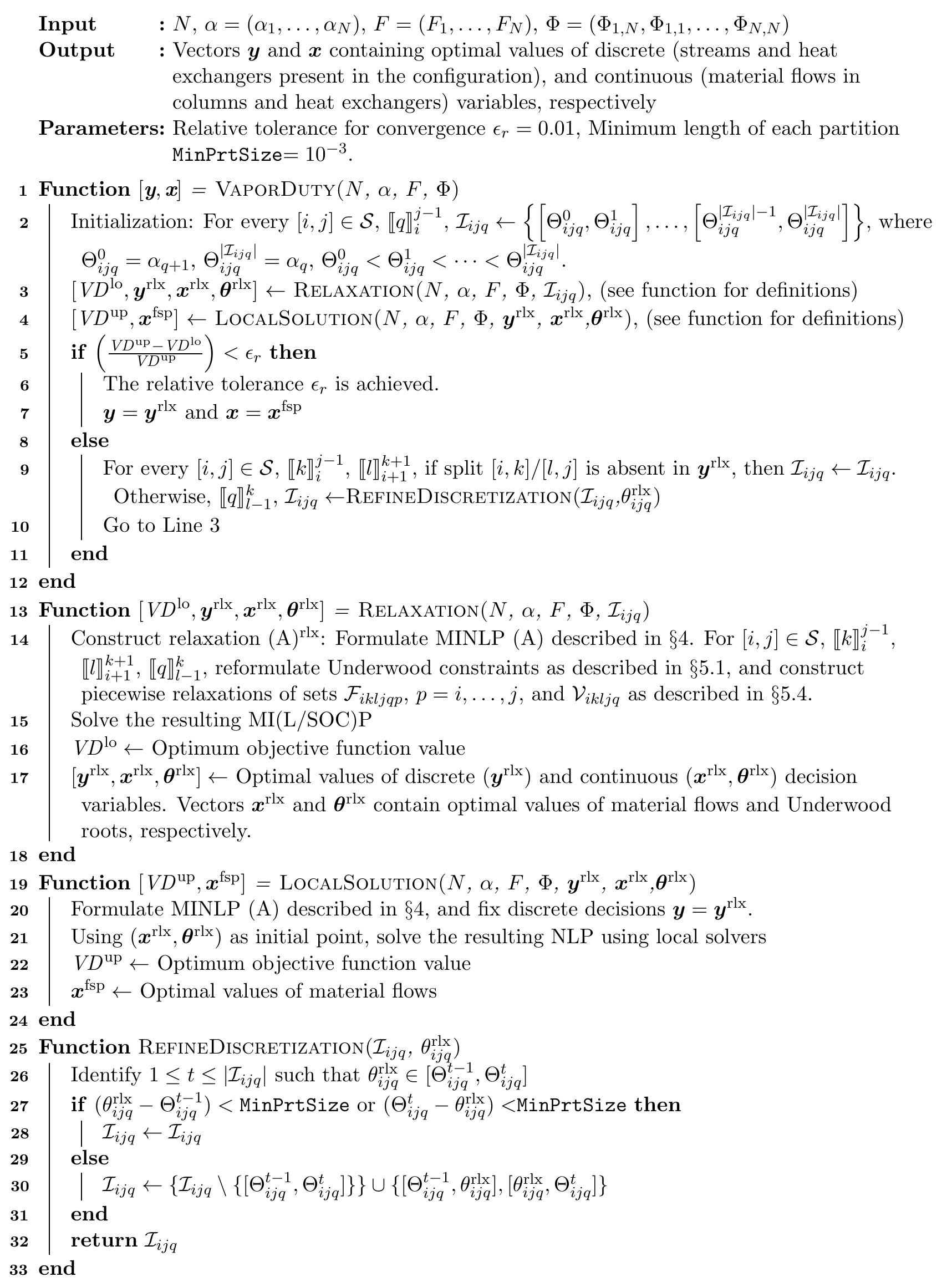}
	\caption{Adaptive partitioning scheme to solve MINLP \MINLP{}}
	\label{algo:adaptive}
\end{algorithm}

 \section{Effect of Individual Cuts on Relaxation}
\label{sec:Scenarios}
This section illustrates, through numerical examples, the impact of of various aspects described in \textsection \ref{sec:Relaxation} in strengthening the overall relaxation of MINLP \MINLP{}. We highlight the individual effect of RDLT cuts derived from Underwood constraints, simultaneous hulls derived in \textsection \ref{sec:part-relaxation}, and discretization on the overall relaxation. In all the scenarios below, stream and heat exchanger variables are considered to be binary.
	\begin{itemize}
		\item[\textbf{Scenario 1}]: (BARON's root node relaxation) Here, we use BARON 18.5.8, on GAMS 25.1, to construct and solve the relaxation of MINLP \MINLP. This is achieved by specifying BARON option \texttt{MaxIter} = 1, which terminates the branch-and-cut algorithm after processing the root node. We let $\theta_{ijq}\in[\alpha_{q+1}+\epsilon_\theta,\alpha_{q}-\epsilon_\theta]$, with $\epsilon_\theta=10^{-7}$, for every $\llbracket q \rrbracket_{i}^{j-1}$, $[i,j]\in\mathcal{S}$ to avoid a possible division by zero. We use BARON's root node relaxation as a reference for comparison. We remark that BARON solves MIP relaxations as needed \citep{kilincc}. We also verified that the bound obtained is close to solving a factorable MIP relaxation.
		
		\item[\textbf{Scenario 2}]: (Simultaneous hull of fractional terms) This scenario illustrates the improvement in relaxation due to the use of simultaneous convexification techniques. We linearize all Underwood constraints in the MINLP by introducing auxiliary variables for each fraction. To relax fractional terms, we use \eqref{eq:conv-Fp-final}, or \eqref{eq:relax-hull-of-H1} if zero is an admissible value for the range of the denominator of fractions. The nonlinear constraints in \eqref{eq:conv-Fp-final} are expressed as second-order cones, and the resulting Mixed Integer Second-order Cone Program (MISOCP) is solved with Gurobi 8.0 using Gurobi/MATLAB interface.
		
		\item[\textbf{Scenario 3}]: (RDLT with linear polynomials of $\theta$) This scenario illustrates the improvement in relaxation due to reformulation of Underwood constraints using RDLT. We reformulate Underwood constraints as in \eqref{eq:reformulated-UW}, convexify fractional terms using \eqref{eq:conv-Fp-final} or \eqref{eq:relax-hull-of-H1}, and convexify bilinear terms of the form $\underline{\Upsilon\theta} =\Upsilon \cdot \theta $ using \eqref{eq:conv-Vth}.
		
		\item[\textbf{Scenario 4}]: (RDLT with quadratic polynomials of $\theta$) To the relaxation in Scenario 3, we add cuts derived by reformulating Underwood constraints with quadratic polynomials of $\theta$ (see \eqref{eq:RDLT-quadratic}), as described in \textsection \ref{sec:RDLT-generalization}. This introduces additional nonlinear terms of the form $\Upsilon \cdot \theta^2$, which we relax in the manner described towards the end of \textsection \ref{sec:part-relaxation}.
		
		\item[\textbf{Scenario 5}]: (RDLT with inverse bound factors of $\theta$) To the relaxation in Scenario 3, we add cuts derived by reformulating Underwood constraints with inverse bound factors (see \eqref{eq:RDLT-inverse-bf}). This introduces additional nonlinear terms of the form $f_1/\theta$ and $V/\theta$, which we relax in the manner described towards the end of \textsection \ref{sec:part-relaxation}.
		
		\item[\textbf{Scenario 6}]: (Discretization) Finally, to illustrate the potential of discretization, we construct piecewise relaxation of Scenario 3. We discretize the domain of each Underwood root into two partitions, and choose the roots of columns performing the split of the process feed, $\{\theta_{1Nq}\}_{q=1}^{N-1}$, as the partition points. In other words, we let $\mathcal{I}_{ijq}=\{[\alpha_{q+1}, \theta_{1Nq}],[\theta_{1Nq},\alpha_q]\}$ for $i\leq q<j$ and $[i,j]\in \mathcal{S}$. As pointed out in Remark \ref{rem:Feed-Bounds}, these roots can be computed prior to solving the optimization problem. We construct the piecewise relaxation of MINLP \MINLP{} as outlined in \textsection\ref{sec:Discretization}.
	\end{itemize}
	
Table \ref{tab:gap-comparison} reports the percentage gap value, defined as
\begin{equation}
	\text{\% Gap}=100\times\left(1-\frac{\text{Optimal value of relaxation}}{\text{Optimal value of \MINLP}}\right)
	\label{eq:gapdef1}
\end{equation}
on a set of cases evaluated for all the Scenarios. To compare against BARON, we also report \% gap closed (numbers in parenthesis in Table \ref{tab:gap-comparison}), defined as
\begin{align}
	\text{\% Gap Closed} = 100 \times\left( 1-\frac{\text{Optimal value of \MINLP} - \text{Optimal value of relaxation}}{\text{Optimal value of \MINLP} - \text{Optimal value in Scenario 1}}\right)
	\label{eq:gap-closed}
\end{align}
We refer to a particular combination of parameter settings: $N$, $\{F_p\}_{p=1}^N$, $\{\alpha_p\}_{p=1}^N$, $\Phi_{1,N}$ and $\{\Phi_{p,p}\}_{p=1}^N$, as a \textit{case}. The parameter settings for the cases considered in Table \ref{tab:gap-comparison} are listed in the caption. It is worth noting that \texttt{Case-A} \citep{caballero2004}, \texttt{Case-B} and \texttt{Case-C} \citep{nadgir} correspond to physical mixtures: mixture of alcohols, mixture of light paraffins and mixture of light olefins and paraffins. The remaining cases do not directly correspond to physical mixtures, but are representative of specific classes of separations (see \cite{giridhar2010a} for more details). Under Scenario 2, we report \% Gap value, and \% Gap closed for all cases when simultaneous hulls are used to convexify fractions. It can be observed that, this approach closes on an average 45.8\% of the gap. In particular, in \texttt{Case-E}, implementation of simultaneous hull completely closes the gap at root node. Next, under Scenario 3, we report the combined effect of our RDLT approach and simultaneous hulls. This approach closes on an average 74.1\% of the gap.
	Under Scenarios 4 and 5, we report further improvement in relaxation due to addition RDLT cuts discussed in \textsection \ref{sec:RDLT-generalization} to the relaxation in Scenario 3.
	RDLT cuts with quadratic polynomials of Underwood roots closes the gap completely in \texttt{Case-B}. Finally, the gap can be completely closed for all the cases considered in Table \ref{tab:gap-comparison} by discretizing the domain of Underwood root into two partitions, as described in Scenario 6.

\begin{sidewaystable}
	\centering
	\renewcommand\arraystretch{1.5}
	\begin{tabular}{cccccccc}
		\toprule
		&  & \multicolumn{6}{c}{\% Gap as defined in \eqref{eq:gapdef1} (\% Reduced Gap as defined in \eqref{eq:gap-closed})} \\
		\cline{3-8}
		& Optimum & Scenario 1 & Scenario 2 & Scenario 3 & Scenario 4 & Scenario 5 & Scenario 6\\
		\midrule
		\texttt{Case-A} & 402.7 & 31.2\% & 27\% (13.5\%) & 19.6\% (37.2\%) & 13.9\% (55.4) & 15.5\% (50.3\%) & 0\% (100\%) \\
		\texttt{Case-B} & 272.5 & 38.3\% & 13.7\% (64.2\%) & 3.3\% (91.4\%) & 0\% (100\%) & 0.1\% (99.7\%) & 0\% (100\%) \\
		\texttt{Case-C} & 260 & 25.7\% & 17.7\% (31.1\%) & 6\% (76.7\%) &1\% (96.1\%) & 1.1\% (95.7\%) & 0\% (100\%) \\
		\texttt{Case-D} & 896.4 & 45\% & 20.4\% (54.7\%) & 14.9\% (66.9\%) & 7.8\% (82.7\%) & 8.4\% (81.3\%) & 0\% (100\%) \\
		\texttt{Case-E} & 695.6 & 27.7\% & 0\% (100\%) & 0\% (100\%) & 0\% (100\%) & 0\% (100\%) & 0\% (100\%) \\
		\texttt{Case-F} & 929.1 & 32.5\% & 21.4\% (34.2\%) & 4.4\% (86.5\%) & 2.4\% (92.6\%) & 4\% (87.7\%) & 0\% (100\%) \\
		\texttt{Case-G} & 902.7 & 45.8\% & 22.7\% (50.4\%) & 7.4\% (83.8\%) & 1.2\% (97.4\%) & 3.2\% (93\%) & 0\% (100\%) \\
		\texttt{Case-H} & 542 & 27.8\% & 22.7\% (18.3\%) & 13.8\% (50.4\%) & 3.8\% (86.3\%) & 4.4\% (84.2\%) & 0\% (100\%) \\
		\midrule
		Average Gap Closed & & & 45.8\% & 74.1\% & 88.8\% & 86.5\% & 100\%\\
		\bottomrule
	\end{tabular}
	\caption{Variation of duality gap across the scenarios described in \textsection \ref{sec:Scenarios}. Here, a gap value less than $10^{-4}\%$ is marked as $0\%$.  For all the cases, $N=5$, $\Phi_{1,N}=\Phi_{1,1}=\cdots=\Phi_{N,N}=1$. In \texttt{Case-A}, $F = \{20,30,20,20,10\}$ and $\alpha = \{4.1,3.6,2.1,1.42,1\}$; In \texttt{Case-B}, $F=\{5,15,25,20,35\}$ and $\alpha=\{7.98,3.99,3,1.25,1\}$; In \texttt{Case-C}, $F = \{25,10,25,20,20\}$ and $\alpha = \{13.72,3.92,3.267,1.21,1\}$; In \texttt{Case-D}, $F=\{42.5,42.5,5,5\}$ and $\alpha = \{3.3275,3.025,1.21,1.1,1\}$; In \texttt{Case-E}, $F=\{30,30,5,5,30\}$ and $\alpha=\{1.4641,1.331,1.21,1.1,1\}$; In \text{Case-F}, $F=\{5,5,5,42.5,42.5\}$ and $\alpha = \{3.3275,1.331,1.21,1.1,1\}$; In \texttt{Case-G}, $F = \{5,5,5,42.5,42.5\}$ and $\alpha = \{17.1875,6.875,2.75,1.1,1\}$; In \texttt{Case-H}, $F = \{20,20,20,20,20\}$ and $\alpha = \{7.5625,3.025,1.21,1.1,1\}$. Data for \texttt{Case-A} is taken from \cite{caballero2004}, \texttt{Case-B} and \texttt{Case-C} from \citet{nadgir}, and \texttt{Case-D} through \texttt{Case-H} from \citet{giridhar2010a}.}
	\label{tab:gap-comparison}
\end{sidewaystable}

\section{Computational Results}
\label{sec:Comp-Results}
We conducted computational experiments on a \textit{test set} of 496 cases, taken from \citep{giridhar2010a,nallasivam2013}, which is a representative of a majority of separations. Parameter settings for the test set are listed in \textsection \ref{sec:ecom:test-set} in e-companion. In this section, we demonstrate that our proposed approach is able to solve MINLP \MINLP{} within a relative tolerance of 1\%. We also compare the performance of our approach with prior approaches in the literature \citep{caballero2004,nallasivam2016,tumbalamgooty2019}. Since the prior approaches develop an (MI)NLP model, we use BARON 18.5.8 via GAMS 25.1 to solve these (MI)NLPs, where all BARON options are set at their default values. For the adpative partitioning scheme described in Algorithm 1, we use Gurobi 8.0 \citep{gurobi} to solve the resulting MIPs, and use IPOPT \citep{ipoptref} as a local solver. The model is loaded into Gurobi using the MATLAB/Gurobi interface, while IPOPT is used via MATLAB/GAMS interface and GAMS 25.1. We used single CPU thread to solve the MIPs so as to keep the comparison with BARON fair. Besides the setting of number of threads, the remaining options for Gurobi and IPOPT were left at their defaults. All computations were done on a Dell Optiplex 5040 with 16 GB RAM, which has Intel Core i7-6700 3.4 GHz processor and is running 64-bit Windows 7.

\subsection{Comparison with Prior Approaches}
Here, we compare the performance of three approaches, namely those of \cite{caballero2006, tumbalamgooty2019}, and the one proposed here. For all the computations, we set the relative tolerance for convergence ($\epsilon_r$), defined as
\begin{align}
\epsilon_r = \left( 1- \frac{\text{BestLB}}{\text{BestUB}} \right)
\end{align}
where BestLB and BestUB are the best-known relaxation bound and feasible solution, to 1\% i.e., $\epsilon_r=0.01$. We impose a CPU time limit of five hours as the termination criterion.

\begin{itemize}
	\item[\textbf{Approach 1}]: We solve MINLP \MINLP{} using the adaptive partitioning approach described in Algorithm \ref{algo:adaptive}. We begin with four partitions for each Underwood root \ie{} $\mathcal{I}_{ijq}=\{ [\alpha_{q+1},(\alpha_{q+1}+\theta_{1Nq})/2], [(\alpha_{q+1}+\theta_{1Nq})/2,\theta_{1Nq}], [\theta_{1Nq},(\alpha_{q}+\theta_{1Nq})/2], [(\theta_{1Nq}+\alpha_{q})/2,\alpha_{q}] \}$ for every $\llbracket q \rrbracket_{i}^{j-1}$, $[i,j] \in\mathcal{S}$. We compute the Underwood roots for the splits of the process feed $\{\theta_{1Nq}\}_{q=1}^{N-1}$ prior to solving the MINLP (see Remark \ref{rem:Feed-Bounds}). For all but 4 cases, we set $\texttt{MinPrtSize}=10^{-3}$. For the remaining cases, we reduced $\texttt{MinPrtSize}$ to $10^{-4}$ in order to achieve the relative tolerance of 1\%. Finally, we point out that the upper bounds on material flows are computed by solving \eqref{eq:bounding-LPs}, where we choose
	\begin{align}
	\vapduty^* = \max_{q\in\{1,\dots,N-1\}} \sum_{p=1}^q \frac{\alpha_pF_p}{\alpha_p-\theta_{1Nq}}
	\label{eq:FTC-Vapduty}
	\end{align}
	and $\phi=1.5$. We note that \eqref{eq:FTC-Vapduty} is the objective function value corresponding to a feasible point of one of the admissible configurations, commonly known in literature as Fully Thermally Coupled or Petlyuk configuration (see \cite{fidkowski1986,halvorsen2003c}).
	
	\item[\textbf{Approach 2}]: We obtained the GAMS code of the model proposed in \citet{caballero2006} from the MINLP library \citep{cabexample}. There, the authors were interested in identifying the configuration minimizing the total annual cost. For our computations, we modify their code in the following manner. First, as mentioned in \citep{tumbalamgooty2019}, the model of \cite{cabexample} admits solutions that are physically infeasible. This is because the constraints corresponding to \eqref{eq:UWmvr-strip1} in their model should be tight for certain Underwood roots, and their model does not impose this requirement. We have added these missing constraints to their GAMS code. Second, the authors employed the BigM approach in order to transform certain disjunctions into a set of inequalities. Unfortunately, the BigM value used for vapor and liquid bypass in their GAMS code made a few test cases infeasible. Therefore, we specified $2.5\vapduty^*$ as the BigM value for the vapor and liquid bypasses. This number was found by choosing the smallest BigM value for which we found a feasible solution. Third, the authors use a parameter $\epsilon_\theta$ and restrict $\theta_{ijq} \in [\alpha_{q+1}+\epsilon_\theta, \alpha_q-\epsilon_\theta]$ for $\llbracket q \rrbracket_{i}^{j-1}$, $[i,j]\in\mathcal{P}$ in order to avoid the singularity associated with $\theta_{ijq}$ approaching $\alpha_q$ or $\alpha_{q+1}$. Their choice of $\epsilon_\theta$, in some cases, made the optimal solution infeasible. Empirically, we found that $\epsilon_\theta=10^{-4}$ does not cut off the optimal solution, so we set $\epsilon_\theta=10^{-4}$. Fourth, the cost equations required for the evaluation of the objective function were removed from the model, and the objective function was modified to compute the total vapor duty instead. The resulting MINLP is then solved with BARON.
	
	\item[\textbf{Approach 3}]: Here, we consider the MINLP proposed in \cite{tumbalamgooty2019}. For a consistent comparison, we set the upper bound on all vapor flows to be $1.5 \vapduty^*$. Further, we restrict $\theta_{ijq} \in [\alpha_{q+1}+\epsilon_\theta, \alpha_q-\epsilon_\theta]$, where $\epsilon_\theta=10^{-4}$, for $\llbracket q \rrbracket_{i}^{j-1}$, $[i,j]\in\mathcal{S}$ in order to avoid the singularity associated with $\theta_{ijq}$ approaching $\alpha_q$ or $\alpha_{q+1}$. The resulting MINLP is then solved using BARON.
\end{itemize}

\begin{figure}[h]
	\centering
	\subfigure[]{\includegraphics[scale=0.55]{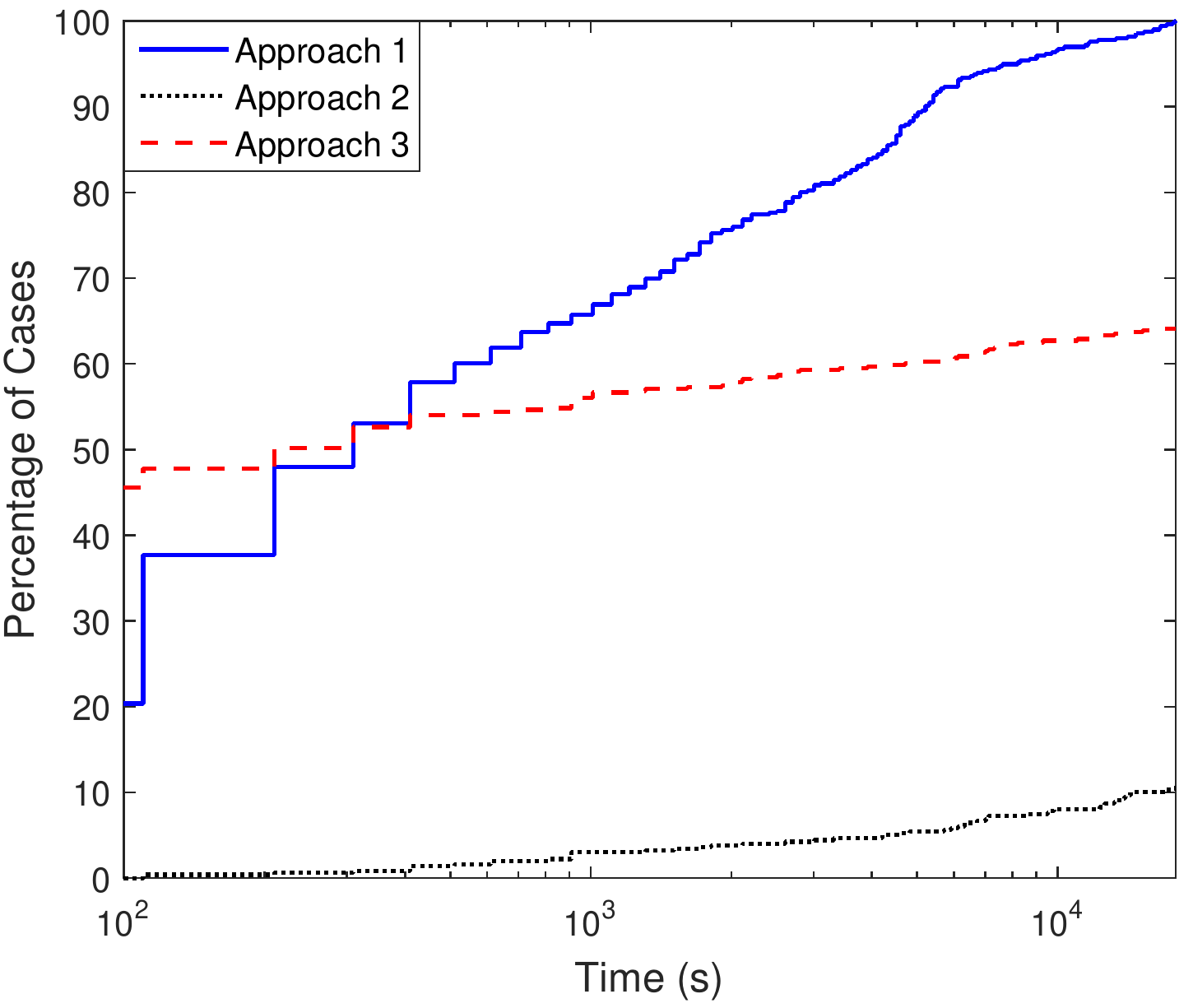}}
	\subfigure[]{\includegraphics[scale=0.55]{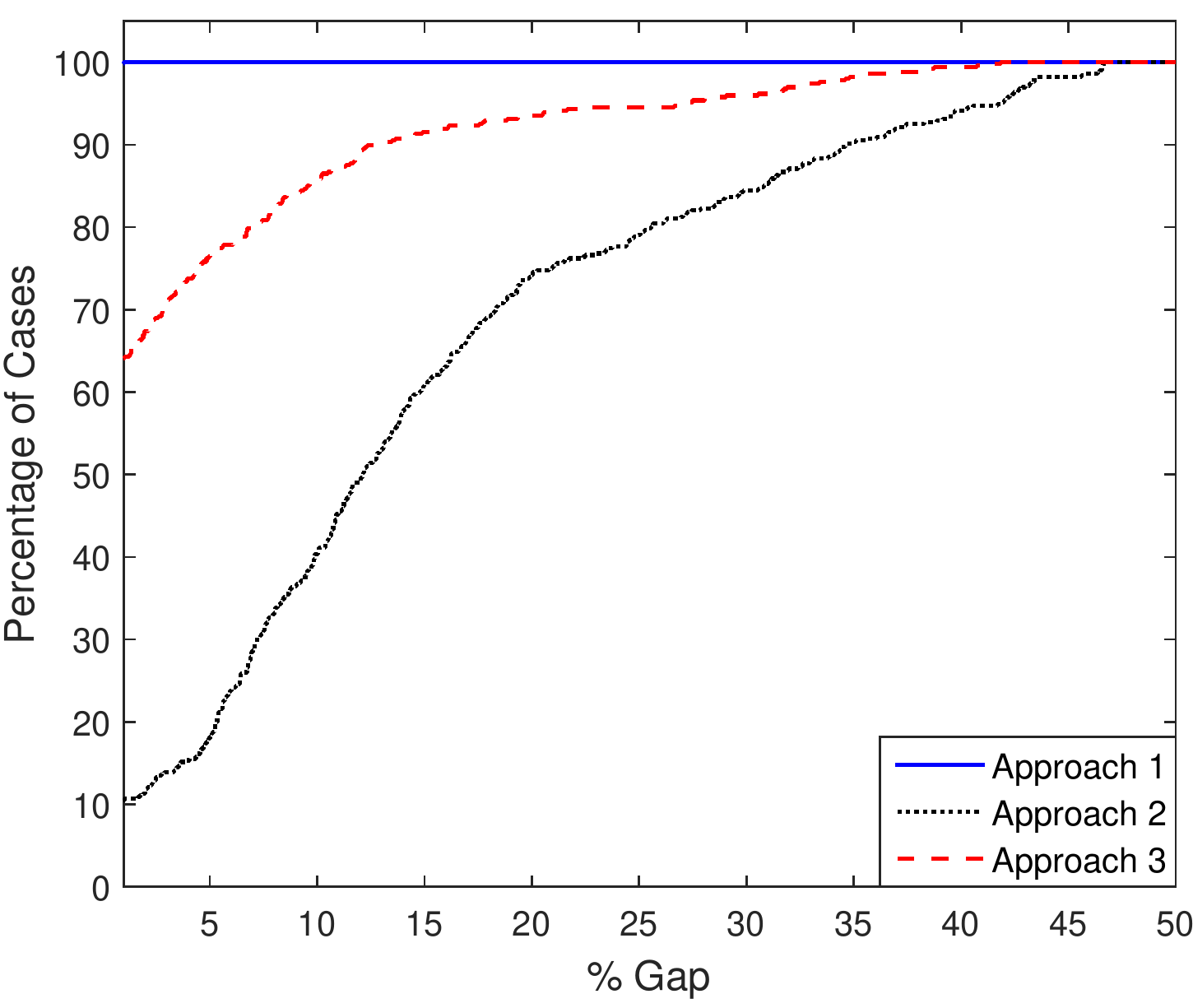}}
	\caption{(a) Plot showing percentage of cases solved to 1\%--optimality against time. Here, Approach 1 corresponds to the current work, Approach 2 corresponds to the model proposed in \citet{caballero2006} solved with BARON, after making the changes described in \textsection \ref{sec:Comp-Results}, and Approach 3 corresponds to the model proposed in \citet{tumbalamgooty2019} solved with BARON  (b) Plot showing the remaining duality Gap at the end of five hours for all the three approaches.}
	\label{fig:MINLPPerfPlot}
\end{figure}

Figure \ref{fig:MINLPPerfPlot}(a) shows the percentage of cases solved to 1\%-optimality against time, with Approach 1 (solid blue curve), Approach 2 (dotted black curve), and Approach 3 (dashed red curve). Observe that Approach 2 solves about 10\% of cases to 1\%-optimality within five hours. This is not surprising because \cite{caballero2004,caballero2006} also reported difficulties in convergence. To overcome the challenges, the authors architected an algorithm by modifying logic-based outer-approximation. While the method resulted in good solutions, optimality was not guaranteed. Approach 3 solves 64\% of the cases in the test set.

We remark that \cite{tumbalamgooty2019} introduced a new search-space formulation, derived cuts that exploit monotonicity of Underwood constraints, and modeled the absence/presence of a column using disjunctions. Nevertheless, this approach fails to solve the problem to 1\%-optimality for 36\% of the cases. The progress of lower bound for a majority of these cases is either stagnant or very slow. Figure \ref{fig:MINLPPerfPlot}(b) depicts the cumulative percentage of cases as a function of the remaining duality gap at the end of five hours. In contrast, our approach, for the first time, solves all 496 cases from this test set within an optimality tolerance of 1\%.


\begin{figure}[h]
	\centering
	\includegraphics[scale=0.55]{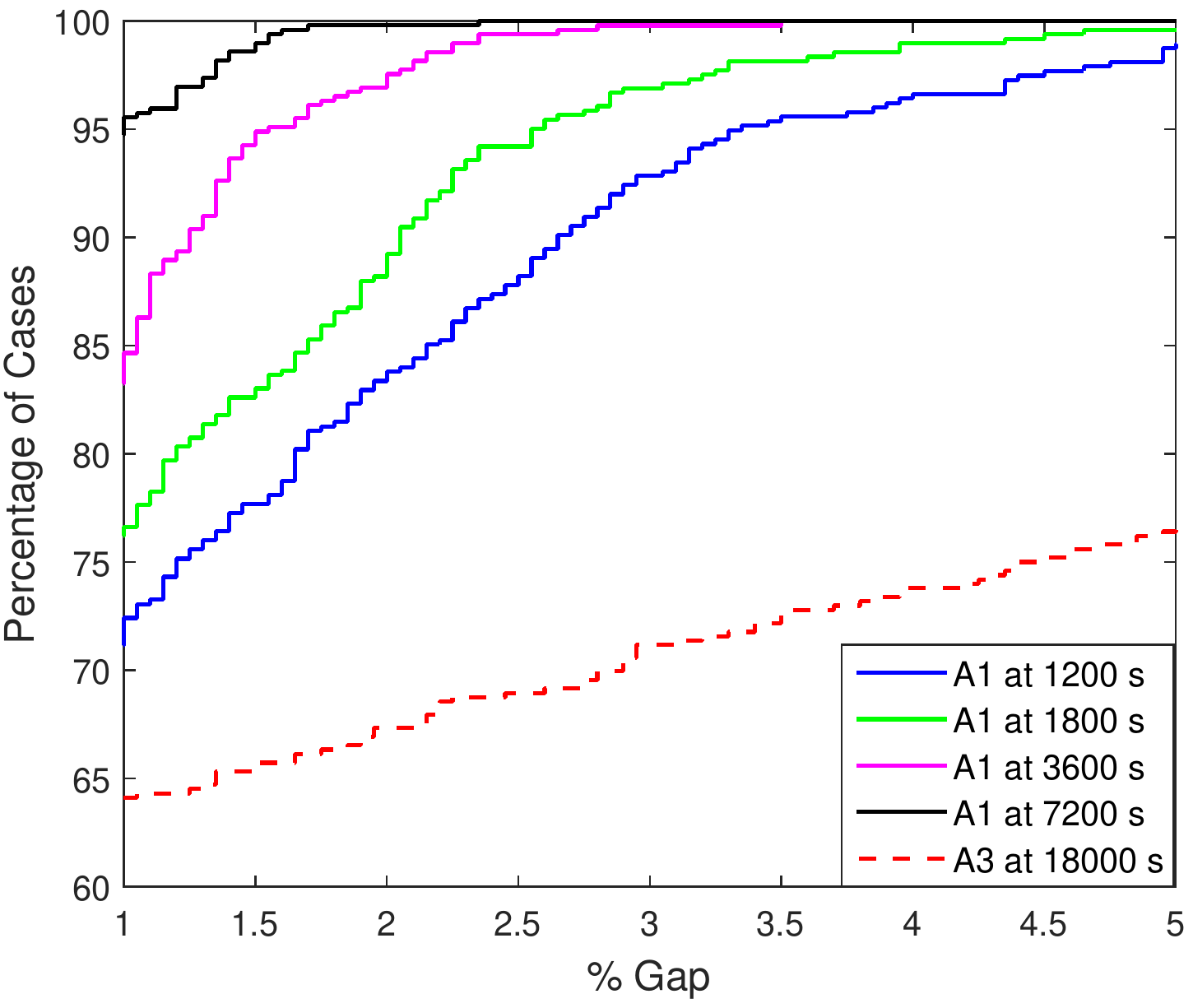}
	\caption{Profiles showing remaining \% Gap at the end of specific time instances for Approach 1 (A1) and Approach 3 (A3).}
	\label{fig:gapplots}
\end{figure}

Figure \ref{fig:gapplots} depicts cumulative percentage of cases as a function of the remaining duality gap at specific time instances for Approach 1. This graph demonstrates that our solution approach, with a CPU time of twenty minutes, already outperforms the best prior MINLP based approach allowed to run for a CPU time of five hours. Further, within 1800 s (green curve), 3600s (magenta curve) and 7200s (black curve), the proposed approach solves all 496 cases to less than 5.5\%, 3.5\% and 2.5\% gap, respectively. Since \MINLP{} is primarily designed as a screening tool for an otherwise highly cumbersome search of optimal distillation configuration, practicing engineers can use Approach 1 to quickly identify near optimal solutions that are worthy of further exploration. Although we do not provide specific configurations found using our procedure, the potential benefits are documented in \cite{shah2010,tumbalamgooty2019} for a crude distillation case study.

\subsection{Comparison with \citet{nallasivam2016}}
Recently, \citet{nallasivam2016} proposed an alternative technique that relies on explicit enumeration for identifying distillation configuration requiring the least vapor duty. After enumerating all the configurations, an NLP is formulated for each configuration and solved to $1\%-$optimality with BARON. We refer to this as \textbf{Approach 4}. We compare the performance of Approach 4, with Approaches 1 and 3 by fixing the discrete decisions to a specific configuration. We choose Fully Thermally Coupled (FTC) configuration, characterized by $\zeta_{i,j}=1\; \forall \; [i,j]\in \mathcal{T}$, $\chi_{i,j}=0 \; \forall \; (i,j)\in \mathcal{C}\setminus\{(1,1)\}$, $\chi_{1,1} = 1$, $\rho_{i,j}=0\; \forall \; (i,j) \in \mathcal{R} \setminus\{(N,N)\}$, and $\rho_{N,N} = 1$, for comparison. This comparison ignores the advances in the search space formulation discussed in \textsection\ref{sec:space_of_admissible_configurations} and other advances that relate Underwood constraints with stream variables, since we fix the binary variables a priori. We set the time limit as one hour and a relative gap of 1\% ($\epsilon_r=0.01$) as termination criteria.

\begin{figure}[h]
	\centering
	\includegraphics[scale=0.55]{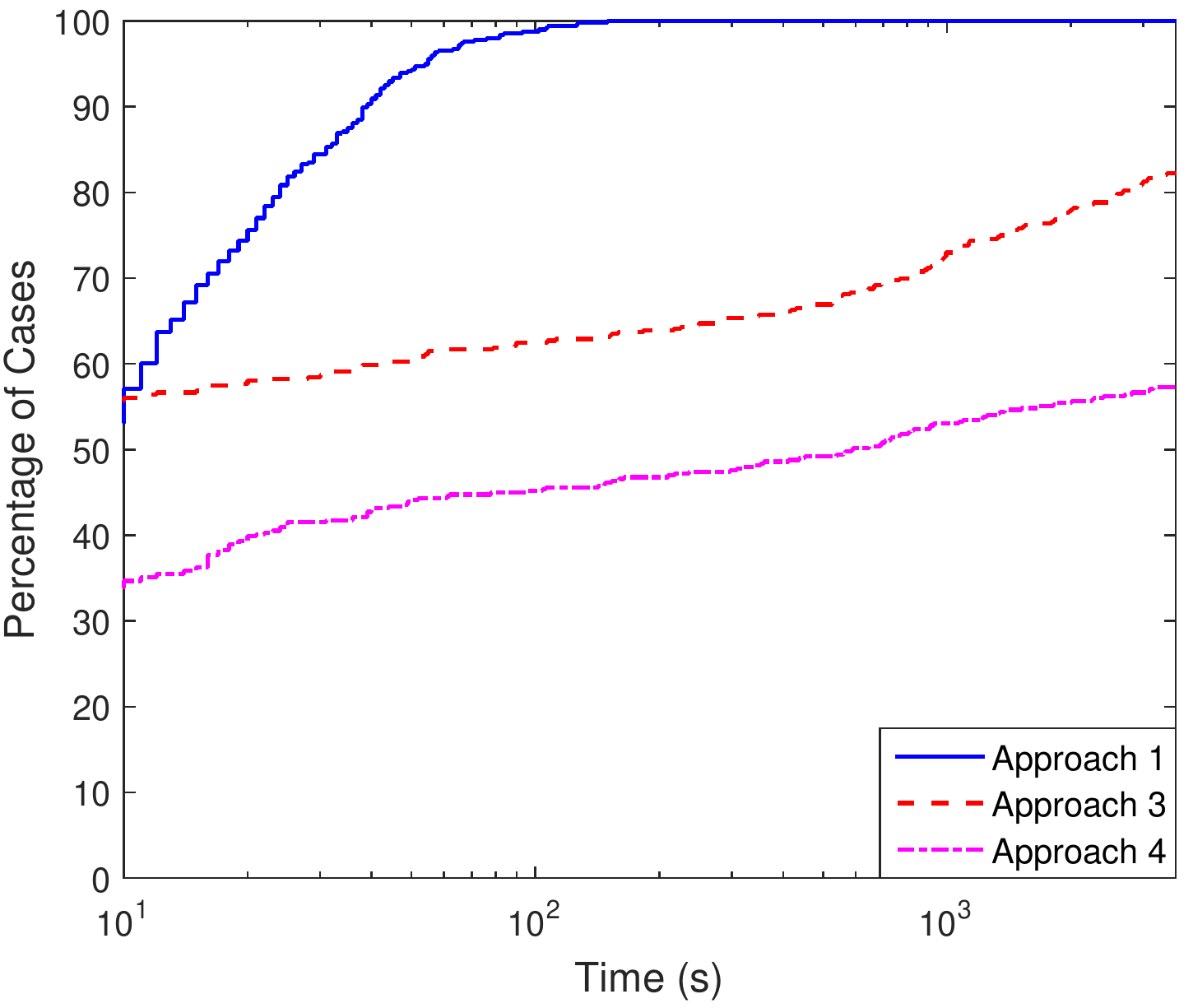}
	\caption{Plot showing percentage of cases solved to 1\%--optimum against time, when discrete variables are fixed to fully thermally coupled configuration (see \textsection \ref{sec:Comp-Results}.2). Approach 4 corresponds to the model proposed in \citet{nallasivam2016} solved with BARON.}
	\label{fig:FTCPerfPlot}
\end{figure}

Figure \ref{fig:FTCPerfPlot} depicts the percentage of cases solved as a function of computational time for the three approaches. Clearly, BARON solves more number of cases to $1\%-$optimality with Approach 3 than with Approach 4.
Despite the improvement, only 82\% of the cases are solved to $1\%-$optimality using Approach 3. In contrast, our approach solves all cases in this test set within 100 s.

\section{Concluding Remarks}
\label{sec:Conclusion}
This work addressed the optimal design of distillation configurations, which are widely used in all chemical and petrochemical industries, and are significant consumers of energy in the world economy. We proposed a novel MINLP that identifies energy-efficient configurations for a given application. Given the complexity from combinatorial explosion of the choice set and nonconvex Underwood constraints, this problem has resisted solution approaches. In this paper, we report on the first successful approach and solve this problem to global optimality for five-component mixtures. The key contributions that make this possible are (i) new formulation for discrete choices that is strictly tighter than the previous formulations, (ii) new valid cuts to the problem using RDLT, and various other convexification results for special structures, and (iii) discretization techniques and an adaptive partitioning scheme to solve the MINLP to $\epsilon-$optimality. On a test set that is a representative of a majority of five-component separations, we demonstrated that our approach solves all the instances in a reasonable amount of time, which was not possible using existing approaches. In summary, this paper describes the first solution approach that can reliably and quickly screen several thousands of alternative distillation configurations and identify solutions that consume less energy and, thereby, lead to less greenhouse gas emissions. This approach has the potential to reduce the carbon footprint and energy usage of thermal separation processes.

\section*{Acknowledgments}
This work is supported by the US Department of Energy (Award number: DE -- EE0005768).

\section*{Disclaimer}
The information, data, or work presented herein was funded in part by an agency of the United States Government. Neither the United States Government nor any agency thereof, nor any of their employees, makes any warranty, express or implied, or assumes any legal liability or responsibility for the accuracy, completeness, or usefulness of any information, apparatus, product, or process disclosed, or represents that its use would not infringe privately owned rights. Reference herein to any specific commercial product, process, or service by trade name, trademark, manufacturer, or otherwise does not necessarily constitute or imply its endorsement, recommendation, or favoring by the United States Government or any agency thereof. The views and opinions of authors expressed herein do not necessarily state or reflect those of the United States Government or any agency thereof.

\bibliographystyle{plainnat}  
\bibliography{references}  
\clearpage

\appendix
\renewcommand{\thepage}{A-\arabic{page}}
\setcounter{page}{1}
\section{Proof of Proposition \ref{prop:bin-simultaneous-hull}}
\label{sec:ecom:acyclic-proof}
Let, $\mathcal{B}=[0,1]^n \times [0,1]^n$. Since the set $S = \{ (x,z) \in \mathcal{B} \; | \; z_j = \prod_{p=1}^{j} x_p, \; j=1,\dots,n \}$ is compact, $\Conv(S)$ is compact and, by Krein-Milman Theorem, is the convex hull of its extreme points. Therefore, we determine the extreme points of $S$, and take their disjunctive union to obtain $\Conv(S)$. When $(x_2,\dots,x_n)$ in $S$ are restricted to $(\overline{x}_2,\dots,\overline{x}_n) \in [0,1]^{n-1}$, then the set $S$ is convex and its extreme points are such that $x_1 \in \{0,1\}$. Let $S_1$ and $\tilde{S}_2$ denote the set $S$ restricted to $x_1=0$ and $x_1=1$, respectively, \ie{} $S_1=\{(x,z)\in \mathcal{B} \; | \; x_1=0, \; z_j=0, \; j=1,\dots,n \}$ and $\tilde{S}_2=\{ (x,z)\in\mathcal{B} \; | \; x_1=z_1=1, \; z_j=\prod_{p=2}^j x_j, \; j=2,\dots,n \}$. Observe that $S_1$ is convex, and $\tilde{S}_2$ is nonconvex. Next, when $(x_3,\dots,x_n)$ in $\tilde{S}_2$ are restricted to $(\overline{x}_3,\dots,\overline{x}_n) \in [0,1]^{n-2}$, then $\tilde{S}_2$ is convex and its extreme points are such that $x_2 \in \{0,1\}$. Let $S_2$ and $\tilde{S}_3$ denote the set $\tilde{S}_2$ restricted to $x_2=0$ and $x_2=1$, respectively, \ie{} $S_2=\{(x,z) \in \mathcal{B} \; | \; x_1=z_1=1, \; x_2=z_2=\dots,z_n=0 \}$ and $\tilde{S}_3 = \{ (x,z)\in \mathcal{B} \; | \; x_1=z_1=x_2=z_2=1, \; z_j = \prod_{p=3}^jx_j, \;j=3,\dots,n \}$. As before, $S_2$ is convex and $\tilde{S}_3$ is nonconvex. Repeating the argument leads to sets $S_3,\dots,S_{n+1}$, where $S_i = \{ (x,z) \in \mathcal{B} \; | \; x_1=z_1=\dots=x_{i-1}=z_{i-1} = 1, \; x_i=z_i=\dots z_n=0 \}$ for $i=3,\dots,n$ and $S_{n+1} = \tilde{S}_{n+1} = \{x_1=z_1=\dots x_n=z_n=1\}$. The sets $S_1$ through $S_{n+1}$ contain the extreme points of convex hull of $S$. Therefore, $\Conv(S) = \Conv(S_1\cup S_2 \cup \dots \cup S_{n+1})$, where $S_1\cup S_2 \cup \dots \cup S_{n+1}$ is given below
\begin{gather*}
\left[ \begin{aligned}
& x_1 = 0 \\
& z_1 = \dots= z_n =0\\
& 0 \le x_j \le 1, \; j = 2,\dots,n
\end{aligned} \right]
\bigvee_{i=2}^n
\left[\begin{aligned}
& x_1=\dots=x_{i-1}=1,\; x_{i}=0\\
& z_1=\dots=z_{i-1}=1\\
& z_{i}=\dots=z_n=0\\
& 0\le x_j \le 1,\; j = i+1,\dots,n
\end{aligned}\right]\bigvee
\left[ \begin{aligned}
& x_1 = \dots = x_n = 1 \\
& z_1 = \dots= z_n =1
\end{aligned} \right].
\end{gather*}
Application of disjunctive programming technique leads to
\begin{gather}
\Conv(S) = \left\{
\begin{aligned}
& x_j^i = \lambda^i, & & \quad \text{for }j = 1,\dots,i-1; & &i = 2,\dots,n+1\\
& x_i^i=0, & & \quad \text{for } & & i = 1,\dots,n\\
& 0 \le x_j^i \le \lambda^i,&  &\quad \text{for } j = i+1,\dots,n;& & i = 1,\dots,n-1\\
& x_j = \sum\nolimits_{i=1}^{n+1} x_j^i,&  &\quad \text{for } j = 1,\dots,n&  &\\
& z_j = \sum\nolimits_{i=j+1}^{n+1} \lambda^i, & & \quad \text{for }j = 1,\dots,n\\
& \sum\nolimits_{i=1}^{n+1} \lambda^i = 1, \; \lambda^i \ge 0 ,&  &\quad \text{for } & & i = 1,\dots,n+1\\
\end{aligned}\right\},
\label{eq:prop1-hull1}
\end{gather}
where $\{ x_j^i \}_{i=1}^{n+1}$ are to be regarded as linearization of $x_j\cdot \lambda^i$. We eliminate $\{ x_j^i \}_{i=j}^{n+1}$ by direct substitution (see \eqref{eq:prop1-hull1}). This leads to $x_j=\sum_{i=1}^{j-1} x_j^i +\sum_{i=j+1}^{n+1} \lambda^i$, or $\sum_{i=1}^{j-1}x_j^i \le x_j - \sum_{i=j+1}^{n+1} \lambda^i \le \sum_{i=1}^{j-1}x_j^i $, where $\{x_j^i\}_{i=1}^{j-1}$ are constrained by $0 \le x_j^i \le \lambda^i$. Now, using Fourier-Motzkin elimination, we eliminate $\{x_j^i\}_{i=1}^{j-1}$ to obtain $0 \le x_j - \sum_{i=j+1}^{n+1} \lambda^i \le \sum_{i=1}^{j-1} \lambda^i$, or $\sum_{i=j+1}^{n+1} \lambda^i \le x_j \le \sum_{i=1}^{j-1} \lambda^i + \sum_{i=j+1}^{n+1} \lambda^i =1-\lambda^j$. This transforms \eqref{eq:prop1-hull1} to
\begin{align}
\Conv(S) = \left\{
\begin{aligned}
& \sum\nolimits_{i=j+1}^{n+1} \lambda^i \le x_j \le 1-\lambda^j,&  &\quad \text{for } j = 1,\dots,n\\
& z_j = \sum\nolimits_{i=j+1}^{n+1} \lambda^i, & & \quad \text{for }j = 1,\dots,n\\
& \sum\nolimits_{i=1}^{n+1} \lambda^i = 1, \; \lambda^i \ge 0 ,&  &\quad \text{for } i = 1,\dots,n+1\\
\end{aligned}\right\}.
\label{eq:prop1-hull2}
\end{align}
Next, we determine $\lambda^i$ in terms of $z_j$. From $z_j=\sum_{i=j+1}^{n+1} \lambda^i$ for $j=1,\dots,n$ and $\sum_{i=1}^{n+1} \lambda^i = 1$, $z_n=\lambda^{n+1}$, $z_{n-1}=\lambda^n+\lambda^{n+1}$ or $z_{n-1}-z_n=\lambda^n$, $z_{n-2}=\lambda^{n-1}+\lambda^n+\lambda^{n+1}$ or $z_{n-2}-z_{n-1}=\lambda^{n-1},\dots,z_1-z_2=\lambda^2$, and $\lambda^1=1-\sum_{i=2}^{n+1} \lambda^i=1-z_1$. Using these relations, we eliminate $\lambda^i$ variables from \eqref{eq:prop1-hull2} to obtain
\begin{align}
\Conv(S) = \left\{
\begin{aligned}
& z_1 \le x_1 \le z_1\\
& z_j \le x_j \le 1-z_{j-1}+z_{j},&  &\quad \text{for } j = 2,\dots,n\\
& z_n\ge 0, \; (1-z_1)\ge 0,\\
& z_{j-1}-z_j \ge 0, & & \quad\text{for }j = 2,\dots,n
\end{aligned}\right\}.
\label{eq:prop1-hull3}
\end{align}
Observe that the same set of inequalities result from recursive McCormick relaxation of $z_j = z_{j-1} \cdot x_j$ for $j = 2,\dots,n$. Therefore, the convex hull of set $S$ can be constructed by a recursive application of McCormick procedure on $z_j = z_{j-1} \cdot x_j$, $j = 2,\dots,n$. \qed

\section{Proof of Remark \ref{rem:linearization-tb}}
\label{sec:ecom:tau-hull-proof}
We show the proof for $\tau_{i,k,j}$ variables, and the proof for $\beta_{i,l,j}$ variables is similar. By Remark \ref{rem:linearization}, the convex hull of $\nu_{i,k,j}=\prod_{n=k}^j (1-\zeta_{i,n})$ over  $(\zeta_{i,k},\dots,\zeta_{i,j})\in [0,1]^{j-k+1}$, given by
\begin{subequations}
	\begin{align}
	& \nu_{i,k,j} \ge \max\{0,-\zeta_{i,k}-\dots-\zeta_{i,j}+1\} \label{eq:nu-hull1}\\
	& \nu_{i,k,j} \le \min\{1-\zeta_{i,k},\dots,1-\zeta_{i,j}\},\label{eq:nu-hull2}
	\end{align}
\end{subequations}
is implied from \eqref{eq:mccor}, for every $[i,j] \in \mathcal{P}$, $\llbracket k \rrbracket_{i}^{j-1}$. We use the above inequalities, in addition to \eqref{eq:secvar} and \eqref{eq:mccor}, for the proof. We consider two cases: $k+1<j$ and $k+1=j$. When, $k+1 <j$, the convex hull of $\tau_{i,k,j}=\zeta_{i,k}(1-\zeta_{i,k+1})\dots(1-\zeta_{i,j-1})\zeta_{i,j}$ over $(\zeta_{i,k},\dots,\zeta_{i,j})\in [0,1]^{j-k+1}$ is given by \citep{crama1993concave}
\begin{subequations}
	\begin{align}
	& \tau_{i,k,j} \ge 0, \label{eq:tau-hull-1}\\
	& \tau_{i,k,j} \ge \zeta_{i,k}-\zeta_{i,k+1}-\dots-\zeta_{i,j-1}+\zeta_{i,j}-1, \label{eq:tau-hull-2}\\
	& \tau_{i,k,j} \le \zeta_{i,k}, \label{eq:tau-hull-3}\\
	& \tau_{i,k,j} \le 1-\zeta_{i,n}, \quad \llbracket n \rrbracket_{k+1}^{j-1}, \label{eq:tau-hull-4}\\
	& \tau_{i,k,j} \le \zeta_{i,j}. \label{eq:tau-hull-5}
	\end{align}
\end{subequations}
On the other hand, when $k+1=j$, the convex hull of $\tau_{i,k,j}=\tau_{i,j-1,j}=\zeta_{i,j-1}\zeta_{i,j}$ over $(\zeta_{i,j-1,\zeta_{i,j}})\in[0,1]^2$ is given by
\begin{subequations}
	\begin{align}
	& \tau_{i,j-1,j} \ge \max\{0,\zeta_{i,j-1}+\zeta_{i,j}-1\},\\
	& \tau_{i,j-1,j} \le \min\{\zeta_{i,j-1},\zeta_{i,j}\}.
	\end{align}
\end{subequations}
In the following, we present the proof only for $k+1<j$, and point out that the proof for the case $k+1=j$ is similar. \\

\noindent\eqref{eq:tau-hull-1}: From \eqref{eq:mccor}, $\nu_{i,k,j-1}+\nu_{i,k+1,j}-\nu_{i,k+1,j-1} \le \nu_{i,k,j} \implies 0 \le \nu_{i,k+1,j-1}-\nu_{i,k,j-1}-\nu_{i,k+1,j}+\nu_{i,k,j} \overset{\eqref{eq:secvar}}{=}\tau_{i,k,j}$.\\

\noindent\eqref{eq:tau-hull-2}: $\tau_{i,k,j} \overset{\eqref{eq:secvar}}{=} \nu_{i,k+1,j-1}-\nu_{i,k,j-1}-\nu_{i,k+1,j}+\nu_{i,k,j} \stackrel{\eqref{eq:nu-hull1}}{\ge} -\zeta_{i,k+1}-\dots-\zeta_{i,j-1}+1-\nu_{i,k,j-1}-\nu_{i,k+1,j}\stackrel{\eqref{eq:nu-hull2}}{\ge}-\zeta_{i,k+1}-\dots-\zeta_{i,j-1}+1 -(1-\zeta_{i,k})-(1-\zeta_{i,j-1})=\zeta_{i,k}-\zeta_{i,k+1}-\dots-\zeta_{i,j-1}+\zeta_{i,j}-1$.\\

\noindent\eqref{eq:tau-hull-3}: $\tau_{i,k,j} \overset{\eqref{eq:secvar}}{=} \nu_{i,k+1,j-1}-\nu_{i,k,j-1}-\nu_{i,k+1,j}+\nu_{i,k,j} \stackrel[\nu_{i,k,j}\le \nu_{i,k+1,j}]{\eqref{eq:mccor}}{\le } \nu_{i,k+1,j-1}-\nu_{i,k,j-1} \stackrel[\nu_{i,k,j-1}\ge \nu_{i,k+1,j-1}+\nu_{i,k,k}-1]{\eqref{eq:mccor}}{\le } 1-\nu_{i,k,k} \overset{\eqref{eq:mccor}}{=}\zeta_{i,k}$.\\

\noindent\eqref{eq:tau-hull-4}: $\tau_{i,k,j} \overset{\eqref{eq:secvar}}{=} \nu_{i,k+1,j-1}-\nu_{i,k,j-1}-\nu_{i,k+1,j}+\nu_{i,k,j} \stackrel[\nu_{i,k,j}\le \nu_{i,k+1,j}]{\eqref{eq:mccor}}{\le } \nu_{i,k+1,j-1}-\nu_{i,k,j-1} \overset{\eqref{eq:nu-hull1}}{\le} \nu_{i,k+1,j-1} \overset{\eqref{eq:nu-hull2}}{\le} 1-\zeta_{i,n}, \; \text{for}\; k+1\le n\le j-1$.\\

\noindent\eqref{eq:tau-hull-5}: $\tau_{i,k,j} \overset{\eqref{eq:secvar}}{=} \nu_{i,k+1,j-1}-\nu_{i,k,j-1}-\nu_{i,k+1,j}+\nu_{i,k,j} \stackrel[\nu_{i,k,j}\le \nu_{i,k,j-1}]{\eqref{eq:mccor}}{\le } \nu_{i,k+1,j-1}-\nu_{i,k+1,j} \stackrel[\nu_{i,k+1,j}\ge \nu_{i,k+1,j-1}+\nu_{i,j,j}-1]{\eqref{eq:mccor}}{\le } 1-\nu_{i,j,j} \overset{\eqref{eq:mccor}}{=}\zeta_{i,j}$.

\section{Proof of Proposition \ref{prop:projection-parent-network}}
\label{sec:ecom:projection-proof}
\begin{definition}
	Let, $\mathcal{D} = (V,A)$ be a digraph and $b \in \mathbb{R}^{|V|}$. A function $f:A\rightarrow \mathbb{R}$ is called as $b-$\textit{transshipment} if $\text{excess}_f(v_i) \coloneqq f\left\langle \delta^{\text{in}}(v_i) \right\rangle - f\left\langle \delta^{\text{out}}(v_i) \right\rangle = b(v_i)$ $\forall \; v_i \in V$, where $\delta^\text{in}(v_i) \subseteq A$ (\resp $\delta^\text{out}(v_i) \subseteq A$) is the set of all arcs entering (\resp leaving) the vertex $v_i$, and $f\left\langle \delta(v_i) \right\rangle \coloneqq \sum\limits_{a\in \delta(v_i)} f(a)$. In our case, the function $f(a)$ evaluates the flow along the arc $a$.
\end{definition}

\begin{lemma}[\cite{rado1943}]
	\label{lemma:Rado}
	Let $\mathcal{D}=(V,A)$ be a digraph, and let $b:V\rightarrow \mathbb{R}$ with $b\, \langle V \rangle =0$. Then there exists a b-transshipment $f\ge \textbf{0}$ if and only if $b\, \langle U \rangle \le 0$ for each $U \subseteq V$ with $\delta^\textnormal{in}(U) = \emptyset $.
\end{lemma}

We now use Lemma \ref{lemma:Rado} to prove Proposition \ref{prop:projection-parent-network}.

Consider the digraph $\mathcal{D} = (V,A)$, where $V = D_6 \cup D_7$ and $A = (D_6 \times D_7)\setminus \{ (N+1,0) \}$ (see \textsection \ref{sec:parent-presence} and Figure \ref{fig:ParentConstraints} for definition of $D_6$ and $D_7$). We have discarded the arc from $N+1 \in D_6$ to $0\in D_7$, because the flow along that arc is zero (see \eqref{eq:Parent-network}). Observe that, for every $n \in D_6$, $b(n) = \text{excess}_\psi (n)= -\sum_{m \in D_7} \psi_{i,n,m,j}= -\tau_{i,j,n}$ (see Figure \ref{fig:ParentConstraints}). Similarly, for every $m \in D_7$, $b(m) = \text{excess}_\psi(m) = \sum_{n \in D_6} \psi_{i,n,m,j} = \beta_{m,i,j}$. Then, $b\, \langle V \rangle = \sum_{n=j+1}^{N+1} b(n) + \sum_{m=0}^{i-1} b(m) = -\sum_{n=j+1}^{N+1} \tau_{i,j,n} + \sum_{m=0}^{i-1} \beta_{m,i,j}  = 0$ (from definition of $S_{i,j}$). From Lemma \ref{lemma:Rado}, a $b-$transshipment $\psi \ge \textbf{0}$ exists if and only if $b\, \langle U \rangle \le 0$ for each $U \subseteq V$ with $\delta^\textnormal{in}(U) = \emptyset $. For every $U \subseteq D_6 \subset V$, $b \, \langle U \rangle \le 0$ is satisfied trivially. On the other hand, $U$ cannot be chosen to be a subset of $D_7$, because for every $U \subseteq D_7$, $\delta^\text{in} (U) \ne \emptyset$. Therefore, in order to derive non-trivial inequalities, we must choose subsets of $V$ containing vertices of both $D_6$ and $D_7$.

Let $U = (D_6\setminus\{N+1\}) \cup \{0\}$. Note that $\delta^\text{in} (U) = \emptyset$. Then, a $b-$transshipment $\psi \ge \textbf{0}$ exists if and only if $b \, \langle U \rangle = -\sum_{n=j+1}^N \tau_{i,j,n} + \beta_{0,i,j} \le 0$, or
\begin{equation}
\beta_{0,i,j} \le \sum_{n=j+1}^N \tau_{i,j,n}.
\label{eq:Prop6-proof-1}
\end{equation}

It can be verified that for every other subset $U \subseteq V$ satisfying $\delta^\text{in} (U) = \emptyset$, the inequality ensuring $b \, \langle U \rangle \le 0$ is implied from $\sum_{m=0}^{i-1} \beta_{m,i,j} = \sum_{n=j+1}^{N+1} \tau_{i,j,n}$. Therefore,
\begin{align}
\text{proj}_{(\tau,\beta)}(S_{i,j}) = \left\{ \eqref{eq:Prop6-proof-1}; \; \sum_{m=0}^{i-1} \beta_{m,i,j} = \sum_{n=j+1}^{N+1} \tau_{i,j,n}; \; \tau_{i,j,n} \ge 0, \; \llbracket n \rrbracket_{j+1}^{N+1}; \; \beta_{m,i,j} \ge 0, \; \llbracket m \rrbracket_{0}^{i-1} \right\}.
\label{eq:ecom:network-proj}
\end{align}
Indeed, $\psi_{i,n,m,j}$ can be defined to verify that \eqref{eq:ecom:network-proj} is the projection of $S_{i,j}$.
\begin{itemize}
	\item[\textbf{Def1}:] Define $\psi_{i,N+1,0,j} = 0$.
	\item[\textbf{Def2}:] For $1 \le m \le i-1$, define
	\begin{align*}
	\psi_{i,N+1,m,j} = \begin{cases}
	\tau_{i,j,N+1} \cdot \frac{\beta_{m,i,j}}{\sum_{m=1}^{i-1} \beta_{m,i,j}}, & \text{ if } \sum_{m=1}^{i-1} \beta_{m,i,j}>0\\
	0, &  \text{ if } \sum_{m=1}^{i-1} \beta_{m,i,j} = 0.
	\end{cases}
	\end{align*}
	Since $\sum_{m=0}^{i-1} \beta_{m,i,j} = \sum_{n=j+1}^{N+1} \tau_{i,j,n}$, \eqref{eq:Prop6-proof-1} implies $\tau_{i,j,N+1} \le \sum_{m=1}^{i-1} \beta_{m,i,j} $. Then, the above definition guarantees that $\psi_{i,N+1,m,j} \le \beta_{m,i,j}$ for every $1 \le m \le i-1$, and $\sum_{m=1}^{i-1} \psi_{i,N+1,m,j} = \tau_{i,j,N+1}$.
	
	\item[\textbf{Def3}:] For $j+1 \le n \le N$, define
	\begin{align*}
	\psi_{i,n,0,j} = \begin{cases}
	\beta_{0,i,j} \cdot \frac{\tau_{i,j,n}}{\sum_{n=j+1}^N \tau_{i,j,n}}, & \text{ if } \sum_{n=j+1}^N \tau_{i,j,n} > 0\\
	0, & \text{ if } \sum_{n=j+1}^N \tau_{i,j,n}=0.
	\end{cases}
	\end{align*}
	Because $\beta_{0,i,j} \le \sum_{n=j+1}^N \tau_{i,j,n}$ (see \eqref{eq:Prop6-proof-1}), the above definition guarantees that $\psi_{i,n,0,j} \le \tau_{i,j,n}$ for every $j+1 \le n \le N$, and $\sum_{n=j+1}^N \psi_{i,n,0,j} = \beta_{0,i,j}$.
	
	\item[\textbf{Def4}:] For every $1 \le m \le i-1$ and $j+1 \le n \le N$, define
	\begin{align*}
	\psi_{i,n,m,j} = \begin{cases}
	\frac{(\tau_{i,j,n}-\psi_{i,n,0,j}) \cdot (\beta_{m,i,j}-\psi_{i,N+1,m,j}) }{\sum_{n=j+1}^N (\tau_{i,j,n}-\psi_{i,n,0,j}) }, & \text{ if } \sum_{n=j+1}^N (\tau_{i,j,n}-\psi_{i,n,0,j})\\
	0, & \text{ if } \sum_{n=j+1}^N (\tau_{i,j,n}-\psi_{i,n,0,j}) = 0.
	\end{cases}
	\end{align*}
	Since $(\beta_{m,i,j}-\psi_{i,N+1,m,j}) \ge 0$ (see \textbf{Def2}) and $(\tau_{i,j,n}-\psi_{i,n,0,j}) \ge 0$ (see \textbf{Def3}), the above definition guarantees $\psi_{i,n,m,j} \ge 0$ for every $1 \le m \le i-1$ and $j+1 \le n \le N$. Next, it can be shown that $\sum_{n=j+1}^N (\tau_{i,j,n}-\psi_{i,n,0,j}) = \sum_{m=1}^{i-1} (\beta_{m,i,j}-\psi_{i,N+1,m,j})$ from $\sum_{n=j+1}^{N+1} \tau_{i,j,n} = \sum_{m=0}^{i-1} \beta_{m,i,j}$, $\sum_{m=1}^{i-1} \psi_{i,N+1,m,j} = \tau_{i,j,N+1}$ (see \textbf{Def2}), and $\sum_{n=j+1}^N \psi_{i,n,0,j} = \beta_{0,i,j}$ (see \textbf{Def3}). Then, the above definition guarantees that $\sum_{n=j+1}^N \psi_{i,n,m,j} = \beta_{m,i,j}-\psi_{i,N+1,m,j}$ and $\sum_{m=1}^{i-1} =\tau_{i,j,n}-\psi_{i,n,0,j} $.
\end{itemize}

\clearpage
\section{Proof of Proposition \ref{prop:searchspace}}
\label{sec:ecom:search-space-proof}
In addition to binary variables associated with the presence/absence condensers and reboilers, CG06 has variables for the presence of heat exchanger, which we denote as $\eta_{i,j}$. To our model, we add
\begin{equation}
\eta_{i,j} = \chi_{i,j} + \rho_{i,j}.
\label{eq:CGM-Hex}
\end{equation}
Further, we remark that for $i \le k \le j-1$, $[i,j] \in \mathcal{P}$,
\begin{align}
\sum_{m=i}^{k} \tau_{i,m,j} \stackrel{\eqref{eq:RLTbin}}{=} \sum_{m=i}^{k} \sum_{l=i+1}^{m+1} \sigma_{i,m,l,j} \stackrel{\text{Fig. \ref{fig:SplitFeasibility}}}{=} \sum_{l=i+1}^{k+1} \sum_{m=l-1}^k \sigma_{i,m,l,j} \stackrel{k \le j-1}{\le} \sum_{l=i+1}^{k+1} \sum_{m=l-1}^{j-1} \sigma_{i,m,l,j} \stackrel{\eqref{eq:RLTbin}}{=} \sum_{l=i+1}^{k+1} \beta_{i,l,j}.
\label{eq:ss-tightness-proof}
\end{align}
In Tables \ref{tab:GAmodel-proof}, \ref{tab:CGModel-proof} and \ref{tab:TATModel-proof}, we prove that the set defined by \eqref{eq:secvar}--\eqref{eq:Presence-of-parent}, $\zeta_{i,j} \in [0,1], \; \forall \; [i,j]\in\mathcal{T}$, $\rho_{i,j}\in[0,1], \; \forall \; (i,j) \in \mathcal{R}$ and $\chi_{i,j}\in[0,1], \; \forall \; (i,j)\in \mathcal{C}$ is tighter than CG06, GA10 and TAT19, respectively. We point out that, in GA06, the authors did not consider thermally coupled configurations. Thus, we show the proof only for the constraints they reported.

Next, we show \textit{strict tightness} with a numerical example. Consider $N=4$:
\begin{enumerate}
	\item When restricted to $\zeta_{1,2}= \zeta_{1,3}= 0$, $\zeta_{1,1} = \zeta_{1,4} = \zeta_{2,2} = \zeta_{3,3} = \zeta_{4,4} = 1$ and $\zeta_{2,3}=\zeta_{2,4}=\zeta_{3,4} =1/2$, CG06 is feasible, while \MINLP{} is infeasible.
	\item The point $\tau_{1,1,3} = \tau_{1,2,3} = \tau_{1,1,4} = \tau_{1,3,4} = \tau_{2,2,4} = \beta_{1,2,3} = \beta_{1,3,3} = \beta_{1,3,4} = \beta_{2,4,4} = 0$; $\tau_{1,1,2}=\tau_{1,2,4}=\beta_{1,2,2} = 1$ and $\tau_{2,2,3}=\tau_{2,3,4}=\tau_{3,3,4}= \beta_{1,2,4} = \beta_{1,4,4} = \beta_{2,3,3} = \beta_{2,3,4} = \beta_{3,4,4} = 1/2$ is an extreme point to GA10, and infeasible to \MINLP{}.
	\item When restricted to $\zeta_{3,4} = 0$, $\zeta_{1,1} = \zeta_{1,2} = \zeta_{1,4} = \zeta_{2,2} = \zeta_{3,3} = \zeta_{4,4} = 1$ and $\zeta_{1,3} = \zeta_{2,3} = \zeta_{2,4} = 1/2$, TAT19 is feasible, while \MINLP{} is infeasible. \qed
\end{enumerate}

\begin{table}
	\centering
	\renewcommand{\arraystretch}{1.5}
	\begin{tabular}{l@{\hspace{0.5cm}}l@{\hspace{0.5cm}}}
		\toprule
		\# & Proof\\
		\midrule
		$[1]$ & $\begin{aligned}
		\sum_{k=i}^{j-1}\sum_{l=i+1}^{k+1} \sigma_{i,k,l,j} \stackrel{\text{\eqref{eq:RLTbin}, Co.\ref{prop:sumteqz}}}{=} \zeta_{i,j} \le 1
		\end{aligned}$\\
		\midrule
		$[2]$ & $\begin{aligned}
		& \sum_{n=j+1}^N\sum_{m=i+1}^{j+1} \sigma_{i,j,m,n} \stackrel{\text{\eqref{eq:RLTbin}}}{=} \sum_{n=j+1}^N \tau_{i,j,n} \stackrel{\text{Pr.\ref{prop:conv-hull-PresParent}}}{\le} \zeta_{i,j} \le 1 \\
		& \sum_{m=1}^{i-1}\sum_{n=i-1}^{j-1} \sigma_{m,n,i,j} \stackrel{\text{\eqref{eq:RLTbin}}}{=} \sum_{m=1}^{i-1} \beta_{m,i,j} \stackrel{\text{Pr.\ref{prop:conv-hull-PresParent}}}{\le} \zeta_{i,j} \le 1
		\end{aligned}$\\
		\midrule
		$[3]$ & $\begin{aligned}
		\sum_{m=1}^{i-1} \sigma_{m,i-1,i,i} + \sum_{n=i+1}^{N} \sigma_{i,i,i+1,n} \stackrel{\text{\eqref{eq:RLTbin}}}{=} \sum_{m=1}^{i-1} \beta_{m,i,i} + \sum_{n=i+1}^{N} \tau_{i,i,n} \stackrel{\text{\eqref{eq:Presence-of-parent}}}{\ge} \zeta_{i,i} \stackrel{\text{\eqref{eq:feedprod}}}{\ge} 1,\quad \text{for} \quad [i,i] \in \mathcal{T}
		\end{aligned}$\\
		\midrule
		$[6]$ & $ \begin{aligned}
		& \left.\sigma_{i,j,m,n} \stackrel{\eqref{eq:RLTbin}}{\le} \tau_{i,j,n} \stackrel{\text{Re.\ref{rem:linearization-tb}}}{\le} \zeta_{i,j} \stackrel{\text{Co.\ref{prop:sumteqz}, \eqref{eq:RLTbin}}}{=} \sum_{k=i}^{j-1}\sum_{l=i+1}^{k+1} \sigma_{i,k,l,j} \right\} & & \text{for}\quad \llbracket m \rrbracket_{i+1}^{j+1}; \quad \llbracket n \rrbracket_{j+1}^N\\
		& \left.\sigma_{m,n,i,j} \stackrel{\eqref{eq:RLTbin}}{\le} \beta_{m,i,j} \stackrel{\text{Re.\ref{rem:linearization-tb}}}{\le} \zeta_{i,j} \stackrel{\text{Co.\ref{prop:sumteqz}, \eqref{eq:RLTbin}}}{=}  \sum_{k=i}^{j-1}\sum_{l=i+1}^{k+1} \sigma_{i,k,l,j} \right\} & & \text{for}\quad \llbracket n \rrbracket_{i-1}^{j-1}; \quad \llbracket m \rrbracket_{1}^{i-1}
		\end{aligned} $\\
		\midrule
		$[6]$ & $\begin{aligned} \left.
		\sigma_{i,k,l,j} \stackrel{\text{Pr.\ref{prop:conv-hull-SplitFeas}}}{\le} \zeta_{i,j}  \stackrel{\text{\eqref{eq:Presence-of-parent}}}{\le}  \sum_{n=i+1}^{N} \tau_{i,j,n} + \sum_{m=1}^{i-1} \beta_{m,i,j} \stackrel{\text{\eqref{eq:RLTbin}}}{=} \sum_{n=j+1}^N\sum_{m=i+1}^{j+1} \sigma_{i,j,m,n} + \sum_{m=1}^{i-1}\sum_{n=i-1}^{j-1} \sigma_{m,n,i,j} \right\} \\
		\text{for}\quad \llbracket l \rrbracket_{i+1}^{k+1};\quad \llbracket k \rrbracket_{i}^{j-1}
		\end{aligned}$\\
		\midrule
		$[12]$ & $\begin{aligned}
		\sigma_{i,k,l,j} \stackrel{\text{Pr.\ref{prop:conv-hull-SplitFeas}}}{\le} \zeta_{i,j}, \quad \text{for}\quad \llbracket l \rrbracket_{i+1}^{k+1};\quad \llbracket k \rrbracket_{i}^{j-1}
		\end{aligned}$\\
		\midrule
		$[13]$ & $\begin{aligned}
		\sum_{k=i}^{j-1}\sum\limits_{l=i+1}^{k+1} \sigma_{i,k,l,j} \stackrel{\text{\eqref{eq:RLTbin}, Co.\ref{prop:sumteqz}}}{\ge} \zeta_{i,j}
		\end{aligned} $\\
		\midrule
		$[4]$  &  $\begin{aligned}
		& \left. 1-\eta_{i,i} \stackrel[\rho_{i,i} \ge 0]{\eqref{eq:CGM-Hex}}{\le} 1- \chi_{i,i} \stackrel{\eqref{eq:hex2}}{\le} 1- \beta_{0,i,i} \stackrel{\eqref{eq:feedprod},\text{ Pr.\ref{prop:conv-hull-parents}}}{=} \sum_{m=1}^{i-1} \beta_{m,i,i} \stackrel{\eqref{eq:RLTbin}}{=} \sum_{m=1}^{i-1} \sigma_{m,i-1,i,i} \right\}\quad
		\text{for} \quad (i,i) \in \mathcal{C} \\
		& \left. 1-\eta_{i,i} \stackrel[\chi_{i,i} \ge 0]{\eqref{eq:CGM-Hex}}{\le} 1- \rho_{i,i} \stackrel{\eqref{eq:hex2}}{\le} 1- \tau_{i,i,N+1} \stackrel{\eqref{eq:feedprod},\text{ Pr.\ref{prop:conv-hull-parents}}}{=} \sum_{n=i+1}^{N} \tau_{i,i,n} \stackrel{\eqref{eq:RLTbin}}{=} \sum_{n=i+1}^{N} \sigma_{i,i,i+1,N} \right\}\quad
		\text{for} \quad (i,i) \in \mathcal{R} \\
		\end{aligned}$\\
		\midrule
		$[5]$  & $ \begin{aligned}
		\eta_{i,i} \stackrel{\eqref{eq:CGM-Hex}}{=} \chi_{i,j} + \rho_{i,j} \stackrel{\eqref{eq:hex1},\text{ Pr.\ref{prop:conv-hull-parents}}}{\le} \zeta_{i,j} - \sum_{m=1}^{i-1} \beta_{m,i,j} + \zeta_{i,j} - \sum_{n=j+1}^N \tau_{i,j,n} \stackrel[\zeta_{i,j} \le 1]{\beta,\tau \ge 0}{\le} (1-\beta_{m,i,j}) + (1-\tau_{i,j,n'}) \\
		\stackrel{\text{Pr.\ref{prop:conv-hull-SplitFeas}}}{\le} (1-\sigma_{m,n,i,j}) + (1-\sigma_{i,j,m',n'}) \quad \text{for} \quad \llbracket m' \rrbracket_{i+1}^{j+1};\quad \llbracket n' \rrbracket_{j+1}^N;\quad \llbracket m \rrbracket_{1}^{i-1};\quad \llbracket n \rrbracket_{i-1}^{j-1}
		\end{aligned} $\\
		\midrule
		$[7]$ & $ \begin{aligned}
		\eta_{i,j} \stackrel{\eqref{eq:CGM-Hex}}{=} \chi_{i,j} + \rho_{i,j} \stackrel{\eqref{eq:hex1},\text{ Pr.\ref{prop:conv-hull-parents}},\; \eqref{eq:Presence-of-parent}}{\le}  \zeta_{i,j} \stackrel{\eqref{eq:Presence-of-parent},\;\eqref{eq:RLTbin}}{\le}   \sum_{n=j+1}^N\sum_{m=i+1}^{j+1} \sigma_{i,j,m,n} + \sum_{m=1}^{i-1}\sum_{n=i-1}^{j-1} \sigma_{m,n,i,j}
		\end{aligned}$\\
		\bottomrule
	\end{tabular}
	\caption{CG06 for the space of admissible configurations. The first column indicates the constraint number in Table 1 of \cite{caballero2006}. `Co.', `Re.' and `Pr.' stand for Corollary, Remark and Proposition, respectively.}
	\label{tab:CGModel-proof}
\end{table}

\begin{table}[h]
	\centering
	\renewcommand{\arraystretch}{1.3}
	\begin{tabular}{l@{\hspace{0.5cm}}l@{\hspace{0.5cm}}}
		\toprule
		\textsection & Proof\\
		\midrule
		3.1 & $\begin{aligned}
		\sum_{l=i+1}^{j} \beta_{i,l,j} = \sum_{k=i}^{j-1} \tau_{i,k,j} \stackrel{\text{Co.\ref{prop:sumteqz}}}{=} \zeta_{i,j} \le 1
		\end{aligned}$\\
		\midrule
		3.1& $\begin{aligned}
		\sum_{n=j+1}^{N} \tau_{i,j,n} \stackrel{\text{Pr.\ref{prop:conv-hull-PresParent}}}{\le} \zeta_{i,j} \stackrel{\text{Co.\ref{prop:sumteqz}}}{=} \sum_{k=i}^{j-1} \tau_{i,k,j}\\
		\sum_{m=1}^{i-1} \beta_{m,i,j} \stackrel{\text{Pr.\ref{prop:conv-hull-PresParent}}}{\le} \zeta_{i,j} \stackrel{\text{Co.\ref{prop:sumteqz}}}{=} \sum_{k=i}^{j-1} \tau_{i,k,j}
		\end{aligned} $\\
		\midrule
		3.1 & $\begin{aligned}
		\sum_{k=i}^{j-1} \tau_{i,k,j} \stackrel{\text{Co.\ref{prop:sumteqz}}}{=} \zeta_{i,j} \stackrel{\eqref{eq:Presence-of-parent}}{\le} \sum_{n=j+1}^{N} \tau_{i,j,n} + \sum_{m=1}^{i-1} \beta_{m,i,j}
		\end{aligned}$\\
		\midrule
		3.2 & $\begin{aligned}
		& \max \left\{ (j-i+1)\sum_{m=1}^{i-1} \beta_{m,i,j}, \; (j-i+1) \sum_{n=j+1}^N \tau_{i,j,n} \right\} \stackrel{\text{Pr.\ref{prop:conv-hull-PresParent}}}{\le} (j-i+1) \zeta_{i,j} \stackrel{\text{Co.\ref{prop:sumteqz}}}{=} \\
		& (j-i+1) \sum_{k=i}^{j-1} \tau_{i,k,j} = \sum_{k=i}^{j-1} (j-k) \tau_{i,k,j} + \sum_{k=i}^{j-1} (k-i+1)\tau_{i,k,j} = \sum_{k=i}^{j-1} \sum_{m=i}^k \tau_{i,m,j}  + \\
		& \sum_{k=i}^{j-1} (k-i+1)\tau_{i,k,j} \stackrel{\eqref{eq:ss-tightness-proof}}{\le} \sum_{k=i}^{j-1} \sum_{l=i+1}^{k+1} \beta_{i,l,j}  + \sum_{k=i}^{j-1} (k-i+1) \tau_{i,k,j} = \sum_{l=i+1}^j \sum_{k=l-1}^{j-1} \beta_{i,l,j} + \\
		& \sum_{k=i}^{j-1} (k-i+1) \tau_{i,k,j} = \sum_{l=i+1}^j (j-l+1) \beta_{i,l,j} + \sum_{k=i}^{j-1} (k-i+1) \tau_{i,k,j}
		\end{aligned}$\\
		\midrule
		3.3 & $\begin{aligned}
		& \sum_{n=j+1}^N \tau_{i,j,n} \stackrel{\text{Pr.\ref{prop:conv-hull-PresParent}}}{\le} \zeta_{i,j} \le 1\\
		& \sum_{m=1}^{i-1} \beta_{m,i,j} \stackrel{\text{Pr.\ref{prop:conv-hull-PresParent}}}{\le} \zeta_{i,j} \le 1\\
		\end{aligned}$\\
		\midrule
		3.3 & $\begin{aligned}
		\sum_{n=i+1}^{N} \tau_{i,i,n} + \sum_{m=1}^{i-1} \beta_{m,i,i} \stackrel{\eqref{eq:Presence-of-parent}}{\ge} \zeta_{i,i} \stackrel{\eqref{eq:feedprod}}{=} 1
		\end{aligned}$\\
		\bottomrule
	\end{tabular}
	\caption{GA10 for space of admissible configurations. The first column indicates the section number in \cite{giridhar2010b}. `Co.', `Re.' and `Pr.' stand for Corollary, Remark and Proposition, respectively.}
	\label{tab:GAmodel-proof}
\end{table}

\begin{table}[h]
	\centering
	\renewcommand{\arraystretch}{1.5}
	\begin{tabular}{l@{\hspace{0.5cm}}l@{\hspace{0.5cm}}}
		\toprule
		\# & Proof\\
		\midrule
		(H2) & $\zeta_{1,N} \stackrel{\eqref{eq:feedprod}}{=} \zeta_{i,i} \stackrel{\eqref{eq:feedprod}}{=} 1$\\
		\midrule
		(H3) & $\begin{aligned}
		\zeta_{i,j} \stackrel{\eqref{eq:Presence-of-parent}}{\le} \sum_{n=j+1}^{N} \tau_{i,j,n} + \sum_{m=1}^{i-1} \beta_{m,i,j} \stackrel{\text{Pr.\ref{prop:conv-hull-SplitFeas}}}{\le} \sum_{n=j+1}^{N} \zeta_{i,n}+ \sum_{m=1}^{i-1} \zeta_{m,j}
		\end{aligned} $  \\
		\midrule
		(H4) & $\begin{aligned}
		& \left. \zeta_{i,k} -\sum_{n=k+1}^{j-1} \zeta_{i,n} + \zeta_{i,j} -1 \stackrel{\text{Re.\ref{rem:linearization-tb}}}{\le} \tau_{i,k,j} \stackrel{\text{Pr.\ref{prop:conv-hull-SplitFeas}}}{\le} \sum_{l=i+1}^{k+1} \beta_{i,l,j} \stackrel{\text{Re.\ref{rem:linearization-tb}}}{\le} \sum_{l=i+1}^{k+1} \zeta_{l,j} \right\} \quad \text{for} \quad \llbracket k \rrbracket_{i}^{j-1}
		\end{aligned}$\\
		\midrule
		(H5) & $\begin{aligned}
		& \left. \zeta_{i,j} - \sum_{m=i+1}^{l-1} \zeta_{m,j} + \zeta_{l,j} -1 \stackrel{\text{Re.\ref{rem:linearization-tb}}}{\le} \beta_{i,l,j} \stackrel{\text{Pr.\ref{prop:conv-hull-SplitFeas}}}{\le} \sum_{k=l-1}^{j-1} \tau_{i,k,j} \stackrel{\text{Re.\ref{rem:linearization-tb}}}{\le} \sum_{k=l-1}^{j-1} \zeta_{i,k} \right\} \quad \text{for} \quad \llbracket l \rrbracket_{i+1}^{j}
		\end{aligned}$\\
		\midrule
		(H6) & $\begin{aligned}
		\chi_{i,j} + \rho_{i,j} \stackrel{\eqref{eq:hex1},\text{ Pr. \ref{prop:conv-hull-parents}}}{\le} \zeta_{i,j} - \sum_{m=1}^{i-1} \beta_{m,i,j} + \zeta_{i,j} - \sum_{n=j+1}^N \tau_{i,j,n} \stackrel{\eqref{eq:Presence-of-parent}}{\le} \zeta_{i,j}
		\end{aligned}$\\
		\midrule
		(H7) & $\begin{aligned}
		& \chi_{i,j} \stackrel{\eqref{eq:hex1},\text{ Pr. \ref{prop:conv-hull-parents}}}{\le} \zeta_{i,j} - \sum_{m=1}^{i-1} \beta_{m,i,j} \stackrel{\eqref{eq:Presence-of-parent}}{\le} \sum_{n=j+1}^N \tau_{i,j,n} \stackrel{\text{Re.\ref{rem:linearization-tb}}}{\le} \sum_{n=j+1}^N \zeta_{i,n}\\
		& \rho_{i,j} \stackrel{\eqref{eq:hex1},\text{ Pr. \ref{prop:conv-hull-parents}}}{\le} \zeta_{i,j} - \sum_{n=j+1}^N \tau_{i,j,n} \stackrel{\eqref{eq:Presence-of-parent}}{\le}  \sum_{m=1}^{i-1} \beta_{m,i,j} \stackrel{\text{Re.\ref{rem:linearization-tb}}}{\le} \sum_{m=1}^{i-1} \zeta_{m,j}
		\end{aligned}$\\
		\midrule
		(H8) & $\begin{aligned}
		\chi_{i,j} + \rho_{i,j} \stackrel{\eqref{eq:hex1}}{\le} \nu_{i,j+1,N} - \nu_{i,j,N} + \omega_{1,i-1,j} - \omega_{1,i,j} \stackrel{\eqref{eq:mccor}}{\le} (1-\zeta_{i,n}) + (1-\zeta_{m,j}),\\
		\text{for} \quad \llbracket m \rrbracket_{1}^{i-1};\; \llbracket n \rrbracket_{j+1}^N
		\end{aligned} $\\
		\midrule
		$\begin{gathered}
		\text{(H9)}\\
		\&\\
		\text{(H10)}
		\end{gathered}$ & $\begin{aligned}
		& \left. \chi_{i,i} \stackrel{\eqref{eq:hex2},\text{ Pr. \ref{prop:conv-hull-parents}}}{\ge} \zeta_{i,i} - \sum_{m=1}^{i-1} \beta_{m,i,i} \stackrel{\text{Re.\ref{rem:linearization-tb},\;\eqref{eq:feedprod}}}{\ge} 1 - \sum_{m=1}^{i-1} \zeta_{m,i} \right\} \quad \text{for} \quad (i,i) \in \mathcal{C}\\
		& \left. \rho_{i,i} \stackrel{\eqref{eq:hex2},\text{ Pr. \ref{prop:conv-hull-parents}}}{\ge} \zeta_{i,i} - \sum_{n=i+1}^{N} \tau_{i,i,n} \stackrel{\text{Re.\ref{rem:linearization-tb},\;\eqref{eq:feedprod}}}{\ge} 1 - \sum_{n=i+1}^{N} \zeta_{i,n} \right\} \quad \text{for} \quad (i,i) \in \mathcal{R}
		\end{aligned}$\\
		\bottomrule
	\end{tabular}
	\caption{TAT19 for the space of admissible configurations. The first column indicates the constraint number in \citet{tumbalamgooty2019}. `Re.' and `Pr.' stand for Remark and Proposition, respectively.}
	\label{tab:TATModel-proof}
\end{table}
\clearpage

\section{Derivation of $\Conv(\mathcal{F}_p)$}
\label{sec:conv-hull-Fp}
Let $X \coloneqq \{(f_p^\feed,f_p^\rec,f_p^\strip) \in [0,F_p]^3 \; | \; f_p^\feed = f_p^\rec+f_p^\strip \}$. Then, the extreme points of the polytope $X$ are $v^1=(0,0,0)$, $v^2=(F_p,F_p,0)$ and $v^3=(F_p,0,F_p)$. From Proposition \ref{prop:poly-simul-hull}, the convex hull of $\mathcal{F}_p$  is obtained as $\Conv(\mathcal{F}_p)=\text{proj}_{(f,\theta,H,\underline{f\theta})}\{\eqref{eq:conv-Fp}\}$, where
\begin{subequations}
	\label{eq:conv-Fp}
	\begin{align}
	& w^i \ge T_p^*(\lambda^i,\theta^i), \hspace{5.8cm} i = 1,2,3\\
	& w^i \le \lambda^iT_p(\theta^\lobnd)+ \left[ \frac{T_p(\theta^\upbnd)-T_p(\theta^\lobnd)}{\theta^\upbnd-\theta^\lobnd} \right] (\theta^i-\lambda^i \theta^\lobnd), \quad i = 1,2,3\\
	&\lambda^i \theta^\lobnd \le \theta^i \le \lambda^i \theta^\upbnd, \hspace{5.35cm} i=1,2,3\\
	& H_p^\feed = F_p w^2+ F_p w^3, \quad H_p^\rec = F_p w^2, \quad H_p^\strip = F_p w^3, \label{eq:flowhull-Hbal}\\
	& \underline{f\theta}_p^\feed = F_p \theta^2 + F_p \theta^3, \quad \underline{f\theta}_p^\rec = F_p \theta^2, \quad \underline{f\theta}_p^\strip = F_p \theta^3, \label{eq:flowhull-fthbal}\\
	& f_p^\feed = F_p \lambda^2 + F_p \lambda^3, \quad f_p^\rec = F_p \lambda^2,\quad f_p^\strip = F_p \lambda^3, \label{eq:flowhull-flowbal}\\
	& w = w^1 + w^2 + w^3, \quad \theta = \theta^1 + \theta^2 + \theta^3,\\
	& \lambda^1+\lambda^2+\lambda^3 = 1, \quad \lambda^1,\lambda^2,\lambda^3 \ge 0.
	\end{align}
\end{subequations}
We solve linear equations and obtain auxiliary variables in terms of problem variables as $(\lambda^2,\theta^2,w^2) = (f_p^\rec/F_p, \underline{f\theta}_p^\rec/F_p, H_p^\rec/F_p)$, $(\lambda^3,\theta^3,w^3) = (f_p^\strip/F_p, \underline{f\theta}_p^\strip/F_p, H_p^\strip/F_p)$, $(\lambda^1,\theta^1,w^1) = (1-\lambda^2-\lambda^3, \theta-\theta^2-\theta^3, w-w^2-w^3)  = ((F_p-f_p^\feed)/F_p,(F_p\theta-\underline{f\theta}_p^\feed)/F_p,(F_pw-H_p^\feed)/F_p)$ (from first equation in \eqref{eq:flowhull-Hbal},\eqref{eq:flowhull-fthbal}, and \eqref{eq:flowhull-flowbal}). Using these relations, all variables can be eliminated from the hull description, except $w$, which is constrained by
\begin{align*}
T_p^*(\lambda^1,\theta^1) \le \frac{F_pw-H_p^\feed}{F_p}  \le \lambda^1 T_p(\theta^\lobnd) + \left(\frac{T_p(\theta^\upbnd)-T_p(\theta^\lobnd)}{\theta^\upbnd-\theta^\lobnd}\right)(\theta^1-\lambda^1\theta^\lobnd).
\end{align*}
We eliminate $w$ using Fourier-Motzkin elimination to obtain $T_p^*(\lambda^1,\theta^1) \le \lambda^1 T_p(\theta^\lobnd) + \left[\frac{T_p(\theta^\upbnd)-T_p(\theta^\lobnd)}{\theta^\upbnd-\theta^\lobnd}\right](\theta^1-\lambda^1\theta^\lobnd)$. The resulting constraint is redundant, so we do not impose it explicitly. This leads to the convex hull description described in \textsection \ref{sec:part-relaxation}.

\section{Proof of Proposition \ref{prop:simultaneous-hull-xy}}
\label{sec:ecom:xy-hull-proof}
When $x$ is restricted to $\overline{x} \in X$, the set $S = \{(x,y,\underline{xy})\in \mathcal{D} \;|\; \underline{xy} = \overline{x} \cdot y, \; x=\overline{x}  \}$ can be expressed as an affine transformation of $y^\lobnd \le y \le y^\upbnd$, whose extreme points are $y\in \{y^\lobnd,y^\upbnd\}$. Therefore, the extreme points of convex hull of $S$ are contained in the set of points where $y \in\{y^\lobnd,y^\upbnd\} $. Let $S^1=\{(x,y,\underline{xy}) \; | \; \underline{xy} = y^\lobnd x, \; Bx \le b, \; y=y^\lobnd \}$ and $S^2=\{(x,y,\underline{xy}) \; | \; \underline{xy} = y^\upbnd x, \; Bx \le b, \; y=y^\lobnd \}$. Then, by Krein-Milman theorem, convex hull of $S$ is obtained by taking the disjunctive union of $S^1$ and $S^2$, i.e., $\Conv(S) = \text{proj}_{(x,y,\underline{xy})} \{ (x,y,\underline{xy},x^1,x^2,\lambda^1,\lambda^2) \; | \;  Bx^i \le b\lambda^i, \; i=1,2,\; \eqref{eq:proof:hull-xy}, \; \lambda^1 \ge 0, \; \lambda^2 \ge 0  \} $, where
\begin{subequations}
	\begin{align}
	& \underline{xy} = x^1 y^\lobnd + x^2 y^\upbnd, \quad x = x^1 + x^2,\\
	& y = y^\lobnd \lambda^1 + y^\upbnd \lambda^2, \quad \lambda^1+\lambda^2 = 1,
	\end{align}
	\label{eq:proof:hull-xy}
\end{subequations}
Solving the above equations leads to
\begin{align}
x^1= \frac{y^\upbnd x-\underline{xy}}{y^\upbnd-y^\lobnd}, \quad x^2 = \frac{\underline{xy}-y^\lobnd x}{y^\upbnd-y^\lobnd}, \quad \lambda^1 = \frac{y^\upbnd-y}{y^\upbnd-y^\lobnd}, \quad \lambda^2 = \frac{y-y^\lobnd}{y^\upbnd-y^\lobnd}.
\end{align}
Using the above relations, we substitute out $x^1$, $x^2$, $\lambda^1$ and $\lambda^2$ to obtain the convex hull description in the proposition. \qed

\section{Proof of Corollary \ref{prop:simultaneous-hull}}
\label{sec:ecom:hull-of-xg-xy}
Here, $x$ lies in the polytope $x^\lobnd \le x \le x^\upbnd$, whose extreme points are $x^\lobnd$ and $x^\upbnd$. Application of Proposition \ref{prop:poly-simul-hull} yields $\Conv(S) = \text{proj}_{(x,y,z,\underline{xy})} \{ (x,y,z,\underline{xy},w,y^1,y^2,w^1,w^2,\lambda^1,\lambda^2)  \} $
\begin{subequations}
	\begin{align}
	& w^1 \ge g^*(\lambda^1,y^1), \quad w^2 \ge g^*(\lambda^2,y^2),\\
	& w^i \le \lambda^i g(y^\lobnd) + \left[\frac{g(y^\upbnd)-g(y^\lobnd)}{y^\upbnd-y^\lobnd}\right](y^i-\lambda^iy^\lobnd),\quad i = 1,2,\\
	& \lambda^1y^\lobnd \le y^1 \le \lambda^1 y^\upbnd, \quad \lambda^2y^\lobnd \le y^2 \le \lambda^2 y^\upbnd,\\
	& z = x^\lobnd w^1 + x^\upbnd w^2, \quad \underline{xy} = x^\lobnd y^1+x^\upbnd y^2, \quad w = w^1+w^2,\\
	& y = y^1+y^2, \quad x=x^\lobnd\lambda^1+x^\upbnd \lambda^2, \lambda^1+\lambda^2 =1, \; \lambda^1,\lambda^2 \ge 0.
	\end{align}
\end{subequations}
We remove the equality $w=w^1+w^2$ to project out $w$. Solving the linear equations yields $\lambda^1 = (x^\upbnd-x)/(x^\upbnd-x^\lobnd)$, $\lambda^2 = (x-x^\lobnd)/(x^\upbnd-x^\lobnd)$, $y^1 = (x^\upbnd y-\underline{xy})/(x^\upbnd-x^\lobnd)$ and $y^2 = (\underline{xy}-x^\lobnd y)/(x^\upbnd-x^\lobnd)$. Using these equations, we substitute out auxiliary variables $y^1$, $y^2$, $\lambda^1$ and $\lambda^2$. Finally, eliminating variables $w^1$ and $w^2$ using Fourier-Motzkin elimination yields the convex hull description in the Proposition. The outer-approximation of the convex hull follows directly from Remark \ref{rem:outer-approx}. \qed

\section{Proof of Proposition \ref{prop:hull-of-H1}}
\label{sec:proof-of-H1}
We assume w.l.o.g that
\begin{equation}\label{eq:thetalo_assume}
\theta^\lobnd \le \alpha_1-F_1/H_1^\upbnd.
\end{equation}
Otherwise, we update $F_1=H_1^\upbnd(\alpha_1-\theta^\lobnd)$ (See Figure \ref{fig:H1hull}).

We begin by determining the extreme points of the convex hull of $\mathcal{H}_1$. When $\theta$ is restricted to $\overline{\theta} \in [\theta^\lobnd,\alpha_1]$, the set $\mathcal{H}_1=\{(f_1,\theta,H_1,\underline{f\theta}_1) \mid \; 0 \le f_1 \le \min\{F_1,H_1^\upbnd(\alpha_1-\overline{\theta})\}, \; \theta=\overline{\theta}, \; \underline{f\theta}_1 =f_1 \cdot \overline{\theta}, H_1 = \{f_1/(\alpha_1-\overline{\theta}), \text{ if }\overline{\theta} < \alpha_1; \; H_1 \in [0,H_1^\upbnd], \text{ if } \overline{\theta}=\alpha_1  \}   \}$ can be expressed as an affine transform of $0 \le f_1 \le \min\{ F_1, \; H_1^\upbnd(\alpha_1-\overline{\theta}) \}$ whose extreme points are $f_1 \in \{0,\min\{ F_1, \; H_1^\upbnd(\alpha_1-\overline{\theta}) \}\}$. Therefore, the extreme points of $\Conv(\mathcal{H}_1)$ are contained in the set of points where $f_1=0$, or $f_1 = F_1$ and $\theta^\lobnd \le \theta \le (\alpha_1-F_1/H_1^\upbnd)$, or $f_1=H_1^\upbnd(\alpha_1-\theta)$ and $(\alpha_1-F_1/H_1^\upbnd) \le \theta \le \alpha_1$ (see Figure \ref{fig:H1hull}). Let,
\begin{enumerate}
	\item $S^a$ be $\mathcal{H}_1$ restricted to $f_1=0$ \ie{} $S^a=\{ (f_1,\theta,H_1,\underline{f\theta}_1) \mid f_1 = 0,\; \theta^\lobnd \le \theta \le \alpha_1,\; \underline{f\theta}_1=0, \; H_1= 0 \text{ if }\theta<\alpha_1; \; H_1\in [0,H_1^\upbnd],\text{ if }\theta=\alpha_1 \}$ (see Figure \ref{fig:H1hull}).
	\item $S^b$ be $\mathcal{H}_1$ restricted to $f_1=H_1^\upbnd(\alpha_1-\theta)$ and $(\alpha_1-F_1/H_1^\upbnd) \le \theta \le \alpha_1$ \ie{} $S^b=\{ (f_1,\theta,H_1,\underline{f\theta}_1) \mid f_1=H_1^\upbnd(\alpha_1-\theta), \; (\alpha_1-F_1/H_1^\upbnd) \le \theta \le \alpha_1, \; H_1 = H_1^\upbnd, \; \underline{f\theta}_1 = H_1^\upbnd (\alpha_1-\theta)\theta \}$ (see Figure \ref{fig:H1hull}).
	\item $S^c$ be $\mathcal{H}_1$ restricted to $f_1 = F_1$ and $\theta^\lobnd \le \theta \le (\alpha_1-F_1/H_1^\upbnd)$ \ie{} $S^c = \{(f_1,\theta,H_1,\underline{f\theta}_1) \; |\; f_1=F_1, \; \theta^\lobnd \le \theta \le (\alpha_1-F_1/H_1^\upbnd), \; H_1 =F_1/(\alpha_1-\theta), \; \underline{f\theta}_1 = F_1 \cdot \theta \}$ (see Figure \ref{fig:H1hull}).
\end{enumerate}
\begin{figure}
	\centering
	\includegraphics[scale=0.9]{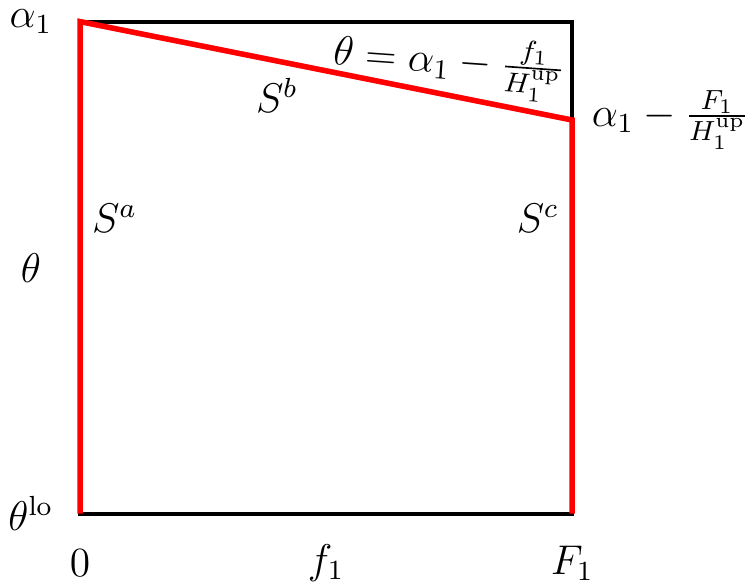}
	\caption{$(f_1,\theta)$ domain for \textsection \ref{sec:proof-of-H1}. The extreme points of $\Conv(\mathcal{H}_1)$ are contained in points in red. }
	\label{fig:H1hull}
\end{figure}

By Krein-Milman theorem, $\Conv(\mathcal{H}_1) = \Conv(S^a \cup S^b \cup S^c) = \Conv(\Conv(S^a)\cup \Conv(S^b)\cup \Conv(S^c))$, where
\begin{align}
& \Conv(S^a) = \left\{(f_1,\theta,H_1,\underline{f\theta}_1)\left|\begin{aligned}
& f_1 = 0, \; \underline{f\theta}_1 =0 \\
& 0 \le H_1 \le H_1^\upbnd\left( \frac{\theta-\theta^\lobnd}{\alpha_1-\theta^\lobnd} \right)\\
& \theta^\lobnd \le \theta \le \alpha_1
\end{aligned}\right.\right\}, \\
& \Conv(S^b) = \left\{(f_1,\theta,H_1,\underline{f\theta}_1)\left|\begin{aligned}
& f_1 = H_1^\upbnd(\alpha_1-\theta), H_1 = H_1^\upbnd\\
& H_1^\upbnd \left(\alpha_1-\frac{F_1}{H_1^\upbnd}\right) (\alpha_1-\theta) \le \underline{f\theta}_1 \\
& \underline{f\theta}_1 \le H_1^\upbnd\theta (\alpha_1-\theta)\\
& \left(\alpha_1-\frac{F_1}{H_1^\upbnd}\right) \le \theta\le  \alpha_1
\end{aligned}\right.\right\},\\
& \Conv(S^c) = \left\{(f_1,\theta,H_1,\underline{f\theta}_1)\left|\begin{aligned}
& f_1 = F_1, \; \underline{f\theta}_1=F_1 \cdot \theta\\
& \frac{F_1}{\alpha_1-\theta} \le H_1\\
& H_1 \le \frac{F_1}{\alpha_1-\theta^\lobnd} + \frac{H_1^\upbnd}{\alpha_1-\theta^\lobnd} (\theta-\theta^\lobnd)\\
& \theta^\lobnd \le \theta \le \left(\alpha_1-\frac{F_1}{H_1^\upbnd}\right)
\end{aligned}\right.\right\}.
\end{align}
Disjunctive union of $\Conv(S^a)$, $\Conv(S^b)$ and $\Conv(S^c)$ leads to \eqref{eq:prop:hull-of-H1}. \qed

\section{Relaxation of \eqref{eq:prop:hull-of-H1}}\label{sec:relax:hull-of-H1}
Since \eqref{eq:prop:hull-of-H1} introduces many variables, we derive a relaxation of $\Conv(\mathcal{H}_1)$ instead.
Let, $\overline{\theta}^r \in [\theta^\lobnd,\alpha_1]$ for $r=1,\dots,R$. First, we outer approximate $\Conv(S^b)$ and $\Conv(S^c)$ as shown below:
\begin{align*}
& \Conv_{OA}(S^2) = \left\{(f_1,\theta,H_1,\underline{f\theta}_1)\left|\begin{aligned}
& f_1 = H_1^\upbnd(\alpha_1-\theta), H_1 = H_1^\upbnd\\
& H_1^\upbnd \left(\alpha_1-\frac{F_1}{H_1^\upbnd}\right) (\alpha_1-\theta) \le \underline{f\theta}_1 \\
& \underline{f\theta}_1 \le \min\left\{ H_1^\upbnd\overline{\theta}^r (\alpha_1-\overline{\theta}^r)+H_1^\upbnd(\alpha_1-2\overline{\theta}^r)(\theta-\overline{\theta}^r)\right\}_{r=1}^R\\
& \left(\alpha_1-\frac{F_1}{H_1^\upbnd}\right) \le \theta\le  \alpha_1
\end{aligned}\right.\right\},\\
& \Conv_{OA}(S^3) = \left\{(f_1,\theta,H_1,\underline{f\theta}_1)\left|\begin{aligned}
& f_1 = F_1, \; \underline{f\theta}_1=F_1 \cdot \theta\\
& H_1 \ge \max \left\{\frac{F_1}{\alpha_1-\overline{\theta}^r} + \frac{F_1}{(\alpha_1-\overline{\theta}^r)^2} (\theta-\overline{\theta}^r) \right\}_{r=1}^R\\
& H_1 \le \frac{F_1}{\alpha_1-\theta^\lobnd} + \frac{H_1^\upbnd}{\alpha_1-\theta^\lobnd} (\theta-\theta^\lobnd)\\
& \theta^\lobnd \le \theta \le \left(\alpha_1-\frac{F_1}{H_1^\upbnd}\right)
\end{aligned}\right.\right\}.
\end{align*}
Next, we take the disjunctive union of $\Conv(S^a)$, $\Conv_{OA}(S^b)$ and $\Conv_{OA}(S^c)$ to obtain
\begin{subequations}
	\begin{align}
	& H_1 \ge H_1^\upbnd \lambda^b + \max \left\{\frac{F_1\lambda^c}{\alpha_1-\overline{\theta}^r} + \frac{F_1}{(\alpha_1-\overline{\theta}^r)^2} (\theta^3-\overline{\theta}^r\lambda^c) \right\}_{r=1}^R
	\label{eq:ecom:h1-1}\\
	& H_1 \le H_1^\upbnd \left(\frac{\theta^a-\theta^\lobnd\lambda^a}{\alpha_1-\theta^\lobnd}\right) + H_1^\upbnd \lambda^b + \frac{F_1\lambda^c}{\alpha_1-\theta^\lobnd} + \frac{H_1^\upbnd}{\alpha_1-\theta^\lobnd}(\theta^c-\theta^\lobnd \lambda^c)
	\label{eq:ecom:h1-2}\\
	& \underline{f\theta}_1 \ge H_1^\upbnd \left(\alpha_1-\frac{F_1}{H_1^\upbnd}\right)(\alpha_1\lambda^b-\theta^b)+ F_1\theta^c
	\label{eq:ecom:h1-3}\\
	& \underline{f\theta}_1 \le \min\left\{ H_1^\upbnd\overline{\theta}^r (\alpha_1-\overline{\theta}^r)\lambda^b+H_1^\upbnd(\alpha_1-2\overline{\theta}^r)(\theta^b-\overline{\theta}^r\lambda^b)\right\}_{r=1}^R+F_1\theta^c
	\label{eq:ecom:h1-4}\\
	& f_1=H_1^\upbnd(\alpha_1\lambda^b-\theta^b)+F_1\lambda^c
	\label{eq:ecom:h1-5}\\
	& \theta = \theta^a+\theta^b+\theta^c
	\label{eq:ecom:h1-6}\\
	& \theta^\lobnd \lambda^a \le \theta^a \le \alpha_1 \lambda^a
	\label{eq:ecom:h1-7}\\
	& \left(\alpha_1-\frac{F_1}{H_1^\upbnd}\right)\lambda^b \le \theta^b \le \alpha_1 \lambda^b
	\label{eq:ecom:h1-8}\\
	& \theta^\lobnd \lambda^c \le \theta^c \le \left(\alpha_1-\frac{F_1}{H_1^\upbnd}\right)\lambda^c
	\label{eq:ecom:h1-9}\\
	& \lambda^a+\lambda^b+\lambda^c = 1, \quad \lambda^a,\lambda^b,\lambda^c \ge 0
	\label{eq:ecom:h1-10}
	\end{align}
	\label{eq:ecom:h1}
\end{subequations}
In the following, we derive relaxed version of each inequality in terms of problem variables.
\begin{alignat*}{3}
\allowdisplaybreaks
&& H_1 &\ge H_1^\upbnd \lambda^b + \frac{F_1\lambda^c}{\alpha_1-\overline{\theta}^r} + \frac{F_1}{(\alpha_1-\overline{\theta}^r)^2} (\theta^c-\overline{\theta}^r\lambda^c)&\quad&\eqref{eq:ecom:h1-1}\\
&&&  =
\frac{(\alpha_1-2\overline{\theta}^r)f_1 + H_1^\upbnd\overline{\theta}^r (\alpha_1-\overline{\theta}^r)\lambda^b}{(\alpha_1-\overline{\theta}^r)^2}\\
&&&\mskip 30mu +\frac{H_1^\upbnd(\alpha_1-2\overline{\theta}^r)(\theta^2-\overline{\theta}^r\lambda^b)+F_1\theta^c}{(\alpha_1-\overline{\theta}^r)^2}&&\eqref{eq:ecom:h1-5}\\
&&& \ge \frac{(\alpha_1-2\overline{\theta}^r)f_1 + \underline{f\theta}_1}{(\alpha_1-\overline{\theta}^r)^2}&&\eqref{eq:ecom:h1-4}\\
&&&= \frac{f_1}{(\alpha_1-\overline{\theta}^r)}+\frac{1}{(\alpha_1-\overline{\theta}^r)^2} (\underline{f\theta}_1-\overline{\theta}^r f_1)&&\\
&&& = f_1 T_1(\overline{\theta}^r) + T_1'(\overline{\theta}^r)(\underline{f\theta}_1-\overline{\theta}^r f_1),\\
\\
&& H_1 &\le \frac{H_1^\upbnd(\theta^a-\theta^\lobnd \lambda^a)}{\alpha_1-\theta^\lobnd}\\
&&&\mskip 30mu+ \frac{H_1^\upbnd(\alpha_1\lambda^b-\theta^b+\theta^b-\theta^\lobnd\lambda^b)}{\alpha_1-\theta^\lobnd}\\
&&&\mskip 30mu+ \frac{F_1\lambda^c + H_1^\upbnd(\theta^c-\theta^\lobnd\lambda^c)}{\alpha_1-\theta^\lobnd}&&\eqref{eq:ecom:h1-2}\\
&&&= \frac{f_1}{\alpha_1-\theta^\lobnd} + H_1^\upbnd \left(\frac{\theta-\theta^\lobnd}{\alpha_1-\theta^\lobnd}\right)&&\eqref{eq:ecom:h1-5},\eqref{eq:ecom:h1-6},\eqref{eq:ecom:h1-10},\\
\\
&&\underline{f\theta}_1 &\ge  H_1^\upbnd \left(\alpha_1-\frac{F_1}{H_1^\upbnd}\right)(\alpha_1\lambda^b-\theta^b) + F_1\theta^c &&\eqref{eq:ecom:h1-3}\\
&&& = \mathrlap{\theta^\lobnd H_1^\upbnd(\alpha_1\lambda^b-\theta^b) + H_1^\upbnd \left(\alpha_1-\frac{F_1}{H_1^\upbnd}-\theta^\lobnd\right)(\alpha_1\lambda^b-\theta^b) + F_1\theta^c}\\
&&& = \theta^\lobnd f_1 + H_1^\upbnd \left(\alpha_1-\frac{F_1}{H_1^\upbnd}-\theta^\lobnd\right)(\alpha_1\lambda^b-\theta^b) \\
&&&\mskip 30mu + F_1 (\theta^c-\theta^\lobnd \lambda^c)&&\eqref{eq:ecom:h1-5},\eqref{eq:thetalo_assume}\\
&&&\ge \theta^\lobnd f_1&&\eqref{eq:ecom:h1-8},\eqref{eq:ecom:h1-9},\\
\\
&&\underline{f\theta}_1 &\ge   H_1^\upbnd \left(\alpha_1-\frac{F_1}{H_1^\upbnd}\right)(\alpha_1\lambda^b-\theta^b)+ F_1\theta^c && \eqref{eq:ecom:h1-3}\\
&&&= \alpha_1H_1^\upbnd (\alpha_1\lambda^b-\theta^b)-F_1(\alpha_1\lambda^b-\theta^b)+F_1\theta^c\\
&&&=\alpha_1(f_1-F_1\lambda^c)-F_1(\alpha_1\lambda^b-\theta^b) + F_1\theta^c&&\eqref{eq:ecom:h1-5}\\
&&& = \alpha_1f_1-\alpha_1F_1(1-\lambda^a)+F_1(\theta-\theta^a) &&\eqref{eq:ecom:h1-6},\eqref{eq:ecom:h1-10}\\
&&&\overset{}{\ge} \alpha_1f_1 +F_1 \theta -F_1\alpha_1&&\eqref{eq:ecom:h1-7},\\
\\
&&\underline{f\theta}_1 &\le  \Bigl\{H_1^\upbnd\overline{\theta}^r (\alpha_1-\overline{\theta}^r)\lambda^b\\
&&&\mskip 30mu +H_1^\upbnd(\alpha_1-2\overline{\theta}^r)(\theta^b-\overline{\theta}^r\lambda^b)\Bigr\}_{\overline{\theta}^r = \alpha_1} + F_1\theta^c&&\eqref{eq:ecom:h1-4}\\
&&& = \alpha_1H_1^\upbnd(\alpha_1\lambda^b-\theta^b) + F_1\theta^c\\
&&&= \alpha_1f_1 -\alpha_1F_1\lambda^c + F_1\theta^c&&\eqref{eq:ecom:h1-5}\\
&&&\le \alpha_1f_1&&\eqref{eq:ecom:h1-9},\\
\\
&&\underline{f\theta}_1 &\le
\Bigl\{
H_1^\upbnd\overline{\theta}^r (\alpha_1-\overline{\theta}^r)\lambda^b\\
&&&\mskip 30mu +H_1^\upbnd(\alpha_1-2\overline{\theta}^r)(\theta^b-\overline{\theta}^r\lambda^b)
\Bigr\}_{
	\overline{\theta}^r = \alpha-\frac{F_1}{H_1^\upbnd}
} + F_1\theta^c&&\eqref{eq:ecom:h1-4}\\
&&& =\mathrlap{F_1\left(\alpha_1-\frac{F_1}{H_1^\upbnd}\right) \lambda^b + \left( \alpha_1-2\frac{F_1}{H_1^\upbnd} \right)\left[H_1^\upbnd(\alpha_1\lambda^b-\theta^b)-F_1\lambda^b\right]+F_1\theta^c} \\
&&& =\mathrlap{\left( \alpha_1-\frac{F_1}{H_1^\upbnd} \right)\left[H_1^\upbnd(\alpha_1\lambda^b-\theta^b)-F_1\lambda^b\right]+F_1\theta^b+F_1\theta^c}\\
&&&\le \theta^\lobnd \left[H_1^\upbnd(\alpha_1\lambda^b-\theta^b)-F_1\lambda^b\right]+F_1\theta^b+F_1\theta^c&&
\eqref{eq:ecom:h1-8},\eqref{eq:thetalo_assume}
\\
&&& \le \theta^\lobnd\left[f_1-F_1\lambda^b-F_1\lambda^c\right] + F_1\theta^b+F_1\theta^c && \eqref{eq:ecom:h1-5}\\
&&&\le \theta^\lobnd\left[f_1-F_1\lambda^a-F_1\lambda^b-F_1\lambda^c\right]\\
&&&\mskip 30mu +F_1\theta^a+F_1\theta^b+F_1\theta^c&&
\eqref{eq:ecom:h1-7}
\\
&&&\le F_1\theta + \theta^\lobnd f_1  - F_1\theta^\lobnd&&
\eqref{eq:ecom:h1-6},\eqref{eq:ecom:h1-10},\\
\\
&&\theta &= \theta^a + \theta^b + \theta^c&&\eqref{eq:ecom:h1-6}\\
&&&\le \alpha_1\lambda^a + \theta^b + \left(\alpha_1 - \frac{F_1}{H_1^\upbnd}\right)\lambda^c
&&\eqref{eq:ecom:h1-7},\eqref{eq:ecom:h1-9}\\
&&&\le \alpha_1 - \frac{f_1}{H_1^\upbnd}
&&\eqref{eq:ecom:h1-5},\eqref{eq:ecom:h1-10}.\\
\end{alignat*}

\clearpage
\section{MIP Representations}
\label{sec:ecom-MIP-Representation}
Here, we present MIP representation of piecewise relaxations of sets $\mathcal{F}_1$ and $\mathcal{V}$. The piecewise relaxation of $\mathcal{F}_2$ can be expressed as a mixed-integer set in a similar manner. The derivation of these sets is provided in \textsection \ref{sec:ecom:mip-F1} and \textsection \ref{sec:ecom:mip-v}.
\MIPRepresentation{}

\section{Derivation of MIP Representation of Piecewise Relaxation of $\mathcal{F}_1$}
\label{sec:ecom:mip-F1}
Let the domain of Underwood root be partitioned as $\mathcal{I}=\{[\Theta^0,\Theta^1],\dots,[\Theta^{|\mathcal{I}|-1},\Theta^{\mathcal{|I|}}]\}$, such that $\alpha_2 = \Theta^0 \le \dots \le \Theta^{|\mathcal{I}|} = \alpha_1$. We express the piecewise relaxation of $\mathcal{F}_1$, given by $\bigcup_{t=1}^{|\mathcal{I}|-1}\Conv_{OA}(\mathcal{F}_{1,t}) \cup \mathcal{F}_{1,|\mathcal{I}|,\Relax}$, as the following disjunction:
\begin{align}
\bigvee_{t=1}^{|\mathcal{I}|-1} \left[ \begin{aligned}
& H_1^\rec \ge f_1^\rec T_1(\Theta^{t-1}) + T_1'(\Theta^{t-1}) (\underline{f\theta}_1^\rec - \Theta^{t-1} f_1^\rec)\\
& H_1^\rec \ge f_1^\rec T_1(\Theta^{t}) + T_1'(\Theta^{t}) (\underline{f\theta}_1^\rec - \Theta^{t} f_1^\rec)\\
& H_1^\strip \ge f_1^\strip T_1(\Theta^{t-1}) + T_1'(\Theta^{t-1}) (\underline{f\theta}_1^\strip - \Theta^{t-1} f_1^\strip)\\
& H_1^\strip \ge f_1^\strip T_1(\Theta^{t}) + T_1'(\Theta^{t}) (\underline{f\theta}_1^\strip - \Theta^{t} f_1^\strip)\\
& H_1^\rec \le f_1^\rec T_1(\Theta^{t-1}) + \left[\frac{T_1(\Theta^{t-1}) -T_1(\Theta^{t}) }{\Theta^{t-1}-\Theta^t}\right] (\underline{f\theta}_1^\rec - \Theta^{t-1} f_1^\rec)\\
& H_1^\strip \le f_1^\strip T_1(\Theta^{t-1}) + \left[\frac{T_1(\Theta^{t-1}) -T_1(\Theta^{t}) }{\Theta^{t-1}-\Theta^t}\right] (\underline{f\theta}_1^\strip - \Theta^{t-1} f_1^\strip)\\
& (F_1-f_1^\feed) \Theta^{t-1}\le  (F_1\theta-\underline{f\theta}_1^\feed) \le (F_1-f_1^\feed) \Theta^t\\
& f_1^\rec \Theta^{t-1} \le \underline{f\theta}_1^\rec \le f_1^\rec \Theta^t, \quad f_1^\strip \Theta^{t-1} \le \underline{f\theta}_1^\strip \le f_1^\strip \Theta^t\\
& H_1^\feed = H_1^\rec + H_1^\strip, \quad \underline{f\theta}_1^\feed =\underline{f\theta}_1^\rec + \underline{f\theta}_1^\strip, \quad f_1^\feed = f_1^\rec + f_1^\strip
\end{aligned}\right]\notag\\
\bigvee_{t=|\mathcal{I}|} \left[\begin{aligned}
& H_1^\rec \ge f_1^\rec T_1(\Theta^{t-1}) + T_1'(\Theta^{t-1}) (\underline{f\theta}_1^\rec - \Theta^{t-1} f_1^\rec)\\
& H_1^\strip \ge f_1^\strip T_1(\Theta^{t-1}) + T_1'(\Theta^{t-1}) (\underline{f\theta}_1^\strip - \Theta^{t-1} f_1^\strip)\\
& H_1^\feed \ge f_1^\feed T_1(\Theta^{t-1}) + T_1'(\Theta^{t-1}) (\underline{f\theta}_1^\feed - \Theta^{t-1} f_1^\feed)\\
& H_1^\rec \le \frac{f_1^\rec}{\alpha_1-\Theta^{t-1}} + (H_1^\rec)^\upbnd \left[\frac{\theta-\Theta^{t-1}}{\alpha_1-\Theta^{t-1}}\right] \\
& H_1^\strip \le \frac{f_1^\strip}{\alpha_1-\Theta^{t-1}} + (H_1^\strip)^\upbnd \left[\frac{\theta-\Theta^{t-1}}{\alpha_1-\Theta^{t-1}}\right]\\
& H_1^\feed \le \frac{f_1^\feed}{\alpha_1-\Theta^{t-1}} + (H_1^\feed)^\upbnd \left[\frac{\theta-\Theta^{t-1}}{\alpha_1-\Theta^{t-1}}\right]\\
& (F_1-f_1^\feed) \Theta^{t-1}\le  (F_1\theta-\underline{f\theta}_1^\feed) \le (F_1-f_1^\feed) \Theta^t\\
& f_1^\rec \Theta^{t-1} \le \underline{f\theta}_1^\rec \le f_1^\rec \Theta^t, \quad f_1^\strip \Theta^{t-1} \le \underline{f\theta}_1^\strip \le f_1^\strip \Theta^t\\
& H_1^\feed = H_1^\rec + H_1^\strip, \quad \underline{f\theta}_1^\feed =\underline{f\theta}_1^\rec + \underline{f\theta}_1^\strip, \quad f_1^\feed = f_1^\rec + f_1^\strip\;\quad\;\;
\end{aligned} \right].
\end{align}
In $\Conv_{OA}(\mathcal{F}_{1,t})$, we choose the extreme points of the partition, $\overline{\theta} = \Theta^{t-1}$ and $\overline{\theta}=\Theta^t$, for linearization; and in $\mathcal{F}_{1,|\mathcal{I}|,\Relax}$, we choose only $\overline{\theta}=\Theta^{t-1}$ since $T_1(\cdot)$ is not defined at $\overline{\theta}=\Theta^{|\mathcal{I}|}$. In order to derive an MIP representation that is reasonable in size, we make the following simplifications to the set $\mathcal{F}_{1,|\mathcal{I}|,\Relax}$. First, observe that the third inequality in $\mathcal{F}_{1,|\mathcal{I}|,\Relax}$ is implied from the first two inequalities and $H_1^\feed = H_1^\rec+H_1^\strip$, so we drop it from the set. Next, if $(H_1^\feed)^\upbnd > (H_1^\rec)^\upbnd + (H_1^\strip)^\upbnd$,
we reduce $(H_1^\feed)^\upbnd$ to $(H_1^\rec)^\upbnd + (H_1^\strip)^\upbnd$ because of fourth and fifth inequalities and $H_1^\feed = H_1^\rec+H_1^\strip$. Otherwise, we relax the sixth inequality by letting $(H_1^\feed)^\upbnd = (H_1^\rec)^\upbnd + (H_1^\strip)^\upbnd$. Then, the sixth inequality is implied from the fourth and fifth inequalities, so we drop it from the set. Next, using disjunctive programming techniques, we obtain
\begin{subequations}
	\begin{align}
	& H_{1,t}^\rec \ge \max\left\{f_{1,t}^\rec T_1(\Theta^{t-1}) + T_1'(\Theta^{t-1}) (\underline{f\theta}_{1,t}^\rec - \Theta^{t-1} f_{1,t}^\rec),\right. \notag\\
	& \hspace{2.5cm}\left. f_{1,t}^\rec T_1(\Theta^{t}) + T_1'(\Theta^{t}) (\underline{f\theta}_{1,t}^\rec - \Theta^{t} f_{1,t}^\rec)\right\}, & & \quad \llbracket t \rrbracket_{1}^{|\mathcal{I}|-1}\\
	& H_{1,t}^\rec \ge f_{1,t}^\rec T_1(\Theta^{t-1}) + T_1'(\Theta^{t-1}) (\underline{f\theta}_{1,t}^\rec - \Theta^{t-1} f_{1,t}^\rec),& & \quad t = |\mathcal{I}|\\
	& H_{1,t}^\strip \ge \max\left\{f_{1,t}^\strip T_1(\Theta^{t-1}) + T_1'(\Theta^{t-1}) (\underline{f\theta}_{1,t}^\strip - \Theta^{t-1} f_{1,t}^\strip),\right. \notag\\
	& \hspace{2.5cm}\left. f_{1,t}^\strip T_1(\Theta^{t}) + T_1'(\Theta^{t}) (\underline{f\theta}_{1,t}^\strip - \Theta^{t} f_{1,t}^\strip)\right\},& & \quad \llbracket t \rrbracket_{1}^{|\mathcal{I}|-1}\\
	& H_{1,t}^\strip \ge f_{1,t}^\strip T_1(\Theta^{t-1}) + T_1'(\Theta^{t-1}) (\underline{f\theta}_{1,t}^\strip - \Theta^{t-1} f_{1,t}^\strip), & & \quad t = |\mathcal{I}|\\
	& H_{1,t}^\rec \le  f_{1,t}^\rec T_1(\Theta^{t-1}) + \left[\frac{T_1(\Theta^{t}) -T_1(\Theta^{t-1}) }{\Theta^{t}-\Theta^{t-1}}\right] (\underline{f\theta}_{1,t}^\rec - \Theta^{t-1} f_{1,t}^\rec), & &\quad \llbracket t \rrbracket_{1}^{|\mathcal{I}|-1}\\
	& H_{1,t}^\rec \le \frac{f_{1,t}^\rec}{\alpha_1-\Theta^{t-1}} + (H_1^\rec)^\upbnd \left[\frac{\theta_{t}-\Theta^{t-1}\mu_t}{\Theta^{t}-\Theta^{t-1}}\right], & & \quad t = |\mathcal{I}|\\
	& H_{1,t}^\strip \le f_{1,t}^\strip T_1(\Theta^{t-1}) + \left[\frac{T_1(\Theta^{t}) -T_1(\Theta^{t-1}) }{\Theta^{t}-\Theta^{t-1}}\right] (\underline{f\theta}_{1,t}^\strip - \Theta^{t-1} f_{1,t}^\strip), & & \quad \llbracket t \rrbracket_{1}^{|\mathcal{I}|-1}\\
	& H_{1,t}^\strip \le \frac{f_{1,t}^\strip}{\alpha_1-\Theta^{t-1}} + (H_1^\strip)^\upbnd \left[\frac{\theta_{t}-\Theta^{t-1}\mu_t}{\Theta^{t}-\Theta^{t-1}}\right], & & \quad t = |\mathcal{I}|\\
	& (F_1\mu_t-f_{1,t}^\feed) \Theta^{t-1}\le  (F_1\theta_t-\underline{f\theta}_{1,t}^\feed) \le (F_1\mu_t-f_{1,t}^\feed) \Theta^t, & & \quad \llbracket t \rrbracket_{1}^{|\mathcal{I}|} \\
	& f_{1,t}^\rec \Theta^{t-1} \le \underline{f\theta}_{1,t}^\rec \le f_{1,t}^\rec \Theta^t, \quad f_{1,t}^\strip \Theta^{t-1} \le \underline{f\theta}_{1,t}^\strip \le f_{1,t}^\strip \Theta^t, & & \quad \llbracket t \rrbracket_{1}^{|\mathcal{I}|}\\
	& H_{1,t}^\feed = H_{1,t}^\rec + H_{1,t}^\strip, \quad \underline{f\theta}_{1,t}^\feed =\underline{f\theta}_{1,t}^\rec + \underline{f\theta}_{1,t}^\strip, \quad f_{1,t}^\feed = f_{1,t}^\rec + f_{1,t}^\strip, & & \quad \llbracket t \rrbracket_{1}^{|\mathcal{I}|}\\
	& H_1^\feed = \sum_{t=1}^{|\mathcal{I}|} H_{1,t}^\feed, \quad H_1^\rec = \sum_{t=1}^{|\mathcal{I}|} H_{1,t}^\rec, \quad H_1^\strip = \sum_{t=1}^{|\mathcal{I}|} H_{1,t}^\strip,\\
	& \underline{f\theta}_1^\feed = \sum_{t=1}^{|\mathcal{I}|} \underline{f\theta}_{1,t}^\feed, \quad \underline{f\theta}_1^\rec = \sum_{t=1}^{|\mathcal{I}|} \underline{f\theta}_{1,t}^\rec, \quad \underline{f\theta}_1^\strip = \sum_{t=1}^{|\mathcal{I}|} \underline{f\theta}_{1,t}^\strip, \\
	& f_1^\feed = \sum_{t=1}^{|\mathcal{I}|} f_{1,t}^\feed, \quad f_1^\rec = \sum_{t=1}^{|\mathcal{I}|} f_{1,t}^\rec, \quad f_1^\strip = \sum_{t=1}^{|\mathcal{I}|} f_{1,t}^\strip, \quad \theta = \sum_{t=1}^{|\mathcal{I}|} \theta_t,\\
	& \sum_{t=1}^{|\mathcal{I}|} \mu_t = 1, \quad  \mu_t \ge 0, \quad \llbracket t \rrbracket_{1}^{|\mathcal{I}|}
	\end{align}
\end{subequations}
Here, $\mu_t$ are the convex multipliers in disjunctive progamming, and variables with subscript $t$ are to be regarded as linearizations of products of the corresponding variables with $\mu_t$. For example, $\theta^\feed_{1,t}$ linearizes $\theta^\feed\mu_t$. To control the problem size, we project out $H_{1,t}^\feed$, $\underline{f\theta}_{1,t}^\feed$ and $f_{1,t}^\feed$ variables by substitution. Next, we eliminate $H_{1,t}^\rec$ and $H_{1,t}^\strip$ variables using Fourier-Motzkin. This leads to

\begin{subequations}\label{eq:proof-mip-F1}
	\begin{align}
	& H_1^\rec \ge\sum_{t=1}^{|\mathcal{I}|-1} \max\left\{f_{1,t}^\rec T_1(\Theta^{t-1}) + T_1'(\Theta^{t-1}) (\underline{f\theta}_{1,t}^\rec - \Theta^{t-1} f_{1,t}^\rec), \right. \notag\\
	& \hspace{1.5cm} \left. f_{1,t}^\rec T_1(\Theta^{t}) + T_1'(\Theta^{t}) (\underline{f\theta}_{1,t}^\rec - \Theta^{t} f_{1,t}^\rec)\right\} + f_{1,t}^\rec T_1(\Theta^{t-1}) + T_1'(\Theta^{t-1}) (\underline{f\theta}_{1,t}^\rec - \Theta^{t-1} f_{1,t}^\rec)
	\label{eq:proof-mip-F1-1}\\
	& H_1^\strip \ge \sum_{t=1}^{|\mathcal{I}|-1} \max\left\{f_{1,t}^\strip T_1(\Theta^{t-1}) + T_1'(\Theta^{t-1}) (\underline{f\theta}_{1,t}^\strip - \Theta^{t-1} f_{1,t}^\strip),\right. \notag\\
	& \hspace{1.5cm} \left. f_{1,t}^\strip T_1(\Theta^{t}) + T_1'(\Theta^{t}) (\underline{f\theta}_{1,t}^\strip - \Theta^{t} f_{1,t}^\strip)\right\} + f_{1,t}^\strip T_1(\Theta^{t-1}) + T_1'(\Theta^{t-1}) (\underline{f\theta}_{1,t}^\strip - \Theta^{t-1} f_{1,t}^\strip)
	\label{eq:proof-mip-F1-2}\\
	& \eqref{eq:mip-F1-5} - \eqref{eq:mip-F1-11}
	\end{align}
\end{subequations}
Now, we observe that each linear function in \eqref{eq:proof-mip-F1-1} and \eqref{eq:proof-mip-F1-2} is nonnegative. For example, consider $f_{1,t}^\rec T_1(\Theta^{t-1}) + T_1'(\Theta^{t-1}) (\underline{f\theta}_{1,t}^\rec - \Theta^{t-1} f_{1,t}^\rec)$ in \eqref{eq:proof-mip-F1-1}. Here, $f_{1,t}^\rec T_1(\Theta^{t-1}) \ge 0$, $T_1'(\Theta^{t-1}) \ge 0$, and $(\underline{f\theta}_{1,t}^\rec - \Theta^{t-1} f_{1,t}^\rec) \ge 0$ (see \eqref{eq:mip-F1-8}). We use this observation, and relax \eqref{eq:proof-mip-F1-1} and \eqref{eq:proof-mip-F1-2} to \eqref{eq:mip-F1-1}--\eqref{eq:mip-F1-4}. 
Finally, we require the solution to lie in a single partition by imposing integrality constraint on $\mu_t$ variables.

\section{Derivation of MIP representation of Piecewise Relaxation of $\mathcal{V}$}
\label{sec:ecom:mip-v}
For convenience, we replace \eqref{eq:conv-Vth-1} and \eqref{eq:conv-Vth-2} in $\Conv(\mathcal{V})$ with $U^\rec-U^\strip= \Upsilon^\rec-\Upsilon^\strip$ and $\underline{U\theta}^\rec - \underline{U\theta}^\strip = \underline{\Upsilon\theta}^\rec - \underline{\Upsilon\theta}^\strip$. Note that this still captures $\Conv(\mathcal{V})$, since the former can be derived by a linear combination of the latter. Next, we use disjunctive programming to construct the convex hull of piecewise relaxation of $\mathcal{V} = \bigcup_{t=1}^{|\mathcal{I}|} \Conv(\mathcal{V}_t) $
\begin{subequations}\label{eq:vmip-proof1}
	\begin{align}
	& U^\rec_t-U^\strip_t= \Upsilon^\rec_t-\Upsilon^\strip_t, & & \llbracket t \rrbracket_{1}^{|\mathcal{I}|}
	\label{eq:vmip-proof1-1}\\
	&\underline{U\theta}^\rec_t -\underline{U\theta}^\strip_t = \underline{\Upsilon\theta}^\rec_t - \underline{\Upsilon\theta}^\strip_t, & & \llbracket t \rrbracket_{1}^{|\mathcal{I}|}
	\label{eq:vmip-proof1-2}\\
	& 0 \le \underline{(\cdot)\theta}_t - \Theta^{t-1} (\cdot)_t \le (\cdot)^\upbnd (\theta_t-\Theta^{t-1} \mu_t), \quad \forall \; (\cdot) \in \{U^\rec,U^\strip,\Upsilon^\rec,\Upsilon^\strip\}, & & \llbracket t \rrbracket_{1}^{|\mathcal{I}|}
	\label{eq:vmip-proof1-3}\\
	& 0 \le \Theta^t (\cdot)_t - \underline{(\cdot)\theta}_t \le (\cdot)^\upbnd (\theta^\upbnd \mu_t - \theta_t), \quad \forall \; (\cdot) \in \{U^\rec,U^\strip,\Upsilon^\rec,\Upsilon^\strip\}, & & \llbracket t \rrbracket_{1}^{|\mathcal{I}|}\\
	& \underline{(\cdot)\theta} = \sum_{t=1}^{|\mathcal{I}|} \underline{(\cdot)\theta}_t, \quad (\cdot) = \sum_{t=1}^{|\mathcal{I}|}(\cdot)_t, \quad \forall \; (\cdot) \in \{U^\rec,U^\strip,\Upsilon^\rec,\Upsilon^\strip\}
	\label{eq:vmip-proof-5}\\
	&\sum_{t=1}^{|\mathcal{I}|}\theta_t = \theta,\quad \sum_{t=1}^{|\mathcal{I}|} \mu_t = 1, \quad \mu_t \ge 0, \; \llbracket t \rrbracket_{1}^{|\mathcal{I}|}
	\end{align}
\end{subequations}
Here, $\mu_t$ are disjunctive programming variables, and variables $\underline{U\theta}^\rec_t$, $U^\rec_t$ are to be regarded as the linearizations of $\underline{U\theta}^\rec \cdot \mu_t$, $U^\rec \cdot \mu_t$, respectively. To the above, we append the redundant constraint $\underline{U\theta}^\rec -\underline{U\theta}^\strip = \underline{\Upsilon\theta}^\rec - \underline{\Upsilon\theta}^\strip$, which is derived by adding all the equations in \eqref{eq:vmip-proof1-1}, and using \eqref{eq:vmip-proof-5}. Then, we relax \eqref{eq:vmip-proof1} by discarding all the equations in \eqref{eq:vmip-proof1-2}. Next, we eliminate variables of the form $\underline{U\theta}_t$ and $\theta_t$ in the following manner. For notational convenience, we present the elimination process assuming we have three partitions. Consider
\begin{subequations}
	\begin{align}
	& 0 \le \underline{U\theta}_t -\Theta^{t-1} U_t \le U^\upbnd (\theta_t-\Theta^{t-1} \mu_t), \quad t=1,2,3\\
	& 0 \le \Theta^t U_t - \underline{U\theta}_t \le U^\upbnd (\Theta^t\mu_t - \theta_t) , \quad t = 1,2,3\\
	& \underline{U\theta} = \underline{U\theta}_1 + \underline{U\theta}_2 + \underline{U\theta}_3, \quad \theta = \theta_1+\theta_2+\theta_3
	\end{align}
\end{subequations}
First, we substitute out $\underline{U\theta}_1$ by  $\underline{U\theta} - \underline{U\theta}_2 - \underline{U\theta}_3$. Then, we rearrange the inequalities governing $\underline{U\theta}_2$ in the following manner:
\begin{align}
\left.\begin{alignedat}{2}
-(\Theta^0 U_1 - \underline{U\theta} + \underline{U\theta}_3)-U^\upbnd(\theta_1-\Theta^0\mu_1) &\le \underline{U\theta}_2 &&\le  -(\Theta^0 U_1 - \underline{U\theta} + \underline{U\theta}_3)\\
-(\Theta^1U_1-\underline{U\theta}+\underline{U\theta}_3) &\le \underline{U\theta}_2 &&\le  U^\upbnd(\Theta^1\mu_1-\theta_1) -(\Theta^1U_1-\underline{U\theta}+\underline{U\theta}_3)\\
\Theta^1 U_2 &\le \underline{U\theta}_2 &&\le U^\upbnd (\theta_2 - \Theta^1 \mu_2) + \Theta^1U_2\\
\Theta^2U_2-U^\upbnd(\Theta^2\mu_2-\theta_2) & \le \underline{U\theta}_2 &&\le \Theta^2 U_2
\end{alignedat}\right\}
\label{eq:proof-mip-v}
\end{align}
Now, we eliminate $\underline{U\theta}_2$ using Fourier-Motzkin. We write (L1R3) to denote first inequality from the left hand side, and third inequality from the right hand side.
\begin{alignat*}{2}
\textnormal{(L1R1) and (L2R2):} &\quad&&\Theta^0 \mu_1 \le \theta_1  \le \Theta^1\mu_1\\
\textnormal{(L2R1) and (L1R2):} &&&0 \le U_1 \le \mu_1 U^\upbnd\\
\textnormal{(L3R3) and (L4R4):} &&&\Theta^1 \mu_2 \le \theta_2  \le \Theta^2\mu_2\\
\textnormal{(L3R4) and (L4R3):} &&& 0 \le U_2 \le \mu_2 U^\upbnd\\
\textnormal{(L1R3) and (L3R1):} &&& -(\Theta^0 U_1 +\Theta^1 U_2-\underline{U\theta})-U^\upbnd(\theta_1+\theta_2-\Theta^0\mu_1-\Theta^1\mu_2)  \\
&&&\mskip 30mu \le \underline{U\theta}_3 \le  -(\Theta^0U_1+\Theta^1U_2 -\underline{U\theta}) \\
\textnormal{(L2R4) and (L4R2):} &&& -(\Theta^1U_1+\Theta^2U_2-\underline{U\theta}) \\
&&&\mskip 30mu \le \underline{U\theta}_3 \le U^\upbnd (\Theta^1\mu_1 + \Theta^2\mu_2 -\theta_1-\theta_2) - (\Theta^1U_1-\Theta^2U_2 - \underline{U\theta})\\
\textnormal{(L1R4) and (L4R1):} &&& -(\Theta^0U_1-\underline{U\theta}) -\Theta^2U_2 - U^\upbnd(\theta_1-\Theta^0 \mu_1)\\
&&&\mskip 30mu  \le \underline{U\theta}_3 \le -(\Theta^0U_1 - \underline{U\theta}) - \Theta^2U_2 + U^\upbnd(\Theta^2\mu_2-\theta_2)  \\
\textnormal{(L2R3) and (L3R2):} &&& -(\Theta^1U_1 - \underline{U\theta}) - \Theta^1 U_2-U^\upbnd(\theta_2-\Theta^1\mu_2)\\
&&&\mskip 30mu \le \underline{U\theta}_3 \le U^\upbnd(\Theta^1\mu_1-\theta_1)-\Theta^1U_2 -(\Theta^1U_1-\underline{U\theta})
\end{alignat*}
We relax the set by discarding inequalities obtained from (L1R4), (L4R1), (L2R3) and (L3R2). The inequalities obtained from (L1R3), (L3R1), (L2R4) and (L4R2) have the same form as the four inequalities in \eqref{eq:proof-mip-v}. As before, we eliminate $\underline{U\theta}_3$ using Fourier-Motzkin, and discard inequalities obtained from (L1R4), (L4R1), (L2R3) and (L3R2). This leads to
\begin{subequations}
	\begin{align}
	& 0\le \underline{U\theta}-\sum_{t=1}^3 \Theta^{t-1}U_t\le U^\upbnd\left(\theta - \sum_{t=1}^3 \Theta^{t-1}\mu_t\right)
	\label{eq:vmip-proof2-1}\\
	& 0 \le \sum_{t=1}^3 \Theta^{t}U_t - \underline{U\theta} \le U^\upbnd\left( \sum_{t=1}^3 \Theta^{t}\mu_t - \theta\right)
	\label{eq:vmip-proof2-2}\\
	& \theta = \sum_{t=1}^3 \theta_t, \quad \Theta^{t-1}\mu_t \le \theta_t \le \Theta^t \mu_t, \quad 0 \le U_t \le U^\upbnd \mu_t, \quad t = 1,2,3
	\label{eq:vmip-proof2-3}
	\end{align}
\end{subequations}
In this manner, we eliminate all variables of the form $\underline{(\cdot)\theta}_t$ from \eqref{eq:vmip-proof1}. Then, we eliminate all $\theta_t$ variables, which are now constrained only by \eqref{eq:vmip-proof2-3}, using Fourier-Motzkin. This leads to $\sum_{t=1}^{|\mathcal{I}|} \Theta^{t-1}\mu^t \le \theta \le \sum_{t=1}^{|\mathcal{I}|} \Theta^{t}\mu^t$. Since it is implied from \eqref{eq:vmip-proof2-1} and \eqref{eq:vmip-proof2-2}, we do not impose it explicitly. Finally, we require the solution to lie in a single partition by imposing integrality constraint on $\mu_t$ variables.

\section{Test Set}
\label{sec:ecom:test-set}
The test set for computational experiments is borrowed from \citet{giridhar2010a}. The current state-of-the-art methods can handle design problem involving four components. However, they are often unable to scale to five components, which are practically relevant and remains challenging. In this study, we focus on five component separations, i.e., $N=5$.

The parameter settings are generated in the following manner. For every $a\in \{1,\dots, 2^N-1\}$, we first construct $N-$digit binary representation of $a$, denoted as $\bin(a)$. Let $\bin(a)(p)$ deonte the $p$\textsuperscript{th} digit of $\bin(a)$. We define two sets: $\mathcal{D}_0 = \{p:\;\bin(a)(p)=0\}$ and $\mathcal{D}_1 = \{p:\;\bin(a)(p)=1\}$. $\bin(a)(p)=0$ indicates that component $p$ is lean in the mixture, and its composition is set to 5\%. On the other hand, $\bin(a)(p) = 1$ indicates that component $p$ is abundant in the mixture. We consider the case, where all abundant components are present in equal proportions. Therefore, for a given $a$, the feed composition $\{F_p^a\}_{p=1}^N$ is obtained as
\begin{gather}
F_p^a = \begin{cases}
5 & \text{ if } p\in \mathcal{D}_0\\
\frac{100-5\times|\mathcal{D}_0|}{|\mathcal{D}_1|} & \text{ if } p \in \mathcal{D}_1
\end{cases}  \quad \forall \quad p \in \{1,\dots,N\}
\end{gather}
In a similar manner, for every $b\in \{0,\dots,2^{N-1}-1\}$, we first construct $(N-1)-$digit binary representation of $b$. Here, $\bin(b)(p) = 0$ (resp. $\bin(b)(p) = 1$) indicates that the separation between component $p$ and $p+1$ is easy (resp. difficult). We take relative volatility value of 2.5 and 1.1 for an easy and difficult separation, respectively. For a given $b$, expressing all relative volatilities w.r.t to the heaviest component, we have $\alpha_N^b = 1$ and
\begin{gather}
\alpha_p^b = \prod_{q=p}^{N-1} \left[2.5\; (1-\bin(b)(q)) + 1.1\; \bin(b)(q)\right] \quad \forall \quad p \in \{1,\dots,N-1\}
\end{gather}
The parameter settings for $\texttt{Case(a,b)}$ are then given by $N=5$, $\{F_p^a\}_{p=1}^N$, $\{\alpha_p^b\}_{p=1}^N$, $\Phi_{1,N}=\Phi_{1,1}=\dots=\Phi_{N,N}=1$. Since $a\in \{1,\dots,2^N-1\}$ and $b\in\{0,\dots,2^{N-1}-1\}$, total number of cases in the test set is $ (2^5-1)\times 2^4 = 496$.

%
%
%
%

\end{document}